\documentclass[a4paper, reqno]{amsart}

\usepackage[T1]{fontenc}
\usepackage[latin1]{inputenc}
\usepackage{newlfont, verbatim, indentfirst, enumerate}

\usepackage{amsmath, amsthm, amssymb, amscd}
\usepackage{latexsym, amsfonts, amsbsy, mathrsfs}
\usepackage{array, pst-all, graphicx, rotating}
\usepackage{tikz} \usetikzlibrary{arrows}
  \usetikzlibrary{patterns} \usetikzlibrary{decorations.markings}
\usepackage[unicode]{hyperref}

\allowdisplaybreaks

\numberwithin{equation}{section}

\theoremstyle{plain}
\newtheorem{theorem}{Theorem}[section]
\newtheorem{theo}[theorem]{Theorem}
\newtheorem{prop}[theorem]{Proposition}
\newtheorem{lem}[theorem]{Lemma}   
\newtheorem{cor}[theorem]{Corollary}

\theoremstyle{definition}

\newtheorem{rem}[theorem]{Remark}

\renewenvironment{itemize}{\begin{list}{$\bullet$}{\leftmargin=0.5cm}\parindent=0pt}{\end{list}}

\newcommand {\thickarrowpersonalized} [1]  
  {\draw[rounded corners,decoration={markings, mark=at position 1 with
  {\arrow[scale=2]{>}}}, postaction={decorate}] #1;}

\newcommand {\functor} [1] {\mathcal{#1}}  
\newcommand {\emphatic} [1] {\emph{#1}}

\newcommand {\thetaa} [3] {\theta_{#1,#2,#3}}  
\newcommand {\thetab} [3] {\theta_{#1,#2,#3}^{-1}}  

\def \CATA {\mathbf{\mathscr{A}}}  
\def \CATB {\mathbf{\mathscr{B}}}
\def \CATC {\mathbf{\mathscr{C}}}
\def \CATD {\mathbf{\mathscr{D}}}

\def \id {\operatorname{id}}  

\def \SETW {\mathbf{W}} 
\def \SETWsat {\mathbf{W}_{\mathbf{\operatorname{sat}}}}  
\def \SETWinv {\mathbf{W}^{-1}}  
\def \SETWsatinv {\mathbf{W}_{\mathbf{\operatorname{sat}}}^{-1}}  

\def \SETWA {\mathbf{W}_{\mathbf{\mathscr{A}}}}  
\def \SETWAsat {\mathbf{W}_{\mathbf{\mathscr{A}},\mathbf{\operatorname{sat}}}}
\def \SETWAinv {\mathbf{W}^{-1}_{\mathbf{\mathscr{A}}}}

\def \SETWB {\mathbf{W}_{\mathbf{\mathscr{B}}}}  
\def \SETWBsat {\mathbf{W}_{\mathbf{\mathscr{B}},\mathbf{\operatorname{sat}}}}
\def \SETWBsatsat {\mathbf{W}_{\mathbf{\mathscr{B}},\mathbf{\operatorname{sat}},\mathbf{\operatorname{sat}}}}
\def \SETWBinv {\mathbf{W}^{-1}_{\mathbf{\mathscr{B}}}}
\def \SETWBsatinv {\mathbf{W}^{-1}_{\mathbf{\mathscr{B}},\mathbf{\operatorname{sat}}}}
\def \SETWBmin {\mathbf{W}_{\mathbf{\mathscr{B}},\mathbf{\operatorname{min}}}}
\def \SETWBmininv {\mathbf{W}^{-1}_{\mathbf{\mathscr{B}},\mathbf{\operatorname{min}}}}
\def \SETWBequiv {\mathbf{W}_{\mathbf{\mathscr{B}},\mathbf{\operatorname{equiv}}}}

\def \RedAtl {(\mathbf{\mathcal{R}ed\,\mathcal{A}tl})}  

\hypersetup{unicode=true, pdftoolbar=true, pdfmenubar=true, pdffitwindow=false,
  pdfstartview={FitH}, pdftitle={Some insights on bicategories of fractions - III},
  pdfauthor={Matteo Tommasini}, pdfkeywords={bicategories of fractions, equivalences of bicategories},
  pdfnewwindow=true,colorlinks=false, linkcolor=red, citecolor=red, filecolor=magenta, urlcolor=cyan}

\numberwithin{equation}{section}

\begin{document}

\title[Some insights on bicategories of fractions - III]
{Some insights on bicategories of fractions - III \\ - \\ \mdseries{\small{Equivalences of bicategories of fractions}}}

\author{Matteo Tommasini}

\address{\flushright Mathematics Research Unit\newline University of Luxembourg\newline
6, rue Richard Coudenhove-Kalergi\newline L-1359 Luxembourg\newline\newline
website: \href{http://matteotommasini.altervista.org/}
{\nolinkurl{http://matteotommasini.altervista.org/}}\newline\newline
email: \href{mailto:matteo.tommasini2@gmail.com}{\nolinkurl{matteo.tommasini2@gmail.com}}}

\date{\today}

\thanks{I would like to thank Dorette Pronk for several interesting discussions about her work on
bicategories of fraction and for some useful suggestions about this series of papers. This research
was performed at the Mathematics Research Unit of the University of Luxembourg, thanks to the grant
4773242 by Fonds National de la Recherche Luxembourg.}

\begin{abstract}
We fix any bicategory $\CATA$ together with a class of morphisms $\SETWA$, such that there is a
bicategory of fractions $\CATA\left[\SETWAinv\right]$ (as described by D.~Pronk). Given another such
pair $(\CATB,\SETWB)$ and any pseudofunctor $\functor{F}:\CATA\rightarrow\CATB$, we
find necessary and sufficient conditions in order to have an induced equivalence of bicategories
from $\CATA\left[\SETWAinv\right]$ to $\CATB\left[\SETWBinv\right]$. In
particular, this gives necessary and sufficient conditions in order to have an equivalence from
any bicategory of fractions $\CATA\left[\SETWAinv\right]$ to any given bicategory $\CATB$. 
\end{abstract}

\maketitle

\begingroup{\hypersetup{linkbordercolor=white}
\tableofcontents}\endgroup

\section*{Introduction}
In 1996 Dorette Pronk introduced the notion of (\emph{right}) \emph{bicalculus of fractions}
(see~\cite{Pr}), generalizing the concept of (right) calculus of fractions (described in 1967 by
Pierre Gabriel and Michel Zisman, see~\cite{GZ}) from the framework of categories to that of
bicategories. Pronk proved that given a bicategory $\CATC$ together with a class of morphisms $\SETW$
(satisfying a set of technical conditions called (BF)),
there is a bicategory $\CATC\left[\SETWinv\right]$ (called (\emph{right}) \emph{bicategory of
fractions}) and a pseudofunctor $\functor{U}_{\SETW}:\CATC \rightarrow\CATC\left[\SETWinv\right]$.
Such a pseudofunctor sends each element of $\SETW$ to an internal equivalence and is universal with
respect to such property (see~\cite[Theorem~21]{Pr}). In particular, the bicategory $\CATC\left[
\SETWinv\right]$ is unique up to equivalences of bicategories.\\

In~\cite[Definition~2.1]{T4} we introduced the notion of right saturated as follows: given
any pair $(\CATC,\SETW)$ as above, we define the (right) saturation $\SETWsat$ of $\SETW$ as the
class of all morphisms $f:B\rightarrow A$ in $\CATC$, such that there are a pair of objects $C,D$
and a pair of morphisms $g:C\rightarrow B$, $h:D\rightarrow C$, such that both $f\circ g$ and
$g\circ h$ belong to $\SETW$. Then we were able to prove that also $(\CATC,\SETWsat)$ admits a
right bicategory of fractions $\CATC\left[\SETWsatinv\right]$, that is equivalent
to $\CATC\left[\SETWinv\right]$. This allowed us to prove the following result (for the notations
used below for bicategories of fractions, we refer directly
to~\cite[\S~2.2 and~2.3]{Pr}).

\begin{theo}\label{theo-01}
\cite[Theorems~0.2 and 0.3(A) and Corollary~0.4]{T4}
Let us fix any $2$ pairs $(\CATA,\SETWA)$ and $(\CATB,\SETWB)$, both satisfying conditions
\emphatic{(BF)}, and any pseudofunctor $\functor{F}:\CATA\rightarrow
\CATB$. Then the following facts are equivalent:

\begin{enumerate}[\emphatic{(}i\emphatic{)}]
 \item $\functor{F}_1(\SETWA)\subseteq\SETWBsat$;
 \item $\functor{F}_1(\SETWAsat)\subseteq\SETWBsat$;
 \item there are a pseudofunctor $\functor{G}$ and a pseudonatural equivalence $\kappa$ as below
 
  \[
  \begin{tikzpicture}[xscale=2.2,yscale=-0.9]
    \node (A0_0) at (0, 0) {$\CATA$};
    \node (A0_2) at (2, 0) {$\CATB$};
    \node (A2_0) at (0, 2) {$\CATA\Big[\SETWAinv\Big]$};
    \node (A2_2) at (2, 2) {$\CATB\Big[\SETWBinv\Big]$;};
    
    \node (A1_1) [rotate=225] at (0.9, 1) {$\Longrightarrow$};
    \node (B1_1) at (1.2, 1) {$\kappa$};
    
    \path (A0_0) edge [->]node [auto] {$\scriptstyle{\functor{F}}$} (A0_2);
    \path (A0_0) edge [->]node [auto,swap] {$\scriptstyle{\functor{U}_{\SETWA}}$} (A2_0);
    \path (A0_2) edge [->]node [auto] {$\scriptstyle{\functor{U}_{\SETWB}}$} (A2_2);
    \path (A2_0) edge [->, dashed]node [auto,swap] {$\scriptstyle{\functor{G}}$} (A2_2);
  \end{tikzpicture}
  \]
 \item there is a pair $(\functor{G},\kappa)$ as in \emphatic{(}iii\emphatic{)},
  such that the pseudofunctor $\mu_{\kappa}:\CATA\rightarrow\operatorname{Cyl}\left(\CATB\left[
  \SETWBinv\right]\right)$ associated to $\kappa$ sends each element of $\SETWA$ to an internal
  equivalence \emphatic{(}here $\operatorname{Cyl}(\CATC)$ is the bicategory of cylinders associated
  to any given bicategory $\CATC$, see~\cite[pag.~60]{B}\emphatic{)}.
\end{enumerate}

Moreover, if any of the conditions above is satisfied, then there are a pseudofunctor

\[\widetilde{\functor{G}}:\CATA\Big[\SETWAinv\Big]\longrightarrow\CATB\Big[\SETWBsatinv\Big]\]
and a pseudonatural equivalence $\widetilde{\kappa}:\functor{U}_{\SETWBsat}\circ\functor{F}
\Rightarrow\widetilde{\functor{G}}\circ\functor{U}_{\SETWA}$, such that:

\begin{itemize}
 \item the pseudofunctor $\mu_{\widetilde{\kappa}}:\CATA\rightarrow\operatorname{Cyl}
  \left(\CATB\left[\SETWBsatinv\right]\right)$ associated to $\widetilde{\kappa}$ sends each morphism
  of $\SETWA$ to an internal equivalence;
 
 \item for each object $A_{\CATA}$, we have $\widetilde{\functor{G}}_0(A_{\CATA})=
  \functor{F}_0(A_{\CATA})$;

 \item for each morphism $(A'_{\CATA},\operatorname{w}_{\CATA},f_{\CATA}):A_{\CATA}\rightarrow
  B_{\CATA}$ in $\CATA\left[\SETWAinv\right]$, we have 
  
  \[\widetilde{\functor{G}}_1\Big(A'_{\CATA},\operatorname{w}_{\CATA},f_{\CATA}\Big)=\Big(
  \functor{F}_0(A'_{\CATA}),\functor{F}_1(\operatorname{w}_{\CATA}),\functor{F}_1(f_{\CATA})
  \Big);\]
  
 \item for each $2$-morphism

  \[\Big[A^3_{\CATA},\operatorname{v}^1_{\CATA},\operatorname{v}^2_{\CATA},\alpha_{\CATA},
  \beta_{\CATA}\Big]:\Big(A^1_{\CATA},\operatorname{w}^1_{\CATA},f^1_{\CATA}\Big)\Longrightarrow
  \Big(A^2_{\CATA},\operatorname{w}^2_{\CATA},f^2_{\CATA}\Big)\]
  in $\CATA\left[\SETWAinv\right]$, we have
  
  \begin{gather}
  \label{eq-07} \widetilde{\functor{G}}_2\Big(\Big[A^3_{\CATA},\operatorname{v}^1_{\CATA},
  \operatorname{v}^2_{\CATA},\alpha_{\CATA},\beta_{\CATA}\Big]\Big)=
  \Big[\functor{F}_0(A^3_{\CATA}),\functor{F}_1(\operatorname{v}^1_{\CATA}),\functor{F}_1
  (\operatorname{v}^2_{\CATA}),\\
  \nonumber \psi_{\operatorname{w}^2_{\CATA},
  \operatorname{v}^2_{\CATA}}^{\functor{F}}\odot\functor{F}_2(\alpha_{\CATA})\odot
  \Big(\psi_{\operatorname{w}^1_{\CATA},\operatorname{v}^1_{\CATA}}^{\functor{F}}\Big)^{-1},
  \psi_{f^2_{\CATA},\operatorname{v}^2_{\CATA}}^{\functor{F}}\odot\functor{F}_2(\beta_{\CATA})
  \odot\Big(\psi_{f^1_{\CATA},\operatorname{v}^1_{\CATA}}^{\functor{F}}\Big)^{-1}\Big]
  \end{gather}
  \emphatic{(}where the $2$-morphisms $\psi_{\bullet}^{\functor{F}}$ are the associators of
  $\functor{F}$\emphatic{)}.
\end{itemize}  

In addition, given any pair $(\functor{G},\kappa)$ as in
\emphatic{(}iv\emphatic{)}, the following facts are equivalent:

\begin{enumerate}[\emphatic{(}a\emphatic{)}]
 \item $\functor{G}:\CATA[\,\SETWAinv]\rightarrow\CATB[\,\SETWBinv]$ is an
  equivalence of bicategories;
 \item $\widetilde{\functor{G}}:\CATA[\,\SETWAinv]\rightarrow\CATB[\,\SETWBsatinv]$
  is an equivalence of bicategories;
\end{enumerate}
\end{theo} 

Using the technical lemmas already proved in~\cite{T3} and~\cite{T4}, in the present paper we will
identify necessary and sufficient conditions such that (b) holds. Combining this with the previous
theorem, we will then prove the following result.

\begin{theo}\label{theo-02}
Let us fix any $2$ pairs $(\CATA,\SETWA)$ and $(\CATB,\SETWB)$, both satisfying conditions
\emphatic{(BF)},
and any pseudofunctor $\functor{F}:\CATA\rightarrow\CATB$ such that $\functor{F}_1(\SETWA)\subseteq
\SETWBsat$. Moreover, let us consider any pair $(\functor{G},\kappa)$ as in
\emphatic{Theorem~\ref{theo-01}(iv)}.
Then $\functor{G}$ is an equivalence of bicategories if and only if $\functor{F}$ satisfies the
following $5$ conditions.

\begin{enumerate}[\emphatic{(}{A}1\emphatic{)}]\label{A}
 \item\label{A1} For any object $A_{\CATB}$, there are a pair of objects $A_{\CATA}$ and
  $A'_{\CATB}$ and a pair of morphisms $\operatorname{w}^1_{\CATB}$ in $\SETWB$ and
  $\operatorname{w}^2_{\CATB}$ in $\SETWBsat$, as follows:
  
  \begin{equation}\label{eq-31}
  \begin{tikzpicture}[xscale=2.4,yscale=-1.2]
    \node (A0_0) at (-0.2, 0) {$\functor{F}_0(A_{\CATA})$};
    \node (A0_1) at (1, 0) {$A'_{\CATB}$};
    \node (A0_2) at (2, 0) {$A_{\CATB}$.};
    \path (A0_1) edge [->]node [auto,swap] {$\scriptstyle{\operatorname{w}^1_{\CATB}}$} (A0_0);
    \path (A0_1) edge [->]node [auto] {$\scriptstyle{\operatorname{w}^2_{\CATB}}$} (A0_2);
  \end{tikzpicture}
  \end{equation}

 \item\label{A2} Let us fix any triple of objects $A^1_{\CATA},A^2_{\CATA},A_{\CATB}$ and any pair of
  morphisms $\operatorname{w}^1_{\CATB}$ in $\SETWB$ and $\operatorname{w}^2_{\CATB}$ in
  $\SETWBsat$ as follows
  
  \begin{equation}\label{eq-08}
  \begin{tikzpicture}[xscale=2.4,yscale=-1.2]
    \node (A0_0) at (0, 0) {$\functor{F}_0(A^1_{\CATA})$};
    \node (A0_1) at (1, 0) {$A_{\CATB}$};
    \node (A0_2) at (2, 0) {$\functor{F}_0(A^2_{\CATA})$.};
    \path (A0_1) edge [->]node [auto,swap] {$\scriptstyle{\operatorname{w}^1_{\CATB}}$} (A0_0);
    \path (A0_1) edge [->]node [auto] {$\scriptstyle{\operatorname{w}^2_{\CATB}}$} (A0_2);
  \end{tikzpicture}
  \end{equation}
  Then there are an object $A^3_{\CATA}$, a pair of morphisms $\operatorname{w}^1_{\CATA}$ in
  $\SETWA$ and $\operatorname{w}^2_{\CATA}$ in $\SETWAsat$ as follows
  
  \[
  \begin{tikzpicture}[xscale=2.4,yscale=-1.2]
    \node (A0_0) at (0, 0) {$A^1_{\CATA}$};
    \node (A0_1) at (1, 0) {$A^3_{\CATA}$};
    \node (A0_2) at (2, 0) {$A^2_{\CATA}$};
    
    \path (A0_1) edge [->]node [auto,swap] {$\scriptstyle{\operatorname{w}^1_{\CATA}}$} (A0_0);
    \path (A0_1) edge [->]node [auto] {$\scriptstyle{\operatorname{w}^2_{\CATA}}$} (A0_2);
  \end{tikzpicture}
  \]
  and a set of data $(A'_{\CATB},\operatorname{z}^1_{\CATB},\operatorname{z}^2_{\CATB},
  \gamma^1_{\CATB},\gamma^2_{\CATB})$ as follows
  
  \[
  \begin{tikzpicture}[xscale=2.2,yscale=-0.8]
    \node (A0_2) at (2, 0) {$A_{\CATB}$};
    \node (A2_2) at (2, 2) {$A'_{\CATB}$};
    \node (A2_0) at (0, 2) {$\functor{F}_0(A^1_{\CATA})$};
    \node (A2_4) at (4, 2) {$\functor{F}_0(A^2_{\CATA})$,};
    \node (A4_2) at (2, 4) {$\functor{F}_0(A^3_{\CATA})$};
    
    \node (A2_3) at (2.8, 2) {$\Downarrow\,\gamma^2_{\CATB}$};
    \node (A2_1) at (1.2, 2) {$\Downarrow\,\gamma^1_{\CATB}$};
    
    \path (A4_2) edge [->]node [auto,swap]
      {$\scriptstyle{\functor{F}_1(\operatorname{w}^2_{\CATA})}$} (A2_4);
    \path (A0_2) edge [->]node [auto] {$\scriptstyle{\operatorname{w}^2_{\CATB}}$} (A2_4);
    \path (A2_2) edge [->]node [auto,swap] {$\scriptstyle{\operatorname{z}^1_{\CATB}}$} (A0_2);
    \path (A2_2) edge [->]node [auto] {$\scriptstyle{\operatorname{z}^2_{\CATB}}$} (A4_2);
    \path (A4_2) edge [->]node [auto]
      {$\scriptstyle{\functor{F}_1(\operatorname{w}^1_{\CATA})}$} (A2_0);
    \path (A0_2) edge [->]node [auto,swap] {$\scriptstyle{\operatorname{w}^1_{\CATB}}$} (A2_0);
  \end{tikzpicture}
  \]
  such that $\operatorname{z}^1_{\CATB}$ belongs to $\SETWB$ and both $\gamma^1_{\CATB}$ and
  $\gamma^2_{\CATB}$ are invertible.

 \item\label{A3} Let us fix any pair of objects $B_{\CATA},A_{\CATB}$ and any morphism $f_{\CATB}:
  A_{\CATB}\rightarrow\functor{F}_0(B_{\CATA})$. Then there are an object $A_{\CATA}$, a morphism
  $f_{\CATA}:A_{\CATA}\rightarrow B_{\CATA}$ and data $(A'_{\CATB},\operatorname{v}^1_{\CATB},
  \operatorname{v}^2_{\CATB},\alpha_{\CATB})$ as follows
  
  \[
  \begin{tikzpicture}[xscale=1.8,yscale=-0.6]
    \node (B0_0) at (-1, 0) {}; 
    \node (B1_1) at (5, 0) {}; 
    
    \node (A0_2) at (2, 0) {$A_{\CATB}$};
    \node (A2_2) at (0, 1) {$A'_{\CATB}$};
    \node (A2_4) at (4, 1) {$\functor{F}_0(B_{\CATA})$,};
    \node (A4_2) at (2, 2) {$\functor{F}_0(A_{\CATA})$};
    
    \node (A2_3) at (2, 1) {$\Downarrow\,\alpha_{\CATB}$};
        
    \path (A4_2) edge [->,bend left=15] node [auto,swap]
      {$\scriptstyle{\functor{F}_1(f_{\CATA})}$} (A2_4);
    \path (A0_2) edge [->,bend right=15] node [auto] {$\scriptstyle{f_{\CATB}}$} (A2_4);
    \path (A2_2) edge [->,bend right=15] node [auto]
      {$\scriptstyle{\operatorname{v}^1_{\CATB}}$} (A0_2);
    \path (A2_2) edge [->,bend left=15] node [auto,swap]
      {$\scriptstyle{\operatorname{v}^2_{\CATB}}$} (A4_2);
  \end{tikzpicture}
  \]
  with $\operatorname{v}^1_{\CATB}$ in $\SETWB$, $\operatorname{v}^2_{\CATB}$ in $\SETWBsat$ and
  $\alpha_{\CATB}$ invertible.

 \item\label{A4} Let us fix any pair of objects $A_{\CATA},B_{\CATA}$, any pair of morphisms
  $f^1_{\CATA},f^2_{\CATA}:A_{\CATA}\rightarrow B_{\CATA}$ and any pair of $2$-morphisms
  $\gamma^1_{\CATA},\gamma^2_{\CATA}:f^1_{\CATA}\Rightarrow f^2_{\CATA}$. Moreover, let us fix
  an object $A'_{\CATB}$ and a morphism $\operatorname{z}_{\CATB}:A'_{\CATB}
  \rightarrow\functor{F}_0(A_{\CATA})$ in $\SETWB$. If $\functor{F}_2(\gamma^1_{\CATA})
  \ast i_{\operatorname{z}_{\CATB}}=\functor{F}_2(\gamma^2_{\CATA})\ast
  i_{\operatorname{z}_{\CATB}}$, then there are an object $A'_{\CATA}$ and a morphism
  $\operatorname{z}_{\CATA}:A'_{\CATA}\rightarrow A_{\CATA}$ in $\SETWA$, such that
  $\gamma^1_{\CATA}\ast i_{\operatorname{z}_{\CATA}}=\gamma^2_{\CATA}\ast
  i_{\operatorname{z}_{\CATA}}$.

 \item\label{A5} Let us fix any triple of objects $A_{\CATA},B_{\CATA},A_{\CATB}$, any pair of
  morphisms $f^1_{\CATA},f^2_{\CATA}:A_{\CATA}\rightarrow B_{\CATA}$, any morphism
  $\operatorname{v}_{\CATB}:A_{\CATB}\rightarrow\functor{F}_0(A_{\CATA})$ in $\SETWB$ and any
  $2$-morphism
  
  \[
  \begin{tikzpicture}[xscale=1.8,yscale=-0.6]
    \node (B0_0) at (-1, 0) {}; 
    \node (B1_1) at (5, 0) {}; 
    
    \node (A0_1) at (2, 0) {$\functor{F}_0(A_{\CATA})$};
    \node (A1_0) at (0, 1) {$A_{\CATB}$};
    \node (A1_2) at (4, 1) {$\functor{F}_0(B_{\CATA})$.};
    \node (A2_1) at (2, 2) {$\functor{F}_0(A_{\CATA})$};
    
    \node (A1_1) at (2, 1) {$\Downarrow\,\alpha_{\CATB}$};
        
    \path (A1_0) edge [->,bend right=20]node [auto]
       {$\scriptstyle{\operatorname{v}_{\CATB}}$} (A0_1);
    \path (A1_0) edge [->,bend left=20]node [auto,swap]
      {$\scriptstyle{\operatorname{v}_{\CATB}}$} (A2_1);
    \path (A0_1) edge [->,bend right=20]node [auto]
      {$\scriptstyle{\functor{F}_1(f^1_{\CATA})}$} (A1_2);
    \path (A2_1) edge [->,bend left=20]node [auto,swap]
      {$\scriptstyle{\functor{F}_1(f^2_{\CATA})}$} (A1_2);
  \end{tikzpicture}
  \]
  Then there are a pair of objects
  $A'_{\CATA},A'_{\CATB}$, a triple of morphisms $\operatorname{v}_{\CATA}:A'_{\CATA}\rightarrow
  A_{\CATA}$ in $\SETWA$, $\operatorname{z}_{\CATB}:A'_{\CATB}\rightarrow\functor{F}_0(A'_{\CATA})$
  in $\SETWB$ and $\operatorname{z}'_{\CATB}:A'_{\CATB}\rightarrow A_{\CATB}$,
  a $2$-morphism 
  
  \[
  \begin{tikzpicture}[xscale=1.8,yscale=-0.6]
    \node (B0_0) at (-1, 0) {}; 
    \node (B1_1) at (5, 0) {}; 
    
    \node (A0_1) at (2, 0) {$A_{\CATA}$};
    \node (A1_0) at (0, 1) {$A'_{\CATA}$};
    \node (A1_2) at (4, 1) {$B_{\CATA}$};
    \node (A2_1) at (2, 2) {$A_{\CATA}$};

    \node (A1_1) at (2, 1) {$\Downarrow\,\alpha_{\CATA}$};

    \path (A1_0) edge [->,bend right=20]node [auto] {$\scriptstyle{\operatorname{v}_{\CATA}}$} (A0_1);
    \path (A1_0) edge [->,bend left=20]node [auto,swap]
      {$\scriptstyle{\operatorname{v}_{\CATA}}$} (A2_1);
    \path (A0_1) edge [->,bend right=20]node [auto] {$\scriptstyle{f^1_{\CATA}}$} (A1_2);
    \path (A2_1) edge [->,bend left=20]node [auto,swap] {$\scriptstyle{f^2_{\CATA}}$} (A1_2);
   \end{tikzpicture}
   \]
  and an invertible $2$-morphism
  
  \[
  \begin{tikzpicture}[xscale=1.8,yscale=-0.6]
    \node (B0_0) at (-1, 0) {}; 
    \node (B1_1) at (5, 0) {}; 
    
    \node (A0_1) at (2, 0) {$\functor{F}_0(A'_{\CATA})$};
    \node (A1_0) at (0, 1) {$A'_{\CATB}$};
    \node (A1_2) at (4, 1) {$\functor{F}_0(A_{\CATA})$,};
    \node (A2_1) at (2, 2) {$A_{\CATB}$};
    
    \node (A1_1) at (2, 1) {$\Downarrow\,\sigma_{\CATB}$};

    \path (A1_0) edge [->,bend right=20]node [auto]
      {$\scriptstyle{\operatorname{z}_{\CATB}}$} (A0_1);
    \path (A1_0) edge [->,bend left=20]node [auto,swap]
      {$\scriptstyle{\operatorname{z}'_{\CATB}}$} (A2_1);
    \path (A0_1) edge [->,bend right=20]node [auto] {$\scriptstyle{\functor{F}_1
      (\operatorname{v}_{\CATA})}$} (A1_2);
    \path (A2_1) edge [->,bend left=20]node [auto,swap]
      {$\scriptstyle{\operatorname{v}_{\CATB}}$} (A1_2);
   \end{tikzpicture}
   \]
  such that $\alpha_{\CATB}\ast i_{\operatorname{z}'_{\CATB}}$
  coincides with the following composition:
  
  \begin{equation}\label{eq-13}
  \begin{tikzpicture}[xscale=3.5,yscale=-1.8]
    \node (A0_1) at (1.5, 0.5) {$A_{\CATB}$};
    \node (A0_2) at (1, 1) {$\functor{F}_0(A_{\CATA})$};
    \node (A1_0) at (2.2, 0.75) {$\Downarrow\,\thetab{\functor{F}_1(f^1_{\CATA})}
      {\operatorname{v}_{\CATB}}{\operatorname{z}'_{\CATB}}$};
    \node (A1_1) at (0.35, 1.3) {$\Downarrow\,\sigma_{\CATB}^{-1}$};
    \node (A1_3) at (1.5, 1.35) {$\Downarrow\,\thetaa{\functor{F}_1(f^1_{\CATA})}
      {\functor{F}_1(\operatorname{v}_{\CATA})}{\operatorname{z}_{\CATB}}$};
    \node (A2_0) at (0, 2) {$A'_{\CATB}$};
    \node (A2_1) at (1, 2) {$\functor{F}_0(A'_{\CATA})$};
    \node (A2_2) at (2, 2) {$\Downarrow\,\psi^{\functor{F}}_{f^2_{\CATA},\operatorname{v}_{\CATA}}
      \odot\functor{F}_2(\alpha_{\CATA})\odot\left(\psi^{\functor{F}}_{f^1_{\CATA},
      \operatorname{v}_{\CATA}}\right)^{-1}$};
    \node (A2_3) at (3, 2) {$\functor{F}_0(B_{\CATA})$};
    \node (A3_0) at (2.2, 3.25) {$\Downarrow\,\thetaa{\functor{F}_1(f^2_{\CATA})}
      {\operatorname{v}_{\CATB}}{\operatorname{z}'_{\CATB}}$};
    \node (A3_1) at (0.35, 2.7) {$\Downarrow\,\sigma_{\CATB}$};
    \node (A3_3) at (1.5, 2.65) {$\Downarrow\,\thetab{\functor{F}_1(f^2_{\CATA})}
      {\functor{F}_1(\operatorname{v}_{\CATA})}{\operatorname{z}_{\CATB}}$};
    \node (A4_1) at (1.5, 3.5) {$A_{\CATB}$};
    \node (A4_2) at (1, 3) {$\functor{F}_0(A_{\CATA})$};
    
    \node (B1_1) at (0.8, 0.38) {$\scriptstyle{\operatorname{z}'_{\CATB}}$};
    \node (B2_2) at (0.8, 3.62) {$\scriptstyle{\operatorname{z}'_{\CATB}}$};
    \node (B3_3) at (2.2, 0.38) {$\scriptstyle{\functor{F}_1(f^1_{\CATA})
      \circ\operatorname{v}_{\CATB}}$};
    \node (B4_4) at (2.2, 3.62) {$\scriptstyle{\functor{F}_1(f^2_{\CATA})
      \circ\operatorname{v}_{\CATB}}$};
    \node (B5_5) at (1.5, 3.12) {$\scriptstyle{\functor{F}_1(f^2_{\CATA})}$};
    \node (B6_6) at (1.5, 0.88) {$\scriptstyle{\functor{F}_1(f^1_{\CATA})}$};
    \node (B7_7) at (0.5, 0.88) {$\scriptstyle{\operatorname{v}_{\CATB}
      \circ\operatorname{z}'_{\CATB}}$};
    \node (B8_8) at (0.5, 3.12) {$\scriptstyle{\operatorname{v}_{\CATB}
      \circ\operatorname{z}'_{\CATB}}$};
    \node (C1_1) at (2.4, 1.48) {$\scriptstyle{\functor{F}_1(f^1_{\CATA})
      \circ\functor{F}_1(\operatorname{v}_{\CATA})}$};
    \node (C2_2) at (2.4, 2.52) {$\scriptstyle{\functor{F}_1(f^2_{\CATA})
      \circ\functor{F}_1(\operatorname{v}_{\CATA})}$};

    \draw [->,rounded corners] (A2_0) to (0.2, 3.5) to (A4_1);
    \draw [->,rounded corners] (A2_0) to (0.2, 0.5) to (A0_1);
    \draw [->,rounded corners] (A0_1) to (2.8, 0.5) to (A2_3);
    \draw [->,rounded corners] (A4_1) to (2.8, 3.5) to (A2_3);
    \draw [->,rounded corners] (A0_2) to (2.8, 1) to (A2_3);
    \draw [->,rounded corners] (A4_2) to (2.8, 3) to (A2_3);
    \draw [->,rounded corners] (A2_0) to (0.2, 1) to (A0_2);
    \draw [->,rounded corners] (A2_0) to (0.2, 3) to (A4_2);
    \draw [->,rounded corners] (A2_1) to (1.2, 1.6) to (2.8, 1.6) to (A2_3);
    \draw [->,rounded corners] (A2_1) to (1.2, 2.4) to (2.8, 2.4) to (A2_3);
    
    \path (A2_0) edge [->]node [auto] {$\scriptstyle{\functor{F}_1(\operatorname{v}_{\CATA})
      \circ\operatorname{z}_{\CATB}}$} (A4_2);
    \path (A2_0) edge [->]node [auto] {$\scriptstyle{\operatorname{z}_{\CATB}}$} (A2_1);
    \path (A2_0) edge [->]node [auto,swap] {$\scriptstyle{\functor{F}_1
      (\operatorname{v}_{\CATA})\circ\operatorname{z}_{\CATB}}$} (A0_2);
  \end{tikzpicture}
  \end{equation}
\end{enumerate}
\emphatic{(}here the $2$-morphisms of the form $\theta_{\bullet}$ are the associators of
$\CATB$\emphatic{)}.
\end{theo}

In this way we have established an useful tool for checking if any $2$ bicategories
of fractions $\CATA\left[\SETWAinv\right]$ and $\CATB\left[\SETWBinv\right]$ are equivalent:
it suffices to define a pseudofunctor $\functor{F}:\CATA\rightarrow\CATB$, such that $\functor{F}_1
(\SETWA)\subseteq\SETWBsat$ and such that conditions (\hyperref[A1]{A1}) -- (\hyperref[A5]{A5})
hold.\\

We are going to apply explicitly Theorem~\ref{theo-02} in our next paper~\cite{T7}, where $\CATB$ will be the
$2$-category of proper, effective, differentiable \'etale groupoids and $\SETWB$ will be the class of
all Morita equivalences between such objects. The role of the bicategory $\CATA$ will be played by a
$2$-category $\RedAtl$ whose objects will be reduced orbifold atlases; the class $\SETWA$ will be the
class of all morphisms that are ``refinements'' of reduced orbifold atlases
(see~\cite[Definition~6.1]{T7}).\\

In the second part of this paper we are going to consider the following result about bicategories of
fractions (a direct consequence of~\cite[Theorem~21]{Pr}):

\begin{theo}\label{theo-03}
\cite{Pr} Let us fix any pair $(\CATA,\SETWA)$ satisfying conditions \emphatic{(BF)},
any bicategory $\CATB$ and any pseudofunctor $\functor{F}:\CATA\rightarrow\CATB$. Then the following
facts are equivalent:

\begin{enumerate}[\emphatic{(}1\emphatic{)}]
 \item $\functor{F}$ sends each morphism of $\SETWA$ to an internal equivalence of $\CATB$;
 \item there is a pseudofunctor $\overline{\functor{G}}:\CATA\left[\SETWAinv\right]\rightarrow
  \CATB$ and a pseudonatural equivalence of pseudofunctors $\overline{\kappa}:\functor{F}\Rightarrow
  \overline{\functor{G}}\circ\functor{U}_{\SETWA}$, whose associated pseudofunctor $\CATA\rightarrow
  \operatorname{Cyl}(\CATB)$ sends each morphism of $\SETWA$ to an internal equivalence.
\end{enumerate}
\end{theo}

As a consequence of Theorem~\ref{theo-02}, we are able to improve Theorem~\ref{theo-03} as follows.

\begin{theo}\label{theo-04}
Let us fix any pair $(\CATA,\SETWA)$ satisfying conditions \emphatic{(BF)},
any bicategory $\CATB$ and any pseudofunctor
$\functor{F}:\CATA\rightarrow\CATB$, such that $\functor{F}$ sends each morphism
of $\SETWA$ to an internal equivalence of $\CATB$. Let us also fix any pair
$(\overline{\functor{G}},\overline{\kappa})$ as in \emphatic{Theorem~\ref{theo-03}(2)}. Then
$\overline{\functor{G}}:\CATA\left[\SETWAinv\right]\rightarrow\CATB$
is an equivalence of bicategories if and only if $\functor{F}$
satisfies the the following $5$ conditions.

\begin{enumerate}[\emphatic{(}{B}1\emphatic{)}]\label{B}
 \item\label{B1} Given any object $A_{\CATB}$, there are an object $A_{\CATA}$ and an internal
  equivalence $e_{\CATB}:\functor{F}_0(A_{\CATA})\rightarrow A_{\CATB}$ \emphatic{(}i.e.\ 
  $\functor{F}_0$ is surjective up to internal equivalences\emphatic{)}.

 \item\label{B2} Let us fix any pair of objects $A^1_{\CATA},A^2_{\CATA}$ and any internal equivalence
  $e_{\CATB}:\functor{F}_0(A^1_{\CATA})\rightarrow\functor{F}_0(A^2_{\CATA})$. Then
  there are an object $A^3_{\CATA}$, a pair of morphisms $\operatorname{w}^1_{\CATA}:A^3_{\CATA}
  \rightarrow A^1_{\CATA}$ in $\SETWA$ and $\operatorname{w}^2_{\CATA}:A^3_{\CATA}\rightarrow
  A^2_{\CATA}$ in $\SETWAsat$, an internal equivalence $e'_{\CATB}:\functor{F}_0(A^1_{\CATA})
  \rightarrow\functor{F}_0(A^3_{\CATA})$ and a pair of invertible $2$-morphisms as follows:
  
  \[
  \begin{tikzpicture}[xscale=2.8,yscale=-2.0]
    \node (A0_3) at (3.5, 0.8) {$\functor{F}_0(A^2_{\CATA})$};
    \node (A1_0) at (0.5, 1) {$\functor{F}_0(A^1_{\CATA})$};
    \node (A1_2) at (2, 1) {$\functor{F}_0(A^3_{\CATA})$};
    \node (A2_3) at (3.5, 1.2) {$\functor{F}_0(A^1_{\CATA})$.};

    \node (A0_2) at (1.85, 0.65) {$\Downarrow\,\delta^2_{\CATB}$};
    \node (A2_2) at (1.85, 1.35) {$\Downarrow\,\delta^1_{\CATB}$};
    
    \path (A1_0) edge [->]node [auto] {$\scriptstyle{e'_{\CATB}}$} (A1_2);
    \path (A1_0) edge [->, bend right=30]node [auto] {$\scriptstyle{e_{\CATB}}$} (A0_3);
    \path (A1_2) edge [->]node [auto] {$\scriptstyle{\functor{F}_1
      (\operatorname{w}^2_{\CATA})}$} (A0_3);
    \path (A1_0) edge [->, bend left=30]node [auto,swap]
      {$\scriptstyle{\id_{\functor{F}_0(A^1_{\CATA})}}$} (A2_3);
    \path (A1_2) edge [->]node [auto,swap] {$\scriptstyle{\functor{F}_1
      (\operatorname{w}^1_{\CATA})}$} (A2_3);
  \end{tikzpicture}
  \]

 \item\label{B3} Let us fix any pair of objects $B_{\CATA}$, $A_{\CATB}$ and any morphism $f_{\CATB}:
  A_{\CATB}\rightarrow\functor{F}_0(B_{\CATA})$. Then there are an object $A_{\CATA}$, a morphism
  $f_{\CATA}:A_{\CATA}\rightarrow B_{\CATA}$, an internal equivalence $e_{\CATB}:A_{\CATB}\rightarrow
  \functor{F}_0(A_{\CATA})$ and an invertible $2$-morphism $\alpha_{\CATB}:f_{\CATB}\Rightarrow
  \functor{F}_1(f_{\CATA})\circ e_{\CATB}$.

 \item\label{B4} Let us fix any pair of objects $A_{\CATA},B_{\CATA}$, any pair of morphisms
  $f^1_{\CATA},f^2_{\CATA}:A_{\CATA}\rightarrow B_{\CATA}$ and any pair of $2$-morphisms
  $\gamma^1_{\CATA},\gamma^2_{\CATA}:f^1_{\CATA}\Rightarrow f^2_{\CATA}$, such that
  that $\functor{F}_2(\gamma^1_{\CATA})=\functor{F}_2(\gamma^2_{\CATA})$. Then there are an object
  $A'_{\CATA}$ and a morphism $\operatorname{z}_{\CATA}:A'_{\CATA}\rightarrow A_{\CATA}$ in
  $\SETWA$, such that $\gamma^1_{\CATA}\ast i_{\operatorname{z}_{\CATA}}=\gamma^2_{\CATA}\ast
  i_{\operatorname{z}_{\CATA}}$.
 
 \item\label{B5} Let us fix any pair of objects $A_{\CATA},B_{\CATA}$, any pair of morphisms
  $f^1_{\CATA},f^2_{\CATA}:A_{\CATA}\rightarrow B_{\CATA}$ and any $2$-morphism $\alpha_{\CATB}:
  \functor{F}_1(f^1_{\CATA})\Rightarrow\functor{F}_1(f^2_{\CATA})$. Then there are an object
  $A'_{\CATA}$, a morphism $\operatorname{v}_{\CATA}:A'_{\CATA}\rightarrow A_{\CATA}$ in $\SETWA$ and
  a $2$-morphism $\alpha_{\CATA}:f^1_{\CATA}\circ\operatorname{v}_{\CATA}\Rightarrow f^2_{\CATA}\circ
  \operatorname{v}_{\CATA}$, such that
  
  \[\alpha_{\CATB}\ast i_{\functor{F}_1(\operatorname{v}_{\CATA})}=\psi^{\functor{F}}_{f^2_{\CATA},
  \operatorname{v}_{\CATA}}\odot\functor{F}_2(\alpha_{\CATA})\odot\Big(
  \psi^{\functor{F}}_{f^1_{\CATA},\operatorname{v}_{\CATA}}\Big)^{-1}.\]
\end{enumerate}
\end{theo}

Therefore, Theorem~\ref{theo-04} fixes the incorrect statement of~\cite[Proposition~24]{Pr}: such
a Proposition was giving necessary and sufficient conditions such that
$\overline{\functor{G}}$ as above is an equivalence of bicategories, but it turned out that such
conditions were only sufficient but not necessary (we refer to the Appendix for more details about this).\\

In particular, Theorem~\ref{theo-04} implies the following result.

\begin{cor}\label{cor-01}
Given any pair $(\CATA,\SETWA)$ satisfying conditions \emphatic{(BF)}
and any bicategory $\CATB$, the following facts are equivalent:

\begin{enumerate}[\emphatic{(}a\emphatic{)}]
 \item there is an equivalence of bicategories $\CATA\left [\SETWAinv\right]\rightarrow\CATB$;
 \item there is a pseudofunctor $\functor{F}:\CATA\rightarrow\CATB$ such that
  \begin{itemize}
  \item $\functor{F}$ sends each morphism of $\SETWA$ to an internal equivalence of $\CATB$;
  \item $\functor{F}$ satisfies conditions
   \emphatic{(\hyperref[B1]{B1})} -- \emphatic{(\hyperref[B5]{B5})}.
  \end{itemize}
\end{enumerate}
\end{cor}

Conditions (\hyperref[A]{A}) and (\hyperref[B]{B}) simplify considerably in the case when $\CATA$
and $\CATB$ are both $1$-categories (considered as trivial bicategories). In that case, all the
$2$-morphisms in those conditions are trivial, thus all such conditions are much shorter to
write; moreover, conditions (\hyperref[A4]{A4}) and (\hyperref[B4]{B4}) are automatically satisfied.\\

In all this paper we are going to use the axiom of choice, that we will assume from now on without
further mention. The reason for this is twofold. First of all, the axiom of choice is used heavily
in~\cite{Pr} in order to construct bicategories of fractions. In~\cite[Corollary~0.6]{T3} we proved
that under some restrictive hypothesis the axiom of choice is not necessary, but in the general case
we need it in order to consider any of the bicategories of fractions mentioned above.
Secondly, even in the cases when the axiom of choice is not necessary in order to construct the
bicategories $\CATA\left[\SETWAinv\right]$ and $\CATB\left[\SETWBinv\right]$, we will have to use
often Theorem~\ref{theo-01}. Such a result was proved in~\cite{T4} using several times the universal
property of any bicategory of fractions $\CATC\left[\SETWinv\right]$, as stated
in~\cite[Theorem~21]{Pr}; the proof of that result in~\cite{Pr} requires most of the time the axiom
of choice (since for each morphism $\operatorname{w}$ of $\SETW$ one
has to choose a quasi-inverse for $\functor{U}_{\SETW}(\operatorname{w})$ in 
$\CATC\left[\SETWinv\right]$), hence we have to assume the axiom of choice also in this paper.

\section{Notations}
Given any bicategory $\CATC$, we denote its objects by $A,B,\cdots$, its morphisms by $f,g,\cdots$
and its $2$-morphisms by $\alpha,\beta,\cdots$ (we will use $A_{\CATC},
f_{\CATC},\alpha_{\CATC},\cdots$ if we have to recall that they belong to $\CATC$ when we are using
more than one bicategory in the computations). Given any triple of morphisms $f:A\rightarrow B$,
$g:B\rightarrow C$, $h:C\rightarrow D$ in $\CATC$, we denote by $\thetaa{h}{g}{f}$ the associator
$h\circ(g\circ f)\Rightarrow(h\circ g)\circ f$ that is part of the structure of $\CATC$; we denote by
$\pi_f:f\circ\id_A\Rightarrow f$ and $\upsilon_f:\id_B\circ f\Rightarrow f$ the right and left unitors
for $\CATC$ relative to any morphism $f$ as above.
We denote any pseudofunctor from $\CATC$ to another bicategory $\CATD$ by $\functor{F}=
(\functor{F}_0,\functor{F}_1,\functor{F}_2,$ $\psi_{\bullet}^{\functor{F}},
\sigma_{\bullet}^{\functor{F}}):\CATC\rightarrow\CATD$. Here for each pair of morphisms $f,g$
as above, $\psi^{\functor{F}}_{g,f}$ is the associator from $\functor{F}_1(g\circ f)$ to
$\functor{F}_1(g)\circ\functor{F}_1(f)$ and for each object $A$, $\sigma^{\functor{F}}_A$ is the
unitor from $\functor{F}_1(\id_A)$ to $\id_{\functor{F}_0(A)}$.\\

For all the notations about axioms (BF1) -- (BF5) and the construction of bicategories of
fractions we refer either to the original reference~\cite{Pr} or to our previous paper~\cite{T3}.\\

We recall from~\cite[(1.33)]{St} that given any pair of bicategories $\CATC$ and $\CATD$, a
pseudofunctor $\functor{M}:\CATC\rightarrow\CATD$ is a \emph{weak equivalence of bicategories} (also
known as \emph{weak biequivalence}) if and only if the following conditions hold:

\begin{enumerate}[({X}1)]\label{X}
 \item\label{X1} for each object $A_{\CATD}$, there are an object $A_{\CATC}$ and an internal
  equivalence from $\functor{M}_0(A_{\CATC})$ to $A_{\CATD}$ in $\CATD$;
 \item\label{X2} for each pair of objects $A_{\CATC},B_{\CATC}$, the functor $\functor{M}(A_{\CATC},
  B_{\CATC})$ is an equivalence of categories, i.e.
  
  \begin{enumerate}[({X2}a)]
   \item\label{X2a} for each morphism $f_{\CATD}:\functor{M}_0(A_{\CATC})\rightarrow\functor{M}_0
    (B_{\CATC})$, there are a morphism $f_{\CATC}:A_{\CATC}\rightarrow B_{\CATC}$ and an invertible
    $2$-morphism $\alpha_{\CATD}:\functor{M}_1(f_{\CATC})\Rightarrow f_{\CATD}$;
   \item\label{X2b} for each pair of morphisms $f^1_{\CATC},f^2_{\CATC}:A_{\CATC}\rightarrow
    B_{\CATC}$ and for each pair of $2$-morphisms $\alpha^1_{\CATC},\alpha^2_{\CATC}:f^1_{\CATC}
    \Rightarrow f^2_{\CATC}$, if $\functor{M}_2(\alpha^1_{\CATC})=\functor{M}_2(\alpha^2_{\CATC})$,
    then $\alpha^1_{\CATC}=\alpha^2_{\CATC}$;
   \item\label{X2c} for each pair $f^1_{\CATC},f^2_{\CATC}$ as above and for each $2$-morphism
    $\alpha_{\CATD}:\functor{M}_1(f^1_{\CATC})\Rightarrow\functor{M}_1(f^2_{\CATC})$, there is a
    $2$-morphism $\alpha_{\CATC}:f^1_{\CATC}\Rightarrow f^2_{\CATC}$ such that $\functor{M}_2
    (\alpha_{\CATC})=\alpha_{\CATD}$.
  \end{enumerate}
\end{enumerate}

\emph{Since in all this paper we assume the axiom of choice, then each weak equivalence of
bicategories is a \emphatic{(}strong\emphatic{)} equivalence of bicategories} (also known as
biequivalence, see~\cite[\S~1]{PW}), i.e.\ it admits a quasi-inverse. Conversely, each strong
equivalence of bicategories is a weak equivalence. So in the present setup we will simply write
``equivalence of bicategories'' meaning weak, equivalently strong, equivalence.\\

In the next $2$ sections we will fix any set of data $(\CATA,\SETWA,\CATB,\SETWB,\functor{F})$ as
in Theorem~\ref{theo-01}(i) and we will prove that conditions (\hyperref[X]{X}) for the
pseudofunctor $\functor{M}:=\widetilde{\functor{G}}$ are equivalent to conditions (\hyperref[A]{A})
for the pseudofunctor $\functor{F}$. To be more precise, the chain of implications that we are going
to prove is the following:

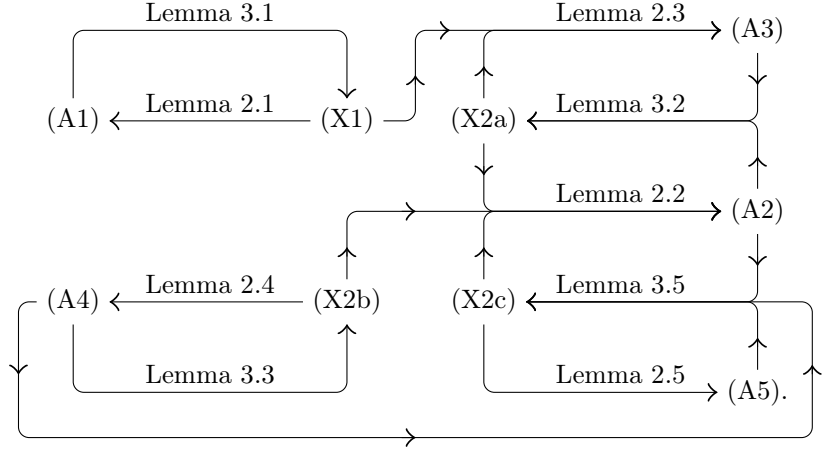
\[
\begin{tikzpicture}[xscale=1.8,yscale=-1.2]
    \node (A0_2) at (2, 0) {(\hyperref[A3]{A3})};
    \node (A1_0) at (0, 1) {(\hyperref[X2a]{X2a})};
    \node (A2_2) at (2, 2) {(\hyperref[A2]{A2})};
    \node (A3_0) at (0, 3) {(\hyperref[X2c]{X2c})};
    \node (A4_2) at (2, 4) {(\hyperref[A5]{A5}).};
    \node (C2_1) at (-1, 3) {(\hyperref[X2b]{X2b})};
    \node (C2_2) at (-3, 3) {(\hyperref[A4]{A4})};
    \node (D2_1) at (-1, 1) {(\hyperref[X1]{X1})};
    \node (D2_2) at (-3, 1) {(\hyperref[A1]{A1})};
    
    \node (B0_1) at (1, -0.2) {Lemma~\ref{lem-08}};
    \node (B1_1) at (1, 0.8) {Lemma~\ref{lem-05}};
    \node (B2_1) at (1, 1.8) {Lemma~\ref{lem-07}};
    \node (B3_1) at (1, 2.8) {Lemma~\ref{lem-12}};
    \node (B4_1) at (1, 3.8) {Lemma~\ref{lem-06}};
    
    \node (E0_1) at (-2, 3.8) {Lemma~\ref{lem-09}};
    \node (E1_1) at (-2, 2.8) {Lemma~\ref{lem-10}};
    \node (E3_1) at (-2, 0.8) {Lemma~\ref{lem-01}};
    \node (E4_1) at (-2, -0.2) {Lemma~\ref{lem-11}};
        
    \thickarrowpersonalized{(C2_2) to (-3, 4) to (-1, 4) to (C2_1)}
    \thickarrowpersonalized{(C2_1) to (C2_2)}
    
    \thickarrowpersonalized{(D2_2) to (-3, 0) to (-1, 0) to (D2_1)}
    \thickarrowpersonalized{(D2_1) to (D2_2)}
    \thickarrowpersonalized{(D2_1) to (-0.5, 1) to (-0.5, 0) to (A0_2)}
    \thickarrowpersonalized{(-0.5, 0.6) to (-0.5, 0.5)}
    \thickarrowpersonalized{(-0.3, 0) to (-0.25, 0)}
    
    \thickarrowpersonalized{(A1_0) to (0, 0) to (A0_2)}
    \thickarrowpersonalized{(0, 0.5) to (0, 0.4)}
    \thickarrowpersonalized{(A3_0) to (0, 4) to (A4_2)}
    
    \thickarrowpersonalized{(A1_0) to (0, 2) to (A2_2)}
    \thickarrowpersonalized{(0, 1.5) to (0, 1.6)}
    \thickarrowpersonalized{(A3_0) to (0, 2) to (A2_2)}
    \thickarrowpersonalized{(0, 2.5) to (0, 2.4)}
    \thickarrowpersonalized{(C2_1) to (-1, 2) to (A2_2)}
    \thickarrowpersonalized{(-0.6, 2) to (-0.5, 2)}
    \thickarrowpersonalized{(-1, 2.5) to (-1, 2.4)}
    
    \thickarrowpersonalized{(A0_2) to (2, 1) to (A1_0)}
    \thickarrowpersonalized{(2, 0.5) to (2, 0.6)}
    \thickarrowpersonalized{(A2_2) to (2, 1) to (A1_0)}
    \thickarrowpersonalized{(2, 1.5) to (2, 1.4)}
  
    \thickarrowpersonalized{(A2_2) to (2, 3) to (A3_0)}
    \thickarrowpersonalized{(2, 2.5) to (2, 2.6)}
    \thickarrowpersonalized{(A4_2) to (2, 3) to (A3_0)}
    \thickarrowpersonalized{(2, 3.5) to (2, 3.4)}
    \draw[rounded corners,-] (C2_2) to (-3.4, 3) to (-3.4, 4.5) to (2.4, 4.5) to (2.4, 3) to (A3_0);
    \thickarrowpersonalized{(-3.4, 3.7) to (-3.4, 3.8)}
    \thickarrowpersonalized{(-0.6, 4.5) to (-0.5, 4.5)}
    \thickarrowpersonalized{(2.4, 3.8) to (2.4, 3.7)}
\end{tikzpicture}
\]

In the special case when $\functor{F}$ sends each morphism of $\SETWA$ to an internal equivalence
of $\CATB$ and $\SETWB$ is the class $\SETWBmin$ of quasi-units of $\CATB$
(see~\cite[Definition~3.3]{PP}), we will show in Proposition~\ref{prop-02} that each
condition of type (\hyperref[A]{A}) above
can be replaced by the analogous condition of type (\hyperref[B]{B}).

\section{Necessity of conditions (A1) -- (A5)}
In all this section and in the next one, \emph{we fix any $2$ pairs $(\CATA,\SETWA)$ and
$(\CATB,\SETWB)$ that satisfy conditions} (BF)
\emph{and any pseudofunctor $\functor{F}:\CATA\rightarrow\CATB$, such that $\functor{F}_1(\SETWA)$
$\subseteq\SETWBsat$}; we denote by $\widetilde{\functor{G}}:\CATA\Big[\SETWAinv\Big]\rightarrow
\CATB\Big[\SETWBsatinv\Big]$ the pseudofunctor mentioned in Theorem~\ref{theo-01}. In this section
we will prove that conditions (\hyperref[X]{X}) for $\widetilde{\functor{G}}$ imply
conditions (\hyperref[A]{A}) for $\functor{F}$.

\begin{lem}\label{lem-01}
If $\widetilde{\functor{G}}$ satisfies \emphatic{(\hyperref[X1]{X1})}, then $\functor{F}$
satisfies \emphatic{(\hyperref[A1]{A1})}.
\end{lem}

\begin{proof}
By (\hyperref[X1]{X1}), given any object
$A_{\CATB}$ there are an object $A_{\CATA}$ and an internal equivalence $\underline{e}_{\CATB}$ from
$\widetilde{\functor{G}}_0(A_{\CATA})=\functor{F}_0(A_{\CATA})$ to $A_{\CATB}$ in $\CATB[
\SETWBsatinv]$.
By~\cite[Proposition~0.1]{T4} applied to $(\CATB,\SETWB)$, there are a pseudofunctor
$\functor{H}:\CATB
[\SETWBsatinv]\rightarrow\CATB[\SETWBinv]$ and a pseudonatural equivalence
of pseudofunctors $\tau:\functor{U}_{\SETWB}\Rightarrow\functor{H}\circ
\functor{U}_{\SETWBsat}$. Since $\underline{e}_{\CATB}$ is an internal equivalence, then 
there is an internal equivalence in $\CATB\left[\SETWBinv\right]$

\begin{equation}\label{eq-23}
\functor{H}_1(\underline{e}_{\CATB}):\,\,\functor{H}_0\circ\functor{F}_0(A_{\CATA})
\longrightarrow\functor{H}_0(A_{\CATB}).
\end{equation}

Moreover, since $\tau$ is a pseudonatural equivalence of pseudofunctors, then there are
internal equivalences

\begin{gather}
\nonumber \tau_{\functor{F}_0(A_{\CATA})}:\,\functor{F}_0(A_{\CATA})=\functor{U}_{\SETWB,0}\circ
 \functor{F}_0(A_{\CATA})\longrightarrow \\
\label{eq-25} \phantom{\Big(}\longrightarrow\functor{H}_0\circ\functor{U}_{\SETWBsat,0}\circ
 \functor{F}_0(A_{\CATA})=\functor{H}_0\circ\functor{F}_0(A_{\CATA})
\end{gather}
and

\begin{equation}\label{eq-10}
\tau_{A_\CATB}:\,A_{\CATB}=\functor{U}_{\SETWBsat,0}(A_{\CATB})\longrightarrow\functor{H}_0
\circ\functor{U}_{\SETWBsat,0}(A_{\CATB})=\functor{H}_0(A_{\CATB}).
\end{equation}

Composing \eqref{eq-25}, \eqref{eq-23} and any quasi-inverse for \eqref{eq-10}, we get an
internal equivalence from $\functor{F}_0(A_{\CATA})$ to $A_{\CATB}$ in
$\CATB\left[\SETWBinv\right]$; then it suffices to apply~\cite[Corollary~2.7]{T4} for
$(\CATB,\SETWB)$.
\end{proof}

\begin{lem}\label{lem-07}
If $\widetilde{\functor{G}}$ satisfies \emphatic{(\hyperref[X2a]{X2a})},
\emphatic{(\hyperref[X2b]{X2b})} and \emphatic{(\hyperref[X2c]{X2c})}, then
$\functor{F}$ satisfies \emphatic{(\hyperref[A2]{A2})}.
\end{lem}

\begin{proof}
Let us fix any set of data as in \eqref{eq-08} with $\operatorname{w}^1_{\CATB}$ in $\SETWB$
and $\operatorname{w}^2_{\CATB}\in\SETWBsat$. By (\hyperref[X2a]{X2a}) for $\functor{M}:=
\widetilde{\functor{G}}$, there are a morphism
$\underline{f}_{\CATA}:A^1_{\CATA}\rightarrow A^2_{\CATA}$ in $\CATA\left[\SETWAinv\right]$ and an
invertible $2$-morphism in $\CATB\left[\SETWBsatinv\right]$

\[\Gamma_{\CATB}:\,\widetilde{\functor{G}}_1\left(\underline{f}_{\CATA}\right)
\Longrightarrow\Big(A_{\CATB},\operatorname{w}^1_{\CATB},\operatorname{w}^2_{\CATB}\Big).\]

Since $\SETWB\subseteq\SETWBsat$ and $\SETWBsatsat=\SETWBsatsat$ (see~\cite[Remark~2.3 and
Proposition~2.11(i)]{T4}), then by~\cite[Corollary~2.7]{T4} for $(\CATB,\SETWBsat)$ we have
that the target of $\Gamma_{\CATB}$ is an internal equivalence in $\CATB\left[\SETWBsatinv
\right]$. Since $\Gamma_{\CATB}$ is an invertible $2$-morphism, then we conclude that
also $\widetilde{\functor{G}}_1
(\underline{f}_{\CATA}):\functor{F}_0(A^1_{\CATA})\rightarrow\functor{F}_0(A^2_{\CATA})$
is an internal equivalence. So there are an internal equivalence
$\underline{g}_{\CATB}:\functor{F}_0(A^2_{\CATA})\rightarrow\functor{F}_0(A^1_{\CATA})$ and invertible
$2$-morphisms in $\CATB\left[\SETWBsatinv\right]$ as follows

\[\Delta_{\CATB}:\,\,\id_{\functor{F}_0(A^1_{\CATA})}\Longrightarrow\underline{g}_{\CATB}\circ
\widetilde{\functor{G}}_1\left(\underline{f}_{\CATA}\right)\quad\quad\textrm{and}\quad\quad
\Xi_{\CATB}:\,\,\widetilde{\functor{G}}_1\left(\underline{f}_{\CATA}\right)\circ
\underline{g}_{\CATB}\Longrightarrow\id_{\functor{F}_0(A^2_{\CATA})}\]
(here the identity $\id_{\functor{F}_0(A^1_{\CATA})}$ belongs to $\CATB\left[\SETWBsatinv\right]$, so
it is given by the triple $(\functor{F}_0(A^1_{\CATA}),$ $\id_{\functor{F}_0(A^1_{\CATA})},
\id_{\functor{F}_0(A^1_{\CATA})})$, and analogously for $\id_{\functor{F}_0(A^2_{\CATA})}$).
By (\hyperref[X2a]{X2a}) for $\underline{g}_{\CATB}$, there are a morphism
$\underline{g}_{\CATA}:A^2_{\CATA}\rightarrow
A^1_{\CATA}$ in $\CATA\left[\SETWAinv\right]$ and an invertible $2$-morphism $\Omega_{\CATB}:
\widetilde{\functor{G}}_1(\underline{g}_{\CATA})\Rightarrow\underline{g}_{\CATB}$
in $\CATB\left[\SETWBsatinv\right]$. Let us denote by

\[\Psi^{\widetilde{\functor{G}}}_{\underline{g}_{\CATA},\underline{f}_{\CATA}}:\,
\widetilde{\functor{G}}_1\left(\underline{g}_{\CATA}\circ\underline{f}_{\CATA}\right)\Longrightarrow
\widetilde{\functor{G}}_1\left(\underline{g}_{\CATA}\right)\circ\widetilde{\functor{G}}_1
\left(\underline{f}_{\CATA}\right)\]
the associator of $\widetilde{\functor{G}}$ for the pair $(\underline{g}_{\CATA},
\underline{f}_{\CATA})$ and by

\[\Sigma_{A^1_{\CATA}}^{\widetilde{\functor{G}}}:\,\,\widetilde{\functor{G}}_1\left(\id_{A^1_{\CATA}}
\right)\Longrightarrow\id_{\widetilde{\functor{G}}_0(A^1_{\CATA})}=\id_{\functor{F}_0(A^1_{\CATA})}\]
the unitor of $\widetilde{\functor{G}}$ for $A^1_{\CATA}$. Then it makes sense to consider the
invertible $2$-morphism

\[\widetilde{\Delta}_{\CATB}:=\left(\Psi_{\underline{g}_{\CATA},
\underline{f}_{\CATA}}^{\widetilde{\functor{G}}}\right)^{-1}\odot\Big(
\Omega_{\CATB}^{-1}\ast i_{\widetilde{\functor{G}}_1(\underline{f}_{\CATA})}\Big)\odot
\Delta_{\CATB}\odot\Sigma_{A^1_{\CATA}}^{\widetilde{\functor{G}}}:\,\,\widetilde{\functor{G}}_1
\left(\id_{A^1_{\CATA}}\right)\Longrightarrow\widetilde{\functor{G}}_1\left(\underline{g}_{\CATA}
\circ\underline{f}_{\CATA}\right).\]

By (\hyperref[X2b]{X2b}) and (\hyperref[X2c]{X2c}) for $\functor{M}:=\widetilde{\functor{G}}$,
there is a unique invertible $2$-morphism $\Delta_{\CATA}:
\id_{A^1_{\CATA}}\Rightarrow\underline{g}_{\CATA}\circ\underline{f}_{\CATA}$ in $\CATA\left[\SETWAinv
\right]$, such that
$\widetilde{\functor{G}}_2(\Delta_{\CATA})=\widetilde{\Delta}_{\CATB}$. Analogously, there is
an invertible $2$-morphism $\Xi_{\CATA}:\underline{f}_{\CATA}\circ\underline{g}_{\CATA}\Rightarrow
\id_{A^2_{\CATA}}$. This proves that $\underline{f}_{\CATA}$ is an internal equivalence in $\CATA
\left[\SETWAinv\right]$. By~\cite[Corollary~2.7]{T4} for $(\CATA,\SETWA)$, we get that necessarily
$\underline{f}_{\CATA}$ has the following form

\[
\begin{tikzpicture}[xscale=2.4,yscale=-1.2]
    \node (A0_0) at (-0.2, 0) {$\underline{f}_{\CATA}=\Big(A^1_{\CATA}$};
    \node (A0_1) at (1, 0) {$A^3_{\CATA}$};
    \node (A0_2) at (2, 0) {$A^2_{\CATA}\Big)$,};
    
    \path (A0_1) edge [->]node [auto,swap] {$\scriptstyle{\operatorname{w}^1_{\CATA}}$} (A0_0);
    \path (A0_1) edge [->]node [auto] {$\scriptstyle{\operatorname{w}^2_{\CATA}}$} (A0_2);
\end{tikzpicture}
\]
with $\operatorname{w}^1_{\CATA}$ in $\SETWA$ and $\operatorname{w}^2_{\CATA}$ in $\SETWAsat$. Now
we use the description of $\widetilde{\functor{G}}_1$ in Theorem~\ref{theo-01}
and~\cite[Proposition~0.8(ii)]{T3} for $(\CATB,\SETWBsat)$; so we have that $\Gamma^{-1}_{\CATB}$ is 
represented by a set of data as in the internal part of the following diagram:

\[
\begin{tikzpicture}[xscale=1.8,yscale=-0.8]
    \node (A0_2) at (2, 0) {$A_{\CATB}$};
    \node (A2_2) at (2, 2) {$\overline{A}_{\CATB}$};
    \node (A2_0) at (0, 2) {$\functor{F}_0(A^1_{\CATA})$};
    \node (A2_4) at (4, 2) {$\functor{F}_0(A^2_{\CATA})$,};
    \node (A4_2) at (2, 4) {$\functor{F}_0(A^3_{\CATA})$};
    
    \node (A2_3) at (2.8, 2) {$\Downarrow\,\eta^2_{\CATB}$};
    \node (A2_1) at (1.2, 2) {$\Downarrow\,\eta^1_{\CATB}$};
    
    \path (A4_2) edge [->]node [auto,swap]
      {$\scriptstyle{\functor{F}_1(\operatorname{w}^2_{\CATA})}$} (A2_4);
    \path (A0_2) edge [->]node [auto] {$\scriptstyle{\operatorname{w}^2_{\CATB}}$} (A2_4);
    \path (A2_2) edge [->]node [auto,swap] {$\scriptstyle{\operatorname{u}^1_{\CATB}}$} (A0_2);
    \path (A2_2) edge [->]node [auto] {$\scriptstyle{\operatorname{u}^2_{\CATB}}$} (A4_2);
    \path (A4_2) edge [->]node [auto]
      {$\scriptstyle{\functor{F}_1(\operatorname{w}^1_{\CATA})}$} (A2_0);
    \path (A0_2) edge [->]node [auto,swap] {$\scriptstyle{\operatorname{w}^1_{\CATB}}$} (A2_0);
\end{tikzpicture}
\]
such that $\operatorname{w}^1_{\CATB}\circ\operatorname{u}^1_{\CATB}$ belongs to $\SETWBsat$ and 
both $\eta^1_{\CATB}$ and $\eta^2_{\CATB}$ are invertible. Since $\operatorname{w}^1_{\CATB}$ belongs
to $\SETWB\subseteq\SETWBsat$, then by~\cite[Proposition~2.11(ii)]{T4} we get that
$\operatorname{u}^1_{\CATB}$ belongs to $\SETWBsat$. By definition of $\SETWBsat$, there are
an object $A'_{\CATB}$ and a morphism $\operatorname{t}_{\CATB}:A'_{\CATB}\rightarrow
\overline{A}_{\CATB}$, such that $\operatorname{u}^1_{\CATB}\circ\operatorname{t}_{\CATB}$ belongs to
$\SETWB$. Then in order to conclude that (\hyperref[A2]{A2}) holds, it suffices to define for each $m=1,2$
$\operatorname{z}^m_{\CATB}:=\operatorname{u}^m_{\CATB}\circ\operatorname{t}_{\CATB}$ and

\[\gamma^m_{\CATB}:=\thetab{\functor{F}_1(\operatorname{w}^m_{\CATA})}{\operatorname{u}^2_{\CATB}}
{\operatorname{t}_{\CATB}}\odot\Big(\eta^m_{\CATB}\ast i_{\operatorname{t}_{\CATB}}\Big)\odot
\thetaa{\operatorname{w}^m_{\CATB}}{\operatorname{u}^1_{\CATB}}{\operatorname{t}_{\CATB}}:\,\,
\operatorname{w}^m_{\CATB}\circ\operatorname{z}^1_{\CATB}\Longrightarrow\functor{F}_1
(\operatorname{w}^m_{\CATA})\circ\operatorname{z}^2_{\CATB}.\]
\end{proof}

\begin{lem}\label{lem-08}
If $\widetilde{\functor{G}}$ satisfies \emphatic{(\hyperref[X1]{X1})} and
\emphatic{(\hyperref[X2a]{X2a})}, then $\functor{F}$ satisfies
\emphatic{(\hyperref[A3]{A3})}.
\end{lem}

\begin{proof}
Let us fix any pair of objects $B_{\CATA},A_{\CATB}$ and any morphism $f_{\CATB}:A_{\CATB}\rightarrow
\functor{F}_0(B_{\CATA})$. By (\hyperref[X1]{X1}) and Lemma~\ref{lem-01}, (\hyperref[A1]{A1}) holds
for $\functor{F}$, so there are a pair of objects
$\overline{A}_{\CATA},\overline{A}_{\CATB}$ and a pair of morphism $\operatorname{w}^1_{\CATB}$ in
$\SETWB$ and $\operatorname{w}^2_{\CATB}$ in $\SETWBsat$ as follows

\[
\begin{tikzpicture}[xscale=2.4,yscale=-1.2]
    \node (A0_0) at (-0.1, 0) {$\functor{F}_0(\overline{A}_{\CATA})$};
    \node (A0_1) at (1, 0) {$\overline{A}_{\CATB}$};
    \node (A0_2) at (2, 0) {$A_{\CATB}$.};
    
    \path (A0_1) edge [->]node [auto,swap] {$\scriptstyle{\operatorname{w}^1_{\CATB}}$} (A0_0);
    \path (A0_1) edge [->]node [auto] {$\scriptstyle{\operatorname{w}^2_{\CATB}}$} (A0_2);
\end{tikzpicture}
\]

Let us consider the morphism in $\CATB\left[\SETWBsatinv\right]$ defined as follows

\[
\begin{tikzpicture}[xscale=2.6,yscale=-1.2]
    \node (A0_0) at (-0.2, 0) {$\underline{f}_{\CATB}:=\Big(\functor{F}_0(\overline{A}_{\CATA})$};
    \node (A0_1) at (1, 0) {$\overline{A}_{\CATB}$};
    \node (A0_2) at (2, 0) {$\functor{F}_0(B_{\CATA})\Big)$.};
    
    \path (A0_1) edge [->]node [auto,swap] {$\scriptstyle{\operatorname{w}^1_{\CATB}}$} (A0_0);
    \path (A0_1) edge [->]node [auto]
      {$\scriptstyle{f_{\CATB}\circ\operatorname{w}^2_{\CATB}}$} (A0_2);
\end{tikzpicture}
\]

By condition (\hyperref[X2a]{X2a}), there are a morphism $\underline{f}_{\CATA}:\overline{A}_{\CATA}
\rightarrow B_{\CATA}$ in $\CATA \left[\SETWAinv\right]$ and an invertible
$2$-morphism

\[\Gamma_{\CATB}:\,\,\widetilde{\functor{G}}_1\left(\underline{f}_{\CATA}\right)\Longrightarrow
\underline{f}_{\CATB}\]
in $\CATB\left[\SETWBsatinv\right]$. By the description of bicategories of fractions
in~\cite[\S~2.2]{Pr},
$\underline{f}_{\CATA}$ is given by a triple $(A_{\CATA},\operatorname{w}_{\CATA},f_{\CATA})$
as follows

\[
\begin{tikzpicture}[xscale=2.4,yscale=-1.2]
    \node (A0_0) at (-0.2, 0) {$\underline{f}_{\CATA}:=\Big(\overline{A}_{\CATA}$};
    \node (A0_1) at (1, 0) {$A_{\CATA}$};
    \node (A0_2) at (2, 0) {$B_{\CATA}\Big)$};
    
    \path (A0_1) edge [->]node [auto,swap] {$\scriptstyle{\operatorname{w}_{\CATA}}$} (A0_0);
    \path (A0_1) edge [->]node [auto] {$\scriptstyle{f_{\CATA}}$} (A0_2);
\end{tikzpicture}
\]
with $\operatorname{w}_{\CATA}$ in $\SETWA$, so by the description of $\widetilde{\functor{G}}_1$ in
Theorem~\ref{theo-01}, the source of $\Gamma_{\CATB}$ is the triple $(\functor{F}_0(A_{\CATA}),
\functor{F}_1(\operatorname{w}_{\CATA}),\functor{F}_1(f_{\CATA}))$. By~\cite[Proposition~0.8(ii)]{T3}
for the pair $(\CATC,\SETW):=(\CATB,\SETWB)$, $\Gamma_{\CATB}^{-1}$ is represented by a set
of data in $\CATB$ as follows

\[
\begin{tikzpicture}[xscale=1.8,yscale=-0.8]
    \node (A0_2) at (2, 0) {$\overline{A}_{\CATB}$};
    \node (A2_2) at (2, 2) {$\widetilde{A}_{\CATB}$};
    \node (A2_0) at (0, 2) {$\functor{F}_0(\overline{A}_{\CATA})$};
    \node (A2_4) at (4, 2) {$\functor{F}_0(B_{\CATA})$,};
    \node (A4_2) at (2, 4) {$\functor{F}_0(A_{\CATA})$};
    
    \node (A2_3) at (2.8, 2) {$\Downarrow\,\gamma^2_{\CATB}$};
    \node (A2_1) at (1.2, 2) {$\Downarrow\,\gamma^1_{\CATB}$};
    
    \path (A4_2) edge [->]node [auto,swap] {$\scriptstyle{\functor{F}_1(f_{\CATA})}$} (A2_4);
    \path (A0_2) edge [->]node [auto]
      {$\scriptstyle{f_{\CATB}\circ\operatorname{w}^2_{\CATB}}$} (A2_4);
    \path (A2_2) edge [->]node [auto,swap] {$\scriptstyle{\operatorname{z}^1_{\CATB}}$} (A0_2);
    \path (A2_2) edge [->]node [auto] {$\scriptstyle{\operatorname{z}^2_{\CATB}}$} (A4_2);
    \path (A4_2) edge [->]node [auto]
      {$\scriptstyle{\functor{F}_1(\operatorname{w}_{\CATA})}$} (A2_0);
    \path (A0_2) edge [->]node [auto,swap] {$\scriptstyle{\operatorname{w}^1_{\CATB}}$} (A2_0);
\end{tikzpicture}
\]
such that $\operatorname{w}^1_{\CATB}\circ\operatorname{z}^1_{\CATB}$ belong to $\SETWBsat$ and both
$\gamma^1_{\CATB}$ and $\gamma^2_{\CATB}$ are invertible. Since $\operatorname{w}^1_{\CATB}$ belongs
to $\SETWB\subseteq\SETWBsat$, then by~\cite[Proposition~2.11(ii)]{T4} we get that
$\operatorname{z}^1_{\CATB}$ belongs to $\SETWBsat$. Therefore, by axiom (BF2)
for $(\CATB,\SETWBsat)$ (see~\cite[Lemma~2.8]{T4}),
$\operatorname{w}^2_{\CATB}\circ\operatorname{z}^1_{\CATB}$
belongs to $\SETWBsat$. So by definition of $\SETWBsat$ there are an object $A'_{\CATB}$ and a morphism
$\operatorname{t}_{\CATB}:A'_{\CATB}\rightarrow\widetilde{A}_{\CATB}$, such that
$\operatorname{v}^1_{\CATB}:=(\operatorname{w}^2_{\CATB}\circ\operatorname{z}^1_{\CATB})\circ
\operatorname{t}_{\CATB}$ belongs to $\SETWB$. Since $\operatorname{w}^2_{\CATB}\circ
\operatorname{z}^1_{\CATB}$ belong to $\SETWBsat$, then by~\cite[Proposition~2.11(ii)]{T4} we get that
$\operatorname{t}_{\CATB}$ belongs to $\SETWBsat$.\\

Using (BF5) for $(\CATB,\SETWBsat)$ (see again~\cite[Lemma~2.8]{T4})
and $(\gamma^1_{\CATB})^{-1}$, we get that $\functor{F}_1(\operatorname{w}_{\CATA})\circ
\operatorname{z}^2_{\CATB}$ belongs to $\SETWBsat$. In all this section we are assuming that
$\functor{F}_1(\SETWA)\subseteq\SETWBsat$, so $\functor{F}_1(\operatorname{w}_{\CATA})$ belongs
to $\SETWBsat$. Hence by~\cite[Proposition~2.11(ii)]{T4} we get that
$\operatorname{z}^2_{\CATB}$ belongs to $\SETWBsat$. So by (BF2) also the morphism
$\operatorname{v}^2_{\CATB}:=\operatorname{z}^2_{\CATB}\circ\operatorname{t}_{\CATB}$ belongs to
$\SETWBsat$. Then we set

\begin{gather*}
\alpha_{\CATB}:=\thetab{\functor{F}_1(f_{\CATA})}{\operatorname{z}^2_{\CATB}}
 {\operatorname{t}_{\CATB}}\odot\Big(\gamma_{\CATB}^2\ast i_{\operatorname{t}_{\CATB}}\Big)\odot
 \Big(\thetaa{f_{\CATB}}{\operatorname{w}^2_{\CATB}}{\operatorname{z}^1_{\CATB}}
 \ast i_{\operatorname{t}_{\CATB}}\Big)\odot\thetaa{f_{\CATB}}{\operatorname{w}^2_{\CATB}\circ
 \operatorname{z}^1_{\CATB}}{\operatorname{t}_{\CATB}}:\\
\phantom{\Big(} f_{\CATB}\circ\operatorname{v}^1_{\CATB}\Longrightarrow\functor{F}_1(f_{\CATA})\circ
 \operatorname{v}^2_{\CATB}.
\end{gather*}

This suffices to conclude that (\hyperref[A3]{A3}) holds for $\functor{F}$.
\end{proof}

\begin{lem}\label{lem-10}
If $\widetilde{\functor{G}}$ satisfies \emphatic{(\hyperref[X2b]{X2b})}, then $\functor{F}$
satisfies \emphatic{(\hyperref[A4]{A4})}.
\end{lem}

\begin{proof}
Let us fix any set of data $(A_{\CATA},B_{\CATA},f^1_{\CATA},f^2_{\CATA},\gamma^1_{\CATA},
\gamma^2_{\CATA},A'_{\CATB},\operatorname{z}_{\CATB})$ as in (\hyperref[A4]{A4}) and let us suppose
that $\functor{F}_2(\gamma^1_{\CATA})\ast i_{\operatorname{z}_{\CATB}}=\functor{F}_2(\gamma^2_{\CATA})
\ast i_{\operatorname{z}_{\CATB}}$. For each $m=1,2$, we consider the $2$-morphism in $\CATA\left[
\SETWAinv\right]$

\begin{gather*}
\Gamma^m_{\CATA}:=\Big[A_{\CATA},\id_{A_{\CATA}},\id_{A_{\CATA}},
 i_{\id_{A_{\CATA}}\circ\id_{A_{\CATA}}},\gamma^m_{\CATA}\ast i_{\id_{A_{\CATA}}}\Big]: \\
\Big(A_{\CATA},\id_{A_{\CATA}},f^1_{\CATA}\Big)\Longrightarrow\Big(A_{\CATA},\id_{A_{\CATA}},
 f^2_{\CATA}\Big).
\end{gather*}

By the description of $\widetilde{\functor{G}}_2$ in Theorem~\ref{theo-01} and the description
of $2$-morphisms in a bicategory of fractions, see~\cite[\S~2.3]{Pr}, for each $m=1,2$ we have

\begin{gather*}
\widetilde{\functor{G}}_2(\Gamma^m_{\CATA})=\Big[\functor{F}_0(A_{\CATA}),\functor{F}_1
 (\id_{A_{\CATA}}),\functor{F}_1(\id_{A_{\CATA}}), \\
\psi^{\functor{F}}_{\id_{A_{\CATA}},\id_{A_{\CATA}}}\odot\functor{F}_2\Big(i_{\id_{A_{\CATA}}\circ
 \id_{A_{\CATA}}}\Big)\odot\Big(\psi^{\functor{F}}_{\id_{A_{\CATA}},\id_{A_{\CATA}}}\Big)^{-1},\\
\psi^{\functor{F}}_{f^2_{\CATA},\id_{A_{\CATA}}}\odot\functor{F}_2\Big(\gamma^m_{\CATA}\ast
 i_{\id_{A_{\CATA}}}\Big)\odot\Big(\psi^{\functor{F}}_{f^1_{\CATA},\id_{A_{\CATA}}}\Big)^{-1}\Big]= \\
=\Big[\functor{F}_0(A_{\CATA}),\functor{F}_1(\id_{A_{\CATA}}),\functor{F}_1(\id_{A_{\CATA}}),
 i_{\functor{F}_1(\id_{A_{\CATA}})\circ\functor{F}_1(\id_{A_{\CATA}})},\functor{F}_2
 (\gamma^m_{\CATA})\ast i_{\functor{F}_1(\id_{A_{\CATA}})}\Big]= \\
=\Big[A'_{\CATB},\operatorname{z}_{\CATB},\operatorname{z}_{\CATB},i_{\functor{F}_1(\id_{A_{\CATA}})
 \circ\operatorname{z}_{\CATB}},\functor{F}_2(\gamma^m_{\CATA})\ast i_{\operatorname{z}_{\CATB}}\Big].
\end{gather*}

Since $\functor{F}_2(\gamma^1_{\CATA})\ast i_{\operatorname{z}_{\CATB}}=\functor{F}_2
(\gamma^2_{\CATA})\ast i_{\operatorname{z}_{\CATB}}$, then we conclude that
$\widetilde{\functor{G}}_2(\Gamma^1_{\CATA})=\widetilde{\functor{G}}_2(\Gamma^2_{\CATA})$.
By (\hyperref[X2b]{X2b}) for $\widetilde{\functor{G}}$, we get that $\Gamma^1_{\CATA}=
\Gamma^2_{\CATA}$. Then by~\cite[Proposition~0.7]{T3} for the pair $(\CATA,\SETWA)$, there are an
object $A'_{\CATA}$ and a morphism $\operatorname{z}_{\CATA}:A'_{\CATA}\rightarrow A_{\CATA}$ in
$\SETWA$, such that

\[\Big(\gamma^1\ast i_{\id_{A_{\CATA}}}\Big)\ast i_{\operatorname{z}_{\CATA}}=
\Big(\gamma^2\ast i_{\id_{A_{\CATA}}}\Big)\ast i_{\operatorname{z}_{\CATA}}.\]

Then the claim follows immediately.
\end{proof}

\begin{lem}\label{lem-06}
If $\widetilde{\functor{G}}$ satisfies \emphatic{(\hyperref[X2c]{X2c})}, then
$\functor{F}$ satisfies \emphatic{(\hyperref[A5]{A5})}.
\end{lem}

\begin{proof}
Let us fix any set of data $(A_{\CATA},B_{\CATA},A_{\CATB},f^1_{\CATA},f^2_{\CATA},
\operatorname{v}_{\CATB},\alpha_{\CATB})$ as in (\hyperref[A5]{A5}). Then let us consider the
$2$-morphism in $\CATB\left[\SETWBsatinv\right]$

\begin{gather}
\nonumber \Gamma_{\CATB}:=\Big[A_{\CATB},\operatorname{v}_{\CATB},\operatorname{v}_{\CATB},
 i_{\functor{F}_1(\id_{A_{\CATA}})\circ\operatorname{v}_{\CATB}},\alpha_{\CATB}\Big]:\,
 \Big(\functor{F}_0(A_{\CATA}),\functor{F}_1(\id_{A_{\CATA}}),\functor{F}_1(f^1_{\CATA})\Big)
 \Longrightarrow \\
\label{eq-14} \Longrightarrow\Big(\functor{F}_0(A_{\CATA}),\functor{F}_1(\id_{A_{\CATA}}),
 \functor{F}_1(f^2_{\CATA})\Big).
\end{gather}

We can interpret $\Gamma_{\CATB}$ as a $2$-morphism from $\widetilde{\functor{G}}_1(A_{\CATA},
\id_{A_{\CATA}},f^1_{\CATA})$ to $\widetilde{\functor{G}}_1(A_{\CATA},\id_{A_{\CATA}},f^2_{\CATA})$,
so by (\hyperref[X2c]{X2c}) there is a $2$-morphism

\[\Gamma_{\CATA}:\Big(A_{\CATA},\id_{A_{\CATA}},f^1_{\CATA}\Big)\Longrightarrow\Big(A_{\CATA},
\id_{A_{\CATA}},f^2_{\CATA}\Big)\]
in $\CATA\left[\SETWAinv\right]$, such that $\widetilde{\functor{G}}_2(\Gamma_{\CATA})=
\Gamma_{\CATB}$. We apply~\cite[Lemma~6.1]{T3} for $(\CATA,\SETWA)$,
$\alpha:=i_{\id_{A_{\CATA}}\circ\id_{A_{\CATA}}}$ and $\Gamma:=\Gamma_{\CATA}$; then there are
an object $A'_{\CATA}$, a
morphism $\operatorname{v}_{\CATA}:A'_{\CATA}\rightarrow A_{\CATA}$ in $\SETWA$ and a $2$-morphism
$\gamma_{\CATA}:f^1_{\CATA}\circ(\id_{A_{\CATA}}\circ\operatorname{v}_{\CATA})\Rightarrow
f^2_{\CATA}\circ(\id_{A_{\CATA}}\circ\operatorname{v}_{\CATA})$, such that $\Gamma_{\CATA}$ is
represented by the data in the following diagram

\begin{equation}\label{eq-12}
\begin{tikzpicture}[xscale=2.6,yscale=-0.8]
    \node (A0_2) at (2, 0) {$A_{\CATA}$};
    \node (A2_2) at (2, 2) {$A'_{\CATA}$};
    \node (A2_0) at (0, 2) {$A_{\CATA}$};
    \node (A2_4) at (4, 2) {$B_{\CATA}$,};
    \node (A4_2) at (2, 4) {$A_{\CATA}$};
    
    \node (A2_3) at (2.8, 2) {$\Downarrow\,\gamma_{\CATA}$};
    \node (A2_1) at (1.2, 2) {$\Downarrow\,\varepsilon_{\CATA}$};
    
    \path (A4_2) edge [->]node [auto,swap] {$\scriptstyle{f^2_{\CATA}}$} (A2_4);
    \path (A0_2) edge [->]node [auto] {$\scriptstyle{f^1_{\CATA}}$} (A2_4);
    \path (A2_2) edge [->]node [auto,swap] {$\scriptstyle{\id_{A_{\CATA}}
      \circ\operatorname{v}_{\CATA}}$} (A0_2);
    \path (A2_2) edge [->]node [auto] {$\scriptstyle{\id_{A_{\CATA}}
      \circ\operatorname{v}_{\CATA}}$} (A4_2);
    \path (A4_2) edge [->]node [auto] {$\scriptstyle{\id_{A_{\CATA}}}$} (A2_0);
    \path (A0_2) edge [->]node [auto,swap] {$\scriptstyle{\id_{A_{\CATA}}}$} (A2_0);
\end{tikzpicture}
\end{equation}
where

\[\varepsilon_{\CATA}:=\thetab{\id_{A_{\CATA}}}{\id_{A_{\CATA}}}
{\operatorname{v}_{\CATA}}\odot\Big(i_{\id_{A_{\CATA}}\circ\id_{A_{\CATA}}}\ast
i_{\operatorname{v}_{\CATA}}\Big)\odot\thetaa{\id_{A_{\CATA}}}{\id_{A_{\CATA}}}
{\operatorname{v}_{\CATA}}=i_{(\id_{A_{\CATA}}\circ\id_{A_{\CATA}})\circ\operatorname{v}_{\CATA}}.\]

If we denote by $\upsilon_{\operatorname{v}_{\CATA}}$ the unitor $\id_{A_{\CATA}}\circ
\operatorname{v}_{\CATA}\Rightarrow\operatorname{v}_{\CATA}$, then we can define

\[\alpha_{\CATA}:=\Big(i_{f^2_{\CATA}}\ast\upsilon_{\operatorname{v}_{\CATA}}\Big)\odot
\gamma_{\CATA}\odot\Big(i_{f^1_{\CATA}}\ast\upsilon^{-1}_{\operatorname{v}_{\CATA}}\Big):\,\,
f^1_{\CATA}\circ\operatorname{v}_{\CATA}\Longrightarrow f^2_{\CATA}\circ\operatorname{v}_{\CATA}\]
and we get easily that

\[\Gamma_{\CATA}=\Big[A'_{\CATA},\operatorname{v}_{\CATA},\operatorname{v}_{\CATA},
i_{\id_{A_{\CATA}}\circ\operatorname{v}_{\CATA}},\alpha_{\CATA}\Big].\]

So by \eqref{eq-07} we have

\begin{gather}
\nonumber \widetilde{\functor{G}}_2(\Gamma_{\CATA})=\Big[\functor{F}_0(A'_{\CATA}),\functor{F}_1
 (\operatorname{v}_{\CATA}),\functor{F}_1(\operatorname{v}_{\CATA}), \\
\nonumber\psi^{\functor{F}}_{\id_{A_{\CATA}},\operatorname{v}_{\CATA}}\odot\functor{F}_2
 (i_{\id_{A_{\CATA}}\circ\operatorname{v}_{\CATA}})\odot\Big(\psi^{\functor{F}}_{\id_{A_{\CATA}},
 \operatorname{v}_{\CATA}}\Big)^{-1}, \\
\nonumber \psi^{\functor{F}}_{f^2_{\CATA},\operatorname{v}_{\CATA}}\odot\functor{F}_2
 (\alpha_{\CATA})\odot\Big(\psi^{\functor{F}}_{f^1_{\CATA},\operatorname{v}_{\CATA}}\Big)^{-1}
 \Big]= \\
\nonumber =\Big[\functor{F}_0(A'_{\CATA}),\functor{F}_1(\operatorname{v}_{\CATA}),
 \functor{F}_1(\operatorname{v}_{\CATA}),i_{\functor{F}_1(\id_{A_{\CATA}})\circ\functor{F}_1
 (\operatorname{v}_{\CATA})}, \\
\label{eq-15}\psi^{\functor{F}}_{f^2_{\CATA},\operatorname{v}_{\CATA}}\odot\functor{F}_2
 (\alpha_{\CATA})\odot\Big(\psi^{\functor{F}}_{f^1_{\CATA},\operatorname{v}_{\CATA}}\Big)^{-1}
 \Big].
\end{gather}

Now by (BF3) for $(\CATB,\SETWB)$ there is a set of data as in the upper part
of the following diagram

\[
\begin{tikzpicture}[xscale=-2.6,yscale=-0.8]
    \node (A0_4) at (4, 0) {$\widetilde{A}_{\CATB}$};
    \node (A2_2) at (3, 2) {$A_{\CATB}$,};
    \node (A2_4) at (4, 2) {$\functor{F}_0(A_{\CATA})$};
    \node (A2_6) at (5, 2) {$\functor{F}_0(A'_{\CATA})$};

    \node (A1_4) at (4, 1.4) {$\Rightarrow$};
    \node (B1_4) at (4, 1) {$\rho_{\CATB}$};
    
    \path (A2_6) edge [->]node [auto,swap] {$\scriptstyle{\functor{F}_1
      (\operatorname{v}_{\CATA})}$} (A2_4);
    \path (A0_4) edge [->]node [auto] {$\scriptstyle{\operatorname{s}_{\CATB}}$} (A2_2);
    \path (A2_2) edge [->]node [auto] {$\scriptstyle{\operatorname{v}_{\CATB}}$} (A2_4);
    \path (A0_4) edge [->]node [auto,swap] {$\scriptstyle{\operatorname{r}_{\CATB}}$} (A2_6);
\end{tikzpicture}
\]
such that $\operatorname{r}_{\CATB}$ belongs to $\SETWB$ and $\rho_{\CATB}$ is invertible.
Then using \eqref{eq-15} and the description of $2$-morphisms in~\cite[\S~2.3]{Pr}, we get

\begin{gather}
\nonumber \widetilde{\functor{G}}_2(\Gamma_{\CATA})=\Big[\widetilde{A}_{\CATB},
 \operatorname{v}_{\CATB}\circ\operatorname{s}_{\CATB},\operatorname{v}_{\CATB}\circ
 \operatorname{s}_{\CATB},i_{\functor{F}_1(\id_{A_{\CATA}})\circ(\operatorname{v}_{\CATB}\circ
 \operatorname{s}_{\CATB})}, \\
\nonumber \Big(i_{\functor{F}_1(f^2_{\CATA})}\ast\rho_{\CATB}\Big)\odot\thetab{\functor{F}_1
 (f^2_{\CATA})}{\functor{F}_1(\operatorname{v}_{\CATA})}{\operatorname{r}_{\CATB}}\odot\Big(
 \psi^{\functor{F}}_{f^2_{\CATA},\operatorname{v}_{\CATA}}\ast i_{\operatorname{r}_{\CATB}}\Big)
 \odot\Big(\functor{F}_2(\alpha_{\CATA})\ast i_{\operatorname{r}_{\CATB}}\Big)\odot \\
\label{eq-16} \odot\Big(\psi^{\functor{F}}_{f^1_{\CATA},\operatorname{v}_{\CATA}}\ast
 i_{\operatorname{r}_{\CATB}}\Big)^{-1}\odot\thetaa{\functor{F}_1(f^1_{\CATA})}{\functor{F}_1
 (\operatorname{v}_{\CATA})}{\operatorname{r}_{\CATB}}\odot\Big(i_{\functor{F}_1(f^1_{\CATA})}\ast
 \rho_{\CATB}^{-1}\Big)\Big].
\end{gather}

Moreover, from \eqref{eq-14} we get

\begin{gather}
\nonumber \Gamma_{\CATB}=\Big[\widetilde{A}_{\CATB},\operatorname{v}_{\CATB}\circ
 \operatorname{s}_{\CATB},\operatorname{v}_{\CATB}\circ\operatorname{s}_{\CATB},i_{\functor{F}_1
 (\id_{A_{\CATA}})\circ(\operatorname{v}_{\CATB}\circ\operatorname{s}_{\CATB})}, \\
\label{eq-17} \thetab{\functor{F}_1(f^2_{\CATA})}{\operatorname{v}_{\CATB}}
 {\operatorname{s}_{\CATB}}\odot\Big(\alpha_{\CATB}\ast i_{\operatorname{s}_{\CATB}}\Big)
 \odot\thetaa{\functor{F}_1(f^1_{\CATA})}{\operatorname{v}_{\CATB}}{\operatorname{s}_{\CATB}}\Big].
\end{gather}

Since $\widetilde{\functor{G}}_2(\Gamma_{\CATA})=\Gamma_{\CATB}$, then
using~\cite[Proposition~0.7]{T3}
for $(\CATB,\SETWBsat)$ together with \eqref{eq-16} and \eqref{eq-17}, we get that there are
an object $A'_{\CATB}$ and a morphism $\operatorname{t}_{\CATB}:A'_{\CATB}\rightarrow
\widetilde{A}_{\CATB}$ in $\SETWBsat$, such that

\begin{gather}
\nonumber \Big(\thetab{\functor{F}_1(f^2_{\CATA})}{\operatorname{v}_{\CATB}}
 {\operatorname{s}_{\CATB}}\odot\Big(\alpha_{\CATB}\ast i_{\operatorname{s}_{\CATB}}\Big)
 \odot\thetaa{\functor{F}_1(f^1_{\CATA})}{\operatorname{v}_{\CATB}}
 {\operatorname{s}_{\CATB}}\Big)\ast i_{\operatorname{t}_{\CATB}}= \\
\nonumber =\Big(\Big(i_{\functor{F}_1(f^2_{\CATA})}\ast\rho_{\CATB}\Big)\odot
 \thetab{\functor{F}_1(f^2_{\CATA})}{\functor{F}_1(\operatorname{v}_{\CATA})}
 {\operatorname{r}_{\CATB}}\odot\Big(\psi^{\functor{F}}_{f^2_{\CATA},\operatorname{v}_{\CATA}}
 \ast i_{\operatorname{r}_{\CATB}}\Big)\odot\Big(\functor{F}_2(\alpha_{\CATA})\ast
 i_{\operatorname{r}_{\CATB}}\Big)\odot \\
\label{eq-18} \odot\Big(\psi^{\functor{F}}_{f^1_{\CATA},\operatorname{v}_{\CATA}}\ast
 i_{\operatorname{r}_{\CATB}}\Big)^{-1}\odot\thetaa{\functor{F}_1(f^1_{\CATA})}{\functor{F}_1
 (\operatorname{v}_{\CATA})}{\operatorname{r}_{\CATB}}\odot\Big(i_{\functor{F}_1(f^1_{\CATA})}\ast
 \rho_{\CATB}^{-1}\Big)\Big)\ast i_{\operatorname{t}_{\CATB}}.
\end{gather}

By definition of $\SETWBsat$, without loss of generality we can assume that
$\operatorname{t}_{\CATB}$ belongs to $\SETWB$ (and not only to $\SETWBsat$) and that it still
verifies \eqref{eq-18}. Then we define a morphism $\operatorname{z}_{\CATB}:=\operatorname{r}_{\CATB}
\circ\operatorname{t}_{\CATB}:A'_{\CATB}\rightarrow\functor{F}_0(A'_{\CATA})$ (this morphism belongs
to $\SETWB$ by axiom (BF2))
and a morphism $\operatorname{z}'_{\CATB}:=
\operatorname{s}_{\CATB}\circ\operatorname{t}_{\CATB}:A'_{\CATB}\rightarrow A_{\CATB}$. Moreover,
we define an invertible $2$-morphism

\[\sigma_{\CATB}:=\thetab{\operatorname{v}_{\CATB}}{\operatorname{s}_{\CATB}}
{\operatorname{t}_{\CATB}}\odot\Big(\rho_{\CATB}\ast i_{\operatorname{t}_{\CATB}}\Big)\odot
\thetaa{\functor{F}_1(\operatorname{v}_{\CATA})}{\operatorname{r}_{\CATB}}
{\operatorname{t}_{\CATB}}:\,\,\functor{F}_1(\operatorname{v}_{\CATA})\circ\operatorname{z}_{\CATB}
\Longrightarrow\operatorname{v}_{\CATB}\circ\operatorname{z}'_{\CATB}.\]

Then identity \eqref{eq-18} implies that $\alpha_{\CATB}\ast i_{\operatorname{z}'_{\CATB}}$
coincides with the composition \eqref{eq-13}, so (\hyperref[A5]{A5}) holds.
\end{proof}

\section{Sufficiency of conditions (A1) -- (A5)}
In this section we assume all the hypothesis and notations on $\CATA,\SETWA,\CATB,\SETWB$ and
$\functor{F}$ of the previous section and
we prove that conditions (\hyperref[A]{A}) for $\functor{F}$ imply conditions (\hyperref[X]{X})
for $\widetilde{\functor{G}}$.

\begin{lem}\label{lem-11}
If $\functor{F}$ satisfies \emphatic{(\hyperref[A1]{A1})}, then $\widetilde{\functor{G}}$
satisfies \emphatic{(\hyperref[X1]{X1})}.
\end{lem}

\begin{proof}
Using~\cite[Corollary~2.7]{T4} for $(\CATB,\SETWBsat)$, \cite[Proposition~2.11(i)]{T4} and the fact
that $\SETWB\subseteq\SETWBsat$, we get that the data of \eqref{eq-31} give an internal equivalence in
$\CATB\left[\SETWBsatinv\right]$ from $\functor{F}_0(A_{\CATA})=\widetilde{\functor{G}}_0(A_{\CATA})$
to $A_{\CATB}$, so condition (\hyperref[X1]{X1}) holds for $\widetilde{\functor{G}}$.
\end{proof}

\begin{lem}\label{lem-05}
If $\functor{F}$ satisfies \emphatic{(\hyperref[A2]{A2})} and
\emphatic{(\hyperref[A3]{A3})}, then $\widetilde{\functor{G}}$ satisfies
\emphatic{(\hyperref[X2a]{X2a})}.
\end{lem}

\begin{proof}
Let us fix any pair of objects $A_{\CATA},B_{\CATA}$ and any morphism

\[
\begin{tikzpicture}[xscale=2.4,yscale=-1.2]
    \node (A0_0) at (-0.25, 0) {$\underline{f}_{\CATB}:=\Big(\functor{F}_0(A_{\CATA})$};
    \node (A0_1) at (1, 0) {$A_{\CATB}$};
    \node (A0_2) at (2, 0) {$\functor{F}_0(B_{\CATA})\Big)$};
    
    \path (A0_1) edge [->]node [auto,swap] {$\scriptstyle{\operatorname{w}_{\CATB}}$} (A0_0);
    \path (A0_1) edge [->]node [auto] {$\scriptstyle{f_{\CATB}}$} (A0_2);
\end{tikzpicture}
\]
in $\CATB\left[\SETWBsatinv\right]$. We need to prove that there are a morphism
$\underline{f}_{\CATA}:A_{\CATA}\rightarrow B_{\CATA}$ in $\CATA\left[\SETWAinv\right]$ and an
invertible $2$-morphism $\Gamma_{\CATB}:\widetilde{\functor{G}}_1(\underline{f}_{\CATA})\Rightarrow
\underline{f}_{\CATB}$ in $\CATB\left[\SETWBsatinv\right]$.\\

By definition of morphisms in $\CATB\left[\SETWBsatinv\right]$, the morphism
$\operatorname{w}_{\CATB}$ belongs to $\SETWBsat$, so by definition of right saturated
there are an object $\widetilde{A}_{\CATB}$ and a morphism $\operatorname{w}'_{\CATB}:
\widetilde{A}_{\CATB}\rightarrow A_{\CATB}$, such that $\operatorname{w}_{\CATB}\circ
\operatorname{w}'_{\CATB}$ belongs to $\SETWB$. By (\hyperref[A3]{A3}) applied to $f_{\CATB}\circ
\operatorname{w}'_{\CATB}$, there are an object $A'_{\CATA}$, a morphism $f_{\CATA}:A'_{\CATA}
\rightarrow B_{\CATA}$ and data $(A'_{\CATB},\operatorname{v}^1_{\CATB},\operatorname{v}^2_{\CATB},
\alpha_{\CATB})$ as follows

\[
\begin{tikzpicture}[xscale=1.8,yscale=-0.6]
    \node (A0_2) at (2, 0) {$\widetilde{A}_{\CATB}$};
    \node (A2_2) at (0, 1) {$A'_{\CATB}$};
    \node (A2_4) at (4, 1) {$\functor{F}_0(B_{\CATA})$,};
    \node (A4_2) at (2, 2) {$\functor{F}_0(A'_{\CATA})$};
    
    \node (A2_3) at (2, 1) {$\Downarrow\,\alpha_{\CATB}$};
    
    \path (A4_2) edge [->,bend left=15]node [auto,swap]
      {$\scriptstyle{\functor{F}_1(f_{\CATA})}$} (A2_4);
    \path (A0_2) edge [->,bend right=15]node [auto]
      {$\scriptstyle{f_{\CATB}\circ\operatorname{w}'_{\CATB}}$} (A2_4);
    \path (A2_2) edge [->,bend right=15]node [auto]
      {$\scriptstyle{\operatorname{v}^1_{\CATB}}$} (A0_2);
    \path (A2_2) edge [->,bend left=15]node [auto,swap]
      {$\scriptstyle{\operatorname{v}^2_{\CATB}}$} (A4_2);
\end{tikzpicture}
\]
with $\operatorname{v}^1_{\CATB}$ in $\SETWB$, $\operatorname{v}^2_{\CATB}$ in $\SETWBsat$ and
$\alpha_{\CATB}$ invertible. By (BF2) for $(\CATB,\SETWB)$, we have that
$(\operatorname{w}_{\CATB}\circ\operatorname{w}'_{\CATB})\circ\operatorname{v}^1_{\CATB}$ belongs to
$\SETWB$. So we can apply (\hyperref[A2]{A2}) to the set of data

\[
\begin{tikzpicture}[xscale=3.4,yscale=-1.2]
    \node (A0_0) at (0, 0) {$\functor{F}_0(A_{\CATA})$};
    \node (A0_1) at (1, 0) {$A'_{\CATB}$};
    \node (A0_2) at (2, 0) {$\functor{F}_0(A'_{\CATA})$.};
    
    \path (A0_1) edge [->]node [auto,swap] {$\scriptstyle{(\operatorname{w}_{\CATB}
      \circ\operatorname{w}'_{\CATB})\circ\operatorname{v}^1_{\CATB}}$} (A0_0);
    \path (A0_1) edge [->]node [auto] {$\scriptstyle{\operatorname{v}^2_{\CATB}}$} (A0_2);
\end{tikzpicture}
\]

Therefore there are an object $A^3_{\CATA}$, a pair of morphisms $\operatorname{w}_{\CATA}$ in
$\SETWA$ and $\operatorname{w}'_{\CATA}$ in $\SETWAsat$ as follows

\[
\begin{tikzpicture}[xscale=2.4,yscale=-1.2]
    \node (A0_0) at (0, 0) {$A_{\CATA}$};
    \node (A0_1) at (1, 0) {$A^3_{\CATA}$};
    \node (A0_2) at (2, 0) {$A'_{\CATA}$};
    
    \path (A0_1) edge [->]node [auto,swap] {$\scriptstyle{\operatorname{w}_{\CATA}}$} (A0_0);
    \path (A0_1) edge [->]node [auto] {$\scriptstyle{\operatorname{w}'_{\CATA}}$} (A0_2);
\end{tikzpicture}
\]
and a set of data in $\CATB$ as in the internal part of the following diagram

\[
\begin{tikzpicture}[xscale=2.6,yscale=-0.8]
    \node (A0_2) at (2, 0) {$A'_{\CATB}$};
    \node (A2_2) at (2, 2) {$\overline{A}_{\CATB}$};
    \node (A2_0) at (0, 2) {$\functor{F}_0(A_{\CATA})$};
    \node (A2_4) at (4, 2) {$\functor{F}_0(A'_{\CATA})$,};
    \node (A4_2) at (2, 4) {$\functor{F}_0(A^3_{\CATA})$};

    \node (A2_3) at (2.8, 2) {$\Downarrow\,\gamma^2_{\CATB}$};
    \node (A2_1) at (1.2, 2) {$\Downarrow\,\gamma^1_{\CATB}$};
    
    \path (A4_2) edge [->]node [auto,swap]
      {$\scriptstyle{\functor{F}_1(\operatorname{w}'_{\CATA})}$} (A2_4);
    \path (A0_2) edge [->]node [auto] {$\scriptstyle{\operatorname{v}^2_{\CATB}}$} (A2_4);
    \path (A2_2) edge [->]node [auto,swap] {$\scriptstyle{\operatorname{z}^1_{\CATB}}$} (A0_2);
    \path (A2_2) edge [->]node [auto] {$\scriptstyle{\operatorname{z}^2_{\CATB}}$} (A4_2);
    \path (A4_2) edge [->]node [auto]
      {$\scriptstyle{\functor{F}_1(\operatorname{w}_{\CATA})}$} (A2_0);
    \path (A0_2) edge [->]node [auto,swap] {$\scriptstyle{(\operatorname{w}_{\CATB}
      \circ\operatorname{w}'_{\CATB})\circ\operatorname{v}^1_{\CATB}}$} (A2_0);
\end{tikzpicture}
\]
with $\operatorname{z}^1_{\CATB}$ in $\SETWB$ and both $\gamma^1_{\CATB}$ and $\gamma^2_{\CATB}$
invertible. Then we define a pair of invertible $2$-morphisms in $\CATB$

\begin{gather*}
\rho^1_{\CATB}:=\thetab{\operatorname{w}_{\CATB}}{\operatorname{w}'_{\CATB}}
 {\operatorname{v}^1_{\CATB}\circ\operatorname{z}^1_{\CATB}}\odot\thetab{\operatorname{w}_{\CATB}
 \circ\operatorname{w}'_{\CATB}}{\operatorname{v}^1_{\CATB}}{\operatorname{z}^1_{\CATB}}\odot
 \left(\gamma^1_{\CATB}\right)^{-1}: \\ 
\phantom{\Big(}\functor{F}_1(\operatorname{w}_{\CATA})\circ\operatorname{z}^2_{\CATB}\Longrightarrow
 \operatorname{w}_{\CATB}\circ(\operatorname{w}'_{\CATB}\circ(\operatorname{v}^1_{\CATB}\circ
 \operatorname{z}^1_{\CATB}))
\end{gather*}
and

\begin{gather*}
\rho^2_{\CATB}:=\thetab{f_{\CATB}}{\operatorname{w}'_{\CATB}}{\operatorname{v}^1_{\CATB}\circ
 \operatorname{z}^1_{\CATB}}\odot\thetab{f_{\CATB}\circ\operatorname{w}'_{\CATB}}
 {\operatorname{v}^1_{\CATB}}{\operatorname{z}^1_{\CATB}}\odot\Big(\alpha_{\CATB}^{-1}\ast
 i_{\operatorname{z}^1_{\CATB}}\Big)\odot \\
\odot\,\thetaa{\functor{F}_1(f_{\CATA})}{\operatorname{v}^2_{\CATB}}{\operatorname{z}^1_{\CATB}}\odot
 \Big(i_{\functor{F}_1(f_{\CATA})}\ast\left(\gamma^2_{\CATB}\right)^{-1}\Big)\odot
 \thetab{\functor{F}_1(f_{\CATA})}{\functor{F}_1(\operatorname{w}'_{\CATA})}
 {\operatorname{z}^2_{\CATB}}\odot\Big(\psi^{\functor{F}}_{f_{\CATA},\operatorname{w}'_{\CATA}}
 \ast i_{\operatorname{z}^2_{\CATB}}\Big): \\
\phantom{\Big(}\functor{F}_1(f_{\CATA}\circ\operatorname{w}'_{\CATA})\circ\operatorname{z}^2_{\CATB}
 \Longrightarrow f_{\CATB}\circ(\operatorname{w}'_{\CATB}\circ(\operatorname{v}^1_{\CATB}\circ
 \operatorname{z}^1_{\CATB})).
\end{gather*}

Then the following $2$-morphism is invertible in $\CATB\left[\SETWBsatinv\right]$

\begin{gather*}
\Gamma_{\CATB}:=\Big[\overline{A}_{\CATB},\operatorname{z}^2_{\CATB},\operatorname{w}'_{\CATB}\circ
 (\operatorname{v}^1_{\CATB}\circ\operatorname{z}^1_{\CATB}),\rho^1_{\CATB},\rho^2_{\CATB}
 \Big]: \\
\widetilde{\functor{G}}_1\Big(A^3_{\CATA},\operatorname{w}_{\CATA},f_{\CATA}\circ
 \operatorname{w}'_{\CATA}\Big)=\Big(\functor{F}_0(A^3_{\CATA}),\functor{F}_1
 (\operatorname{w}_{\CATA}),\functor{F}_1(f_{\CATA}\circ\operatorname{w}'_{\CATA})\Big)
 \Longrightarrow \underline{f}_{\CATB}.
\end{gather*}

This suffices to conclude that (\hyperref[X2a]{X2a}) holds for $\widetilde{\functor{G}}$.
\end{proof}

\begin{lem}\label{lem-09}
Let us suppose that $\functor{F}$ satisfies \emphatic{(\hyperref[A4]{A4})}. Let us fix any
pair of objects $A_{\CATA},
B_{\CATA}$ and any pair of morphisms $(A^m_{\CATA},\operatorname{w}^m_{\CATA},f^m_{\CATA}):A_{\CATA}
\rightarrow B_{\CATA}$ for $m=1,2$ in $\CATA\left[\SETWAinv\right]$. Moreover, let us fix any pair of
$2$-morphisms

\[\Gamma_{\CATA}^m:\Big(A^1_{\CATA},\operatorname{w}^1_{\CATA},f^1_{\CATA}\Big)\Longrightarrow
\Big(A^2_{\CATA},\operatorname{w}^2_{\CATA},f^2_{\CATA}\Big)\quad\textrm{for}\quad m=1,2\]
in $\CATA\left[\SETWAinv\right]$ and let us suppose that $\widetilde{\functor{G}}_2
(\Gamma_{\CATA}^1)=\widetilde{\functor{G}}_2(\Gamma_{\CATA}^2)$. Then $\Gamma_{\CATA}^1=
\Gamma_{\CATA}^2$, i.e.\ $\widetilde{\functor{G}}$ satisfies condition
\emphatic{(\hyperref[X2b]{X2b})}.
\end{lem}

\begin{proof}
By~\cite[Proposition~0.7]{T3} for the pair $(\CATA,\SETWA)$, there are an object
$A^3_{\CATA}$, a pair of morphisms $\operatorname{v}^1_{\CATA}:A^3_{\CATA}\rightarrow A^1_{\CATA}$
and $\operatorname{v}^2_{\CATA}:A^3_{\CATA}\rightarrow A^2_{\CATA}$, an
invertible $2$-morphism $\alpha_{\CATA}:\operatorname{w}^1_{\CATA}\circ\operatorname{v}^1_{\CATA}
\Rightarrow\operatorname{w}^2_{\CATA}\circ\operatorname{v}^2_{\CATA}$ and a pair of $2$-morphisms
$\gamma^m_{\CATA}:f^1_{\CATA}\circ\operatorname{v}^1_{\CATA}\Rightarrow f^2_{\CATA}\circ
\operatorname{v}^2_{\CATA}$ for $m=1,2$, such that

\begin{equation}\label{eq-05}
\Gamma_{\CATA}^m=\Big[A^3_{\CATA},\operatorname{v}^1_{\CATA},\operatorname{v}^2_{\CATA},
\alpha_{\CATA},\gamma^m_{\CATA}\Big]\quad\textrm{for}\quad m=1,2.
\end{equation}

Then we have the following identity in $\CATB\left[\SETWBsatinv\right]$:

\begin{gather*}
\Big[\functor{F}_0(A^3_{\CATA}),\functor{F}_1(\operatorname{v}^1_{\CATA}),\functor{F}_1
 (\operatorname{v}^2_{\CATA}),\psi^{\functor{F}}_{\operatorname{w}^2_{\CATA},
 \operatorname{v}^2_{\CATA}}\odot\functor{F}_2(\alpha_{\CATA})\odot\Big(
 \psi^{\functor{F}}_{\operatorname{w}^1_{\CATA},
 \operatorname{v}^1_{\CATA}}\Big)^{-1}, \\
\psi^{\functor{F}}_{f^2_{\CATA},\operatorname{v}^2_{\CATA}}\odot\functor{F}_2(\gamma^1_{\CATA})\odot
 \Big(\psi^{\functor{F}}_{f^1_{\CATA},\operatorname{v}^1_{\CATA}}\Big)^{-1}\Big]
 \stackrel{\eqref{eq-07}}{=}\widetilde{\functor{G}}_2(\Gamma^1_{\CATA})=\widetilde{\functor{G}}_2
 (\Gamma^2_{\CATA})\stackrel{\eqref{eq-07}}{=} \\
=\Big[\functor{F}_0(A^3_{\CATA}),\functor{F}_1(\operatorname{v}^1_{\CATA}),\functor{F}_1
 (\operatorname{v}^2_{\CATA}),\psi^{\functor{F}}_{\operatorname{w}^2_{\CATA},
 \operatorname{v}^2_{\CATA}}\odot\functor{F}_2(\alpha_{\CATA})\odot\Big(
 \psi^{\functor{F}}_{\operatorname{w}^1_{\CATA},
 \operatorname{v}^1_{\CATA}}\Big)^{-1}, \\
\psi^{\functor{F}}_{f^2_{\CATA},\operatorname{v}^2_{\CATA}}\odot\functor{F}_2(\gamma^2_{\CATA})\odot
 \Big(\psi^{\functor{F}}_{f^1_{\CATA},\operatorname{v}^1_{\CATA}}\Big)^{-1}\Big].
\end{gather*}

Then by~\cite[Proposition~0.7]{T3} for $(\CATB,\SETWBsat)$, there are an object $A^1_{\CATB}$ and a
morphism $\operatorname{z}_{\CATB}^1:A^1_{\CATB}\rightarrow\functor{F}_0(A^3_{\CATA})$ in
$\SETWBsat$, such that

\begin{gather}
\nonumber \Big(\psi^{\functor{F}}_{f^2_{\CATA},\operatorname{v}^2_{\CATA}}\odot\functor{F}_2
 (\gamma^1_{\CATA})\odot\Big(\psi^{\functor{F}}_{f^1_{\CATA},\operatorname{v}^1_{\CATA}}\Big)^{-1}
 \Big)\ast i_{\operatorname{z}^1_{\CATB}}= \\
\label{eq-19} =\Big(\psi^{\functor{F}}_{f^2_{\CATA},\operatorname{v}^2_{\CATA}}\odot\functor{F}_2
 (\gamma^2_{\CATA})\odot\Big(\psi^{\functor{F}}_{f^1_{\CATA},\operatorname{v}^1_{\CATA}}\Big)^{-1}
 \Big)\ast i_{\operatorname{z}^1_{\CATB}}.
\end{gather}

By definition of $\SETWBsat$, there are an object $A^2_{\CATB}$ and a morphism
$\operatorname{z}_{\CATB}^2:A^2_{\CATB}\rightarrow A^1_{\CATB}$, such that
$\operatorname{z}^1_{\CATB}\circ\operatorname{z}^2_{\CATB}$ belongs to $\SETWB$. Then from
\eqref{eq-19} we get easily

\[\functor{F}_2(\gamma^1_{\CATA})\ast i_{\operatorname{z}^1_{\CATB}\circ\operatorname{z}^2_{\CATB}}=
\functor{F}_2(\gamma^2_{\CATA})\ast i_{\operatorname{z}^1_{\CATB}\circ\operatorname{z}^2_{\CATB}}.\]

By (\hyperref[A4]{A4}) there are an object $A^4_{\CATA}$ and a morphism
$\operatorname{z}_{\CATA}:A^4_{\CATA}\rightarrow A^3_{\CATA}$ in $\SETWA$, such that
$\gamma^1_{\CATA}\ast i_{\operatorname{z}_{\CATA}}=\gamma^2_{\CATA}\ast
i_{\operatorname{z}_{\CATA}}$; using \eqref{eq-05}, this implies easily that $\Gamma^1_{\CATA}=
\Gamma^2_{\CATA}$.
\end{proof}

In order to prove that conditions (\hyperref[A]{A}) imply also condition (X2c), we need firstly to
prove a preliminary lemma as follows:

\begin{lem}\label{lem-03}
Let us suppose that \emphatic{(\hyperref[A4]{A4})} and \emphatic{(\hyperref[A5]{A5})} hold. Let
us fix any set of data $(A_{\CATA},B_{\CATA},A_{\CATB},f^1_{\CATA},f^2_{\CATA},
\operatorname{v}_{\CATB},
\alpha_{\CATB})$ as in \emphatic{(\hyperref[A5]{A5})} and let us suppose that $\alpha_{\CATB}$
is \emph{invertible}. Then there is a set of data $(A'_{\CATA},A'_{\CATB},\operatorname{v}_{\CATA},
\operatorname{z}_{\CATB},\operatorname{z}'_{\CATB},\alpha_{\CATA},\sigma_{\CATB})$ satisfying all
the conditions of \emphatic{(\hyperref[A5]{A5})} and such that $\alpha_{\CATA}$ is \emph{invertible}.
\end{lem}

\begin{proof}
Using (\hyperref[A5]{A5}), there are a pair of objects $\widetilde{A}'_{\CATA}$,
$\widetilde{A}'_{\CATB}$, a triple of morphisms $\widetilde{\operatorname{v}}_{\CATA}:
\widetilde{A}'_{\CATA}\rightarrow A_{\CATA}$ in $\SETWA$, $\widetilde{\operatorname{z}}_{\CATB}:
\widetilde{A}'_{\CATB}\rightarrow\functor{F}_0(\widetilde{A}'_{\CATA})$ in $\SETWB$ and
$\widetilde{\operatorname{z}}'_{\CATB}:\widetilde{A}'_{\CATB}\rightarrow A_{\CATB}$, a $2$-morphism

\[
\begin{tikzpicture}[xscale=1.8,yscale=-0.6]
    \node (B0_0) at (-1, 0) {}; 
    \node (B1_1) at (5, 0) {}; 
    
    \node (A0_1) at (2, 0) {$A_{\CATA}$};
    \node (A1_0) at (0, 1) {$\widetilde{A}'_{\CATA}$};
    \node (A1_2) at (4, 1) {$B_{\CATA}$};
    \node (A2_1) at (2, 2) {$A_{\CATA}$};
    
    \node (A1_1) at (2, 1) {$\Downarrow\,\widetilde{\alpha}_{\CATA}$};

    \path (A1_0) edge [->,bend right=20]node [auto]
      {$\scriptstyle{\widetilde{\operatorname{v}}_{\CATA}}$} (A0_1);
    \path (A1_0) edge [->,bend left=20]node [auto,swap]
      {$\scriptstyle{\widetilde{\operatorname{v}}_{\CATA}}$} (A2_1);
    \path (A0_1) edge [->,bend right=20]node [auto] {$\scriptstyle{f^1_{\CATA}}$} (A1_2);
    \path (A2_1) edge [->,bend left=20]node [auto,swap] {$\scriptstyle{f^2_{\CATA}}$} (A1_2);
\end{tikzpicture}
\]
and an invertible $2$-morphism
  
\[
\begin{tikzpicture}[xscale=1.8,yscale=-0.6]
    \node (B0_0) at (-1, 0) {}; 
    \node (B1_1) at (5, 0) {}; 
    
    \node (A0_1) at (2, 0) {$\functor{F}_0(\widetilde{A}'_{\CATA})$};
    \node (A1_0) at (0, 1) {$\widetilde{A}'_{\CATB}$};
    \node (A1_2) at (4, 1) {$\functor{F}_0(A_{\CATA})$,};
    \node (A2_1) at (2, 2) {$A_{\CATB}$};
    
    \node (A1_1) at (2, 1) {$\Downarrow\,\widetilde{\sigma}_{\CATB}$};

    \path (A1_0) edge [->,bend right=20]node [auto]
      {$\scriptstyle{\widetilde{\operatorname{z}}_{\CATB}}$} (A0_1);
    \path (A1_0) edge [->,bend left=20]node [auto,swap]
      {$\scriptstyle{\widetilde{\operatorname{z}}'_{\CATB}}$} (A2_1);
    \path (A0_1) edge [->,bend right=20]node [auto] {$\scriptstyle{\functor{F}_1
      (\widetilde{\operatorname{v}}_{\CATA})}$} (A1_2);
    \path (A2_1) edge [->,bend left=20]node [auto,swap]
      {$\scriptstyle{\operatorname{v}_{\CATB}}$} (A1_2);
\end{tikzpicture}
\]
such that $\alpha_{\CATB}\ast i_{\widetilde{\operatorname{z}}'_{\CATB}}$ coincides with the following
composition:

\begin{equation}\label{eq-20}
\begin{tikzpicture}[xscale=3.5,yscale=-1.8]
    \node (A2_0) at (0, 2) {$\widetilde{A}'_{\CATB}$};
    \node (A2_1) at (1, 2) {$\functor{F}_0(\widetilde{A}'_{\CATA})$};
    \node (A4_1) at (1.5, 3.5) {$A_{\CATB}$};
    \node (A4_2) at (1, 3) {$\functor{F}_0(A_{\CATA})$};
    \node (A0_1) at (1.5, 0.5) {$A_{\CATB}$};
    \node (A0_2) at (1, 1) {$\functor{F}_0(A_{\CATA})$};
    \node (A2_3) at (3, 2) {$\functor{F}_0(B_{\CATA})$.};
        
    \node (A1_0) at (2.2, 0.75) {$\Downarrow\,\thetab{\functor{F}_1(f^1_{\CATA})}
      {\operatorname{v}_{\CATB}}{\widetilde{\operatorname{z}}'_{\CATB}}$};
    \node (A1_1) at (0.35, 1.3) {$\Downarrow\,\widetilde{\sigma}_{\CATB}^{-1}$};
    \node (A1_3) at (1.5, 1.35) {$\Downarrow\,\thetaa{\functor{F}_1(f^1_{\CATA})}
      {\functor{F}_1(\widetilde{\operatorname{v}}_{\CATA})}{\widetilde{\operatorname{z}}_{\CATB}}$};
    \node (A2_2) at (2, 2) {$\Downarrow\,\psi^{\functor{F}}_{f^2_{\CATA},
      \widetilde{\operatorname{v}}_{\CATA}}\odot\functor{F}_2(\widetilde{\alpha}_{\CATA})\odot
      \left(\psi^{\functor{F}}_{f^1_{\CATA},\widetilde{\operatorname{v}}_{\CATA}}\right)^{-1}$};
    \node (A3_0) at (2.2, 3.25) {$\Downarrow\,\thetaa{\functor{F}_1(f^2_{\CATA})}
      {\operatorname{v}_{\CATB}}{\widetilde{\operatorname{z}}'_{\CATB}}$};
    \node (A3_1) at (0.35, 2.7) {$\Downarrow\,\widetilde{\sigma}_{\CATB}$};
    \node (A3_3) at (1.5, 2.65) {$\Downarrow\,\thetab{\functor{F}_1(f^2_{\CATA})}
      {\functor{F}_1(\widetilde{\operatorname{v}}_{\CATA})}{\widetilde{\operatorname{z}}_{\CATB}}$};
    
    \node (B1_1) at (0.8, 0.38) {$\scriptstyle{\widetilde{\operatorname{z}}'_{\CATB}}$};
    \node (B2_2) at (0.8, 3.62) {$\scriptstyle{\widetilde{\operatorname{z}}'_{\CATB}}$};
    \node (B3_3) at (2.2, 0.38) {$\scriptstyle{\functor{F}_1(f^1_{\CATA})
      \circ\operatorname{v}_{\CATB}}$};
    \node (B4_4) at (2.2, 3.62) {$\scriptstyle{\functor{F}_1(f^2_{\CATA})
      \circ\operatorname{v}_{\CATB}}$};
    \node (B5_5) at (1.5, 3.12) {$\scriptstyle{\functor{F}_1(f^2_{\CATA})}$};
    \node (B6_6) at (1.5, 0.88) {$\scriptstyle{\functor{F}_1(f^1_{\CATA})}$};
    \node (B7_7) at (0.5, 0.88) {$\scriptstyle{\operatorname{v}_{\CATB}
      \circ\widetilde{\operatorname{z}}'_{\CATB}}$};
    \node (B8_8) at (0.5, 3.12) {$\scriptstyle{\operatorname{v}_{\CATB}
      \circ\widetilde{\operatorname{z}}'_{\CATB}}$};
    \node (C1_1) at (2.4, 1.48) {$\scriptstyle{\functor{F}_1(f^1_{\CATA})
      \circ\functor{F}_1(\widetilde{\operatorname{v}}_{\CATA})}$};
    \node (C2_2) at (2.4, 2.52) {$\scriptstyle{\functor{F}_1(f^2_{\CATA})
      \circ\functor{F}_1(\widetilde{\operatorname{v}}_{\CATA})}$};

    \draw [->,rounded corners] (A2_0) to (0.2, 3.5) to (A4_1);
    \draw [->,rounded corners] (A2_0) to (0.2, 0.5) to (A0_1);
    \draw [->,rounded corners] (A0_1) to (2.8, 0.5) to (A2_3);
    \draw [->,rounded corners] (A4_1) to (2.8, 3.5) to (A2_3);
    \draw [->,rounded corners] (A0_2) to (2.8, 1) to (A2_3);
    \draw [->,rounded corners] (A4_2) to (2.8, 3) to (A2_3);
    \draw [->,rounded corners] (A2_0) to (0.2, 1) to (A0_2);
    \draw [->,rounded corners] (A2_0) to (0.2, 3) to (A4_2);
    \draw [->,rounded corners] (A2_1) to (1.2, 1.6) to (2.8, 1.6) to (A2_3);
    \draw [->,rounded corners] (A2_1) to (1.2, 2.4) to (2.8, 2.4) to (A2_3);
    
    \path (A2_0) edge [->]node [auto] {$\scriptstyle{\functor{F}_1
      (\widetilde{\operatorname{v}}_{\CATA})
      \circ\widetilde{\operatorname{z}}_{\CATB}}$} (A4_2);
    \path (A2_0) edge [->]node [auto] {$\scriptstyle{\widetilde{\operatorname{z}}_{\CATB}}$} (A2_1);
    \path (A2_0) edge [->]node [auto,swap] {$\scriptstyle{\functor{F}_1
      (\widetilde{\operatorname{v}}_{\CATA})\circ\widetilde{\operatorname{z}}_{\CATB}}$} (A0_2);
\end{tikzpicture}
\end{equation}

Now we define a $2$-morphism $\Gamma_{\CATA}:(A_{\CATA},\id_{A_{\CATA}},f^1_{\CATA})\Rightarrow
(A_{\CATA},\id_{A_{\CATA}},f^2_{\CATA})$ as the $2$-morphism represented by the following diagram:

\[
\begin{tikzpicture}[xscale=2.6,yscale=-0.8]
    \node (A0_2) at (2, 0) {$A_{\CATA}$};
    \node (A2_2) at (2, 2) {$\widetilde{A}'_{\CATA}$};
    \node (A2_0) at (0, 2) {$A_{\CATA}$};
    \node (A2_4) at (4, 2) {$B_{\CATA}$.};
    \node (A4_2) at (2, 4) {$A_{\CATA}$};
    
    \node (A2_3) at (2.8, 2) {$\Downarrow\,\widetilde{\alpha}_{\CATA}$};
    \node (A2_1) at (1.2, 2) {$\Downarrow\,i_{\id_{A_{\CATA}}
      \circ\widetilde{\operatorname{v}}_{\CATA}}$};
    
    \path (A4_2) edge [->]node [auto,swap] {$\scriptstyle{f^2_{\CATA}}$} (A2_4);
    \path (A0_2) edge [->]node [auto] {$\scriptstyle{f^1_{\CATA}}$} (A2_4);
    \path (A2_2) edge [->]node [auto,swap]
     {$\scriptstyle{\widetilde{\operatorname{v}}_{\CATA}}$} (A0_2);
    \path (A2_2) edge [->]node [auto] {$\scriptstyle{\widetilde{\operatorname{v}}_{\CATA}}$} (A4_2);
    \path (A4_2) edge [->]node [auto] {$\scriptstyle{\id_{A_{\CATA}}}$} (A2_0);
    \path (A0_2) edge [->]node [auto,swap] {$\scriptstyle{\id_{A_{\CATA}}}$} (A2_0);
\end{tikzpicture}
\]

Using \eqref{eq-07}, we get that $\widetilde{\functor{G}}_2(\Gamma_{\CATA})$ coincides with the class

\[\left[\functor{F}_0(\widetilde{A}'_{\CATA}),\functor{F}_1
(\widetilde{\operatorname{v}}_{\CATA}),\functor{F}_1(\widetilde{\operatorname{v}}_{\CATA}),
i_{\functor{F}_1(\id_{A_{\CATA}})\circ\functor{F}_1(\widetilde{\operatorname{v}}_{\CATA})},
\psi^{\functor{F}}_{f^2_{\CATA},\widetilde{\operatorname{v}}_{\CATA}}\odot\functor{F}_2
(\widetilde{\alpha}_{\CATA})\odot\Big(\psi^{\functor{F}}_{f^1_{\CATA},
\widetilde{\operatorname{v}}_{\CATA}}\Big)^{-1}\right].\]

Now we consider the $2$-morphism in $\CATB\left[\SETWBsatinv\right]$ defined as follows

\begin{gather*}
\Gamma_{\CATB}:=\Big[A_{\CATB},\operatorname{v}_{\CATB},\operatorname{v}_{\CATB},
 i_{\functor{F}_1(\id_{A_{\CATA}})\circ\operatorname{v}_{\CATB}},\alpha_{\CATB}\Big]:\,
 \Big(\functor{F}_0(A_{\CATA}),\functor{F}_1(\id_{A_{\CATA}}),\functor{F}_1(f^1_{\CATA})\Big)
 \Longrightarrow \\
\Longrightarrow\Big(\functor{F}_0(A_{\CATA}),\functor{F}_1(\id_{A_{\CATA}}),
 \functor{F}_1(f^2_{\CATA})\Big);
\end{gather*}
using the definition of $2$-morphism in a bicategory of fractions (see~\cite[\S~2.3]{Pr}) together
with the following diagram

\[
\begin{tikzpicture}[xscale=2.4,yscale=-0.8]
    \node (A0_1) at (1, 0) {$\functor{F}_0(A_{\CATA})$};
    \node (A2_0) at (0, 2) {$A_{\CATB}$};
    \node (A2_1) at (1, 2) {$\widetilde{A}'_{\CATB}$};
    \node (A2_2) at (2, 2) {$\functor{F}_0(\widetilde{A}'_{\CATA})$};
    \node (A4_1) at (1, 4) {$\functor{F}_0(A_{\CATA})$};

    \node (A1_1) at (1, 1.3) {$\widetilde{\sigma}_{\CATB}^{-1}$};
    \node (B1_1) at (1, 0.8) {$\Rightarrow$};
    \node (A3_1) at (1, 2.6) {$\widetilde{\sigma}_{\CATB}$};
    \node (B3_1) at (1, 3.1) {$\Leftarrow$};
    
    \path (A2_1) edge [->]node [auto,swap]
      {$\scriptstyle{\widetilde{\operatorname{z}}'_{\CATB}}$} (A2_0);
    \path (A2_0) edge [->]node [auto,swap] {$\scriptstyle{\operatorname{v}_{\CATB}}$} (A4_1);
    \path (A2_2) edge [->]node [auto] {$\scriptstyle{\functor{F}_1
      (\widetilde{\operatorname{v}}_{\CATA})}$} (A4_1);
    \path (A2_0) edge [->]node [auto] {$\scriptstyle{\operatorname{v}_{\CATB}}$} (A0_1);
    \path (A2_2) edge [->]node [auto,swap] {$\scriptstyle{\functor{F}_1
      (\widetilde{\operatorname{v}}_{\CATA})}$} (A0_1);
    \path (A2_1) edge [->]node [auto] {$\scriptstyle{\widetilde{\operatorname{z}}_{\CATB}}$} (A2_2);
\end{tikzpicture}
\]
and \eqref{eq-20}, we get easily that $\widetilde{\functor{G}}_2(\Gamma_{\CATA})=
\Gamma_{\CATB}$.\\

Since $\alpha_{\CATB}$ is invertible by hypothesis, then it makes sense to consider the
inverse for $\Gamma_{\CATB}$, defined as follows:

\[\Gamma_{\CATB}^{-1}:=\Big[A_{\CATB},\operatorname{v}_{\CATB},\operatorname{v}_{\CATB},
 i_{\functor{F}_1(\id_{A_{\CATA}})\circ\operatorname{v}_{\CATB}},\alpha_{\CATB}^{-1}\Big].\]

Using again (\hyperref[A5]{A5}) on the set of data $(A_{\CATA},B_{\CATA},A_{\CATB},
f^2_{\CATA},f^1_{\CATA},
\operatorname{v}_{\CATB},\alpha_{\CATB}^{-1})$ and proceeding as above, we get a $2$-morphism
$\Gamma'_{\CATA}:(A_{\CATA},\id_{\CATA},f^2_{\CATA})\Rightarrow(A_{\CATA},\id_{\CATA},f^1_{\CATA})$
in $\CATA\left[\SETWAinv\right]$, such that $\widetilde{\functor{G}}_2(\Gamma'_{\CATA})=
\Gamma^{-1}_{\CATB}$. Then 

\[\widetilde{\functor{G}}_2(\Gamma'_{\CATA}\odot\Gamma_{\CATA})=\Gamma^{-1}_{\CATB}\odot
\Gamma_{\CATB}=i_{\widetilde{\functor{G}}_1(A_{\CATA},\id_{A_{\CATA}},f^1_{\CATA})}
\stackrel{\eqref{eq-07}}{=}
\widetilde{\functor{G}}_2\left(i_{(A_{\CATA},\id_{A_{\CATA}},f^1_{\CATA})}\right).\]

Since we are assuming condition (\hyperref[A4]{A4}), then by Lemma~\ref{lem-09} we conclude that
$\Gamma'_{\CATA}\odot\Gamma_{\CATA}=i_{(A_{\CATA},\id_{A_{\CATA}},f^1_{\CATA})}$. Analogously, we
get that $\Gamma_{\CATA}\odot\Gamma'_{\CATA}=i_{(A_{\CATA},\id_{A_{\CATA}},f^2_{\CATA})}$. This
proves that $\Gamma_{\CATA}$ is invertible. So by~\cite[Proposition~0.8(iii)]{T3} for
$(\CATA,\SETWA)$, there are an object
$A'_{\CATA}$ and a morphism $\operatorname{p}_{\CATA}:A'_{\CATA}\rightarrow\widetilde{A}'_{\CATA}$ in
$\SETWA$, such that $\widetilde{\alpha}_{\CATA}\ast i_{\operatorname{p}_{\CATA}}$ is invertible. 
Now we use axiom (BF3) for $(\CATB,\SETWB)$ in order to get data as in the upper
part of the following diagram

\[
\begin{tikzpicture}[xscale=3.3,yscale=-0.8]
    \node (A0_1) at (1, 0) {$A'_{\CATB}$};
    \node (A1_0) at (0, 2) {$\functor{F}_0(A'_{\CATA})$};
    \node (A1_2) at (2, 2) {$\widetilde{A}'_{\CATB}$,};
    \node (A2_1) at (1, 2) {$\functor{F}_0(\widetilde{A}'_{\CATA})$};
    
    \node (A1_1) at (1, 1) {$\mu_{\CATB}$};
    \node (B1_1) at (1, 1.4) {$\Rightarrow$};
    
    \path (A1_2) edge [->]node [auto]
      {$\scriptstyle{\widetilde{\operatorname{z}}_{\CATB}}$} (A2_1);
    \path (A0_1) edge [->]node [auto] {$\scriptstyle{\operatorname{q}_{\CATB}}$} (A1_2);
    \path (A1_0) edge [->]node [auto,swap] {$\scriptstyle{\functor{F}_1
      (\operatorname{p}_{\CATA})}$} (A2_1);
    \path (A0_1) edge [->]node [auto,swap]
      {$\scriptstyle{\operatorname{z}_{\CATB}}$} (A1_0);
\end{tikzpicture}
\]
with $\operatorname{z}_{\CATB}$ in $\SETWB$ and $\mu_{\CATB}$ invertible. Then we define
a morphism $\operatorname{v}_{\CATA}:=\widetilde{\operatorname{v}}_{\CATA}\circ
\operatorname{p}_{\CATA}:A'_{\CATA}\rightarrow A_{\CATA}$ (this morphism belongs to $\SETWA$ by
(BF2)) and a morphism $\operatorname{z}'_{\CATB}:=\widetilde{\operatorname{z}}'_{\CATB}\circ
\operatorname{q}_{\CATB}: A'_{\CATB}\rightarrow A_{\CATB}$. Moreover, we define a $2$-morphism

\[\alpha_{\CATA}:=\thetab{f^2_{\CATA}}{\widetilde{\operatorname{v}}_{\CATA}}{\operatorname{p}_{\CATA}}
\odot\Big(\widetilde{\alpha}_{\CATA}\ast i_{\operatorname{p}_{\CATA}}\Big)\odot
\thetaa{f^1_{\CATA}}{\widetilde{\operatorname{v}}_{\CATA}}{\operatorname{p}_{\CATA}}:\,\,
f^1_{\CATA}\circ\operatorname{v}_{\CATA}\Longrightarrow f^2_{\CATA}\circ\operatorname{v}_{\CATA}.\]

Such a $2$-morphism is invertible because $\widetilde{\alpha}_{\CATA}\ast 
i_{\operatorname{p}_{\CATA}}$
is invertible by construction. In addition, we define an invertible
$2$-morphism $\sigma_{\CATB}:\functor{F}_1(\operatorname{v}_{\CATA})\circ\operatorname{z}_{\CATB}
\Rightarrow\operatorname{v}_{\CATB}\circ\operatorname{z}'_{\CATB}$ as the following composition:

\[
\begin{tikzpicture}[xscale=3.9,yscale=-1.8]
    \node (A3_2) at (1.3, 3) {$\widetilde{A}'_{\CATB}$};
    \node (A4_2) at (1.8, 3.6) {$A_{\CATB}$};
    \node (A0_2) at (1.3, 0.9) {$\functor{F}_0(A'_{\CATA})$};
    \node (A2_0) at (0, 2) {$A'_{\CATB}$};
    \node (A2_1) at (1.3, 2) {$\functor{F}_0(\widetilde{A}'_{\CATA})$};
    \node (A2_3) at (3, 2) {$\functor{F}_0(A_{\CATA})$.};

    \node (A1_1) at (0.8, 1.4) {$\Downarrow\,\thetab{\functor{F}_1
      (\widetilde{\operatorname{v}}_{\CATA})}{\functor{F}_1(\operatorname{p}_{\CATA})}
      {\operatorname{z}_{\CATB}}$};
    \node (A1_2) at (2.2, 1.5) {$\Downarrow\,
      \psi^{\functor{F}}_{\widetilde{\operatorname{v}}_{\CATA},\operatorname{p}_{\CATA}}$};
    \node (A2_2) at (1.2, 2.5) {$\Downarrow\,\thetaa{\functor{F}_1
      (\widetilde{\operatorname{v}}_{\CATA})}
      {\widetilde{\operatorname{z}}_{\CATB}}{\operatorname{q}_{\CATB}}$};
    \node (A3_1) at (1.7, 3.25) {$\Downarrow\,\thetab{\operatorname{v}_{\CATB}}
      {\widetilde{\operatorname{z}}'_{\CATB}}{\operatorname{q}_{\CATB}}$};
    \node (B1_1) at (2.15, 2.55) {$\Downarrow\,\widetilde{\sigma}_{\CATB}$};
    \node (A1_0) at (0.65, 2) {$\Downarrow\,\mu_{\CATB}$};
    
    \node (C1_1) at (1.55, 1.6) {$\scriptstyle{\functor{F}_1
      (\widetilde{\operatorname{v}}_{\CATA})\circ\functor{F}_1(\operatorname{p}_{\CATA})}$};
    \node (C2_2) at (1.8, 1.88) {$\scriptstyle{\functor{F}_1
      (\widetilde{\operatorname{v}}_{\CATA})}$};
    
    \path (A2_0) edge [->,bend right=35]node [auto]
      {$\scriptstyle{\operatorname{z}_{\CATB}}$} (A0_2);
    \path (A2_0) edge [->,bend left=30]node [auto,swap] {$\scriptstyle{\operatorname{z}'_{\CATB}=
      \widetilde{\operatorname{z}}'_{\CATB}\circ\operatorname{q}_{\CATB}}$} (A4_2);
    
    \path (A0_2) edge [->,bend right=18]node [auto] {$\scriptstyle{\functor{F}_1
      (\operatorname{v}_{\CATA})=\functor{F}_1(\widetilde{\operatorname{v}}_{\CATA}\circ
      \operatorname{p}_{\CATA})}$} (A2_3);
    \path (A0_2) edge [->,bend left=18]node [auto,swap] {} (A2_3);
    \path (A3_2) edge [->,bend right=15]node [auto] {$\scriptstyle{\functor{F}_1
      (\widetilde{\operatorname{v}}_{\CATA})\circ\widetilde{\operatorname{z}}_{\CATB}}$} (A2_3);
    \path (A3_2) edge [->,bend left=20]node [auto,swap] {$\scriptstyle{\operatorname{v}_{\CATB}
      \circ\widetilde{\operatorname{z}}'_{\CATB}}$} (A2_3);
    \path (A2_0) edge [->,bend left=10]node [auto,swap]
      {$\scriptstyle{\operatorname{q}_{\CATB}}$} (A3_2);
    \path (A4_2) edge [->,bend left=30	]node [auto,swap]
      {$\scriptstyle{\operatorname{v}_{\CATB}}$} (A2_3);
    \path (A2_1) edge [->]node [auto] {} (A2_3);
    \path (A2_0) edge [->, bend right=20]node [auto] {$\scriptstyle{\functor{F}_1
      (\operatorname{p}_{\CATA})\circ\operatorname{z}_{\CATB}}$} (A2_1);
    \path (A2_0) edge [->,bend left=20]node [auto,swap]
      {$\scriptstyle{\widetilde{\operatorname{z}}_{\CATB}\circ\operatorname{q}_{\CATB}}$} (A2_1);
\end{tikzpicture}
\]

Now we recall that $\alpha_{\CATB}\ast i_{\widetilde{\operatorname{z}}'_{\CATB}}$ coincides with
diagram \eqref{eq-20}. Then it is not difficult to prove that the set of data $(A'_{\CATA},
A'_{\CATB},\operatorname{v}_{\CATA},\operatorname{z}_{\CATB},\operatorname{z}'_{\CATB},\alpha_{\CATA},
\sigma_{\CATB})$ satisfies the claim: basically, one has to consider the composition of
\eqref{eq-20} with $i_{\operatorname{q}_{\CATB}}$; then one has to insert in the center of such a
diagram $\mu_{\CATB}$ and $\mu^{-1}_{\CATB}$, together with all the necessary associators of $\CATB$.
\end{proof}

\begin{lem}\label{lem-12}
If $\functor{F}$ satisfies \emphatic{(\hyperref[A2]{A2})}, \emphatic{(\hyperref[A4]{A4})} and
\emphatic{(\hyperref[A5]{A5})}, then $\widetilde{\functor{G}}$ satisfies
\emphatic{(\hyperref[X2c]{X2c})}.
\end{lem}

\begin{proof}
\emph{For simplicity of exposition, we are giving this proof only the special case when both
$\CATA$ and $\CATB$ are $2$-categories instead of bicategories}. The interested reader can
easily fill out the details for the general case.\\

Let us fix any pair of objects $A_{\CATA},B_{\CATA}$ and any pair of morphisms $(A^m_{\CATA},
\operatorname{w}^m_{\CATA}$, $f^m_{\CATA}):A_{\CATA}\rightarrow B_{\CATA}$ in $\CATA\left[\SETWAinv
\right]$ for $m=1,2$. Moreover, let us fix any $2$-morphism

\[\Gamma_{\CATB}:\,\,\widetilde{\functor{G}}_1\Big(A^1_{\CATA},\operatorname{w}^1_{\CATA},f^1_{\CATA}
\Big)\Longrightarrow\widetilde{\functor{G}}_1\Big(A^2_{\CATA},\operatorname{w}^2_{\CATA},f^2_{\CATA}
\Big)\]
in $\CATB\left[\SETWBsatinv\right]$. We need to prove that there is a $2$-morphism $\Gamma_{\CATA}:
(A^1_{\CATA},\operatorname{w}^1_{\CATA},f^1_{\CATA})$ $\Rightarrow(A^2_{\CATA},
\operatorname{w}^2_{\CATA},f^2_{\CATA})$ in $\CATA\left[\SETWAinv\right]$, such that
$\widetilde{\functor{G}}_2(\Gamma_{\CATA})=\Gamma_{\CATB}$.\\

By construction of $\CATB\left[\SETWBsatinv\right]$, $\Gamma_{\CATB}$ is
represented by a set of data as follows

\begin{equation}\label{eq-26}
\begin{tikzpicture}[xscale=2.6,yscale=-0.8]
    \node (A0_2) at (2, 0) {$\functor{F}_0(A^1_{\CATA})$};
    \node (A2_2) at (2, 2) {$A_{\CATB}$};
    \node (A2_0) at (0, 2) {$\functor{F}_0(A_{\CATA})$};
    \node (A2_4) at (4, 2) {$\functor{F}_0(B_{\CATA})$,};
    \node (A4_2) at (2, 4) {$\functor{F}_0(A^2_{\CATA})$};
    
    \node (A2_3) at (2.8, 2) {$\Downarrow\,\alpha_{\CATB}$};
    \node (A2_1) at (1.2, 2) {$\Downarrow\,\gamma_{\CATB}$};
    
    \path (A4_2) edge [->]node [auto,swap] {$\scriptstyle{\functor{F}_1(f^2_{\CATA})}$} (A2_4);
    \path (A0_2) edge [->]node [auto] {$\scriptstyle{\functor{F}_1(f^1_{\CATA})}$} (A2_4);
    \path (A2_2) edge [->]node [auto,swap] {$\scriptstyle{\operatorname{v}^1_{\CATB}}$} (A0_2);
    \path (A2_2) edge [->]node [auto] {$\scriptstyle{\operatorname{v}^2_{\CATB}}$} (A4_2);
    \path (A4_2) edge [->]node [auto]
      {$\scriptstyle{\functor{F}_1(\operatorname{w}^2_{\CATA})}$} (A2_0);
    \path (A0_2) edge [->]node [auto,swap]
      {$\scriptstyle{\functor{F}_1(\operatorname{w}^1_{\CATA})}$} (A2_0);
\end{tikzpicture}
\end{equation}
such that $\functor{F}_1(\operatorname{w}^1_{\CATA})\circ\operatorname{v}^1_{\CATB}$ belongs
to $\SETWBsat$ and $\gamma_{\CATB}$ is invertible. We recall that in all this section
we are assuming that $\functor{F}_1(\SETWA)\subseteq\SETWBsat$, so
$\functor{F}_1(\operatorname{w}^m_{\CATA})$ belongs to $\SETWBsat$ for each $m=1,2$. Therefore
using~\cite[Proposition~2.11(ii)]{T4}, $\operatorname{v}^1_{\CATB}$ belongs to $\SETWBsat$. Hence,
using the definition of right saturation there is
no loss of generality in assuming that the data in \eqref{eq-26} are such that
$\operatorname{v}^1_{\CATB}$
belongs to $\SETWB$ and $\gamma_{\CATB}$ is invertible. Moreover, using (BF5) for $(\CATB,\SETWB)$
and $\gamma_{\CATB}^{-1}$, we get that $\functor{F}_1(\operatorname{w}^2_{\CATA})\circ
\operatorname{v}^2_{\CATB}$ belongs to $\SETWBsat$, hence also $\operatorname{v}^2_{\CATB}$ belongs to
$\SETWBsat$ by~\cite[Proposition~2.11(ii)]{T4}. If we apply (\hyperref[A2]{A2}) to the set of data

\[
\begin{tikzpicture}[xscale=2.4,yscale=-1.2]
    \node (A0_0) at (0, 0) {$\functor{F}_0(A^1_{\CATA})$};
    \node (A0_1) at (1, 0) {$A_{\CATB}$};
    \node (A0_2) at (2, 0) {$\functor{F}_0(A^2_{\CATA})$,};
    
    \path (A0_1) edge [->]node [auto,swap] {$\scriptstyle{\operatorname{v}^1_{\CATB}}$} (A0_0);
    \path (A0_1) edge [->]node [auto] {$\scriptstyle{\operatorname{v}^2_{\CATB}}$} (A0_2);
\end{tikzpicture} 
\]
we get an object $A^3_{\CATA}$, a pair of morphisms $\operatorname{v}^1_{\CATA}$ in
$\SETWA$ and $\operatorname{v}^2_{\CATA}$ in $\SETWAsat$ as follows

\[
\begin{tikzpicture}[xscale=2.4,yscale=-1.2]
    \node (A0_0) at (0, 0) {$A^1_{\CATA}$};
    \node (A0_1) at (1, 0) {$A^3_{\CATA}$};
    \node (A0_2) at (2, 0) {$A^2_{\CATA}$};
    
    \path (A0_1) edge [->]node [auto,swap] {$\scriptstyle{\operatorname{v}^1_{\CATA}}$} (A0_0);
    \path (A0_1) edge [->]node [auto] {$\scriptstyle{\operatorname{v}^2_{\CATA}}$} (A0_2);
\end{tikzpicture} 
\]
and a set of data in $\CATB$ as in the internal part of the following diagram

\[
\begin{tikzpicture}[xscale=2.6,yscale=-0.8]
    \node (A0_2) at (2, 0) {$A_{\CATB}$};
    \node (A2_2) at (2, 2) {$\overline{A}''_{\CATB}$};
    \node (A2_0) at (0, 2) {$\functor{F}_0(A^1_{\CATA})$};
    \node (A2_4) at (4, 2) {$\functor{F}_0(A^2_{\CATA})$,};
    \node (A4_2) at (2, 4) {$\functor{F}_0(A^3_{\CATA})$};
    
    \node (A2_3) at (2.8, 2) {$\Downarrow\,\phi^2_{\CATB}$};
    \node (A2_1) at (1.2, 2) {$\Downarrow\,\phi^1_{\CATB}$};
    
    \path (A4_2) edge [->]node [auto,swap]
      {$\scriptstyle{\functor{F}_1(\operatorname{v}^2_{\CATA})}$} (A2_4);
    \path (A0_2) edge [->]node [auto] {$\scriptstyle{\operatorname{v}^2_{\CATB}}$} (A2_4);
    \path (A2_2) edge [->]node [auto,swap]
      {$\scriptstyle{\operatorname{u}'_{\CATB}}$} (A0_2);
    \path (A2_2) edge [->]node [auto]
      {$\scriptstyle{\operatorname{u}_{\CATB}}$} (A4_2);
    \path (A4_2) edge [->]node [auto]
      {$\scriptstyle{\functor{F}_1(\operatorname{v}^1_{\CATA})}$} (A2_0);
    \path (A0_2) edge [->]node [auto,swap] {$\scriptstyle{\operatorname{v}^1_{\CATB}}$} (A2_0);
\end{tikzpicture}
\]
such that $\operatorname{u}'_{\CATB}$ belongs to $\SETWB$ and both
$\phi^1_{\CATB}$ and $\phi^2_{\CATB}$ are invertible. Since
$\phi^1_{\CATB}$ is invertible, then by (BF2) and (BF5)
for $(\CATB,\SETWB)$ we get that $\functor{F}_1(\operatorname{v}^1_{\CATA})\circ
\operatorname{u}_{\CATB}$ belongs to $\SETWB\subseteq\SETWBsat$. Moreover, $\functor{F}_1
(\operatorname{v}^1_{\CATA})$ belongs to $\SETWBsat$ because $\functor{F}_1(\SETWA)\subseteq
\SETWBsat$ by hypothesis. So again by~\cite[Proposition~2.11(ii)]{T4} we
get that $\operatorname{u}_{\CATB}$ belongs to $\SETWBsat$. So by definition of right
saturation there are an object
$A''_{\CATB}$ and a morphism $\operatorname{r}_{\CATB}:A''_{\CATB}\rightarrow\overline{A}''_{\CATB}$,
such that $\operatorname{u}_{\CATB}\circ
\operatorname{r}_{\CATB}$ belongs to $\SETWB$. Then we set:

\begin{gather*}
\gamma'_{\CATB}:=\Big(i_{\functor{F}_1(\operatorname{w}^2_{\CATA})}\ast
 \phi^2_{\CATB}\ast i_{\operatorname{r}_{\CATB}}
 \Big)\odot\Big(\gamma_{\CATB}\ast i_{\operatorname{u}'_{\CATB}\circ
 \operatorname{r}_{\CATB}}\Big)\odot\Big(
 i_{\functor{F}_1(\operatorname{w}^1_{\CATA})}\ast\left(\phi^1_{\CATB}\right)^{-1}
 \ast i_{\operatorname{r}_{\CATB}}\Big): \\
\phantom{\Big(}\functor{F}_1(\operatorname{w}^1_{\CATA})\circ\functor{F}_1
 (\operatorname{v}^1_{\CATA})\circ\operatorname{u}_{\CATB}\circ\operatorname{r}_{\CATB}
 \Longrightarrow\functor{F}_1
 (\operatorname{w}^2_{\CATA})\circ\functor{F}_1(\operatorname{v}^2_{\CATA})\circ
 \operatorname{u}_{\CATB}\circ\operatorname{r}_{\CATB}
\end{gather*}
and

\begin{gather*}
\alpha'_{\CATB}:=\Big(i_{\functor{F}_1(f^2_{\CATA})}\ast\phi^2_{\CATB}
 \ast i_{\operatorname{r}_{\CATB}}\Big)\odot\Big(\alpha_{\CATB}\ast
 i_{\operatorname{u}'_{\CATB}\circ\operatorname{r}_{\CATB}}\Big)\odot\Big(i_{\functor{F}_1
 (f^1_{\CATA})}\ast\left(\phi^1_{\CATB}\right)^{-1}\ast i_{\operatorname{r}_{\CATB}}
 \Big): \\
\phantom{\Big(}\functor{F}_1(f^1_{\CATA})\circ\functor{F}_1(\operatorname{v}^1_{\CATA})\circ
 \operatorname{u}_{\CATB}\circ\operatorname{r}_{\CATB}\Longrightarrow\functor{F}_1
 (f^2_{\CATA})\circ\functor{F}_1(\operatorname{v}^2_{\CATA})\circ\operatorname{u}_{\CATB}
 \circ\operatorname{r}_{\CATB}.
\end{gather*}

Then if we use the following diagram

\[
\begin{tikzpicture}[xscale=3.9,yscale=-0.8]
    \node (A0_4) at (4, 0) {$\functor{F}_0(A^1_{\CATA})$};
    \node (A2_2) at (3, 2) {$A''_{\CATB}$};
    \node (A2_4) at (4, 2) {$A''_{\CATB}$};
    \node (A2_6) at (5, 2) {$A_{\CATB}$};
    \node (A4_4) at (4, 4) {$\functor{F}_0(A^2_{\CATA})$};
    
    \node (A1_4) at (4, 0.7) {$\Rightarrow$};
    \node (B1_4) at (4, 1.3) {$\left(\phi^1_{\CATB}\right)^{-1}\ast
      i_{\operatorname{r}_{\CATB}}$};
    \node (A3_4) at (4, 3.3) {$\Leftarrow$};
    \node (B3_4) at (4, 2.7) {$\phi^2_{\CATB}\ast i_{\operatorname{r}_{\CATB}}$};
    
    \path (A2_2) edge [->]node [auto,swap] {$\scriptstyle{\functor{F}_1
      (\operatorname{v}^2_{\CATA})\circ\operatorname{u}_{\CATB}
      \circ\operatorname{r}_{\CATB}}$} (A4_4);
    \path (A2_4) edge [->]node [auto] {$\scriptstyle{\operatorname{u}'_{\CATB}
      \circ\operatorname{r}_{\CATB}}$} (A2_6);
    \path (A2_6) edge [->]node [auto] {$\scriptstyle{\operatorname{v}^2_{\CATB}}$} (A4_4);
    \path (A2_2) edge [->]node [auto] {$\scriptstyle{\functor{F}_1
      (\operatorname{v}^1_{\CATA})\circ\operatorname{u}_{\CATB}
      \circ\operatorname{r}_{\CATB}}$} (A0_4);
    \path (A2_4) edge [->]node [auto,swap] {$\scriptstyle{\id_{A''_{\CATB}}}$} (A2_2);
    \path (A2_6) edge [->]node [auto,swap] {$\scriptstyle{\operatorname{v}^1_{\CATB}}$} (A0_4);
\end{tikzpicture}
\]
together with \eqref{eq-26}, we get easily that

\begin{equation}\label{eq-24}
\Gamma_{\CATB}=\Big[A''_{\CATB},\functor{F}_1(\operatorname{v}^1_{\CATA})\circ
\operatorname{u}_{\CATB}\circ\operatorname{r}_{\CATB},\functor{F}_1
(\operatorname{v}^2_{\CATA})\circ\operatorname{u}_{\CATB}\circ\operatorname{r}_{\CATB},
\gamma'_{\CATB},\alpha'_{\CATB}\Big].
\end{equation}

Now we define

\begin{gather}
\nonumber \widetilde{\alpha}_{\CATB}:=\Big(\psi^{\functor{F}}_{f^2_{\CATA},
 \operatorname{v}^2_{\CATA}}\ast
 i_{\operatorname{u}_{\CATB}\circ\operatorname{r}_{\CATB}}\Big)^{-1}
 \odot\,\alpha'_{\CATB}\odot\Big(\psi^{\functor{F}}_{f^1_{\CATA},\operatorname{v}^1_{\CATA}}
 \ast i_{\operatorname{u}_{\CATB}\circ\operatorname{r}_{\CATB}}\Big): \\
\label{eq-01} \phantom{\Big(}\functor{F}_1(f^1_{\CATA}\circ\operatorname{v}^1_{\CATA})
 \circ\operatorname{u}_{\CATB}\circ\operatorname{r}_{\CATB}\Longrightarrow
 \functor{F}_1(f^2_{\CATA}\circ\operatorname{v}^2_{\CATA})\circ\operatorname{u}_{\CATB}
 \circ\operatorname{r}_{\CATB}.
\end{gather}

Then we apply (\hyperref[A5]{A5}) to the set of data

\[A^3_{\CATA},\quad B_{\CATA},\quad A''_{\CATB},\quad f^1_{\CATA}\circ\operatorname{v}^1_{\CATA},
\quad f^2_{\CATA}\circ\operatorname{v}^2_{\CATA},\quad\operatorname{u}_{\CATB}
\circ\operatorname{r}_{\CATB},\quad\widetilde{\alpha}_{\CATB},\]
so there are a pair of objects $A'_{\CATA},A'_{\CATB}$, a triple of morphisms
$\operatorname{u}_{\CATA}:A'_{\CATA}\rightarrow A^3_{\CATA}$ in $\SETWA$, $\operatorname{z}_{\CATB}:
A'_{\CATB}\rightarrow\functor{F}_0(A'_{\CATA})$ in $\SETWB$ and $\operatorname{z}'_{\CATB}:A'_{\CATB}
\rightarrow A''_{\CATB}$, a $2$-morphism

\[
\begin{tikzpicture}[xscale=1.8,yscale=-0.6]
    \node (B0_0) at (-1, 0) {}; 
    \node (B1_1) at (5, 0) {}; 
    
    \node (A0_1) at (2, 0) {$A^3_{\CATA}$};
    \node (A1_0) at (0, 1) {$A'_{\CATA}$};
    \node (A1_2) at (4, 1) {$B_{\CATA}$};
    \node (A2_1) at (2, 2) {$A^3_{\CATA}$};

    \node (A1_1) at (2, 1) {$\Downarrow\,\alpha_{\CATA}$};
    
    \path (A1_0) edge [->,bend right=20]node [auto] {$\scriptstyle{\operatorname{u}_{\CATA}}$} (A0_1);
    \path (A1_0) edge [->,bend left=20]node [auto,swap]
      {$\scriptstyle{\operatorname{u}_{\CATA}}$} (A2_1);
    \path (A0_1) edge [->,bend right=20]node [auto] {$\scriptstyle{f^1_{\CATA}
      \circ\operatorname{v}^1_{\CATA}}$} (A1_2);
    \path (A2_1) edge [->,bend left=20]node [auto,swap] {$\scriptstyle{f^2_{\CATA}
      \circ\operatorname{v}^2_{\CATA}}$} (A1_2);
\end{tikzpicture}
\]
and an invertible $2$-morphism
  
\[
\begin{tikzpicture}[xscale=1.8,yscale=-0.6]
    \node (B0_0) at (-1, 0) {}; 
    \node (B1_1) at (5, 0) {}; 
    
    \node (A0_1) at (2, 0) {$\functor{F}_0(A'_{\CATA})$};
    \node (A1_0) at (0, 1) {$A'_{\CATB}$};
    \node (A1_2) at (4, 1) {$\functor{F}_0(A^3_{\CATA})$,};
    \node (A2_1) at (2, 2) {$A^{\prime\prime}_{\CATB}$};
    
    \node (A1_1) at (2, 1) {$\Downarrow\,\sigma_{\CATB}$};

    \path (A1_0) edge [->,bend right=20]node [auto]
      {$\scriptstyle{\operatorname{z}_{\CATB}}$} (A0_1);
    \path (A1_0) edge [->,bend left=20]node [auto,swap]
      {$\scriptstyle{\operatorname{z}'_{\CATB}}$} (A2_1);
    \path (A0_1) edge [->,bend right=20]node [auto] {$\scriptstyle{\functor{F}_1
      (\operatorname{u}_{\CATA})}$} (A1_2);
    \path (A2_1) edge [->,bend left=20]node [auto,swap]
      {$\scriptstyle{\operatorname{u}_{\CATB}\circ\operatorname{r}_{\CATB}}$} (A1_2);
\end{tikzpicture}
\]
such that:

\begin{gather}
\nonumber \widetilde{\alpha}_{\CATB}\ast i_{\operatorname{z}'_{\CATB}}=
 \Big(i_{\functor{F}_1(f^2_{\CATA}\circ
 \operatorname{v}^2_{\CATA})}\ast\sigma_{\CATB}\Big)\odot
 \Big(\psi^{\functor{F}}_{f^2_{\CATA}\circ\operatorname{v}^2_{\CATA},
 \operatorname{u}_{\CATA}}\ast i_{\operatorname{z}_{\CATB}}\Big)\odot\Big(\functor{F}_2
 (\alpha_{\CATA})\ast i_{\operatorname{z}_{\CATB}}\Big)\odot\\
\label{eq-22} \odot\Big(\psi^{\functor{F}}_{f^1_{\CATA}
 \circ\operatorname{v}^1_{\CATA},\operatorname{u}_{\CATA}}\ast i_{\operatorname{z}_{\CATB}}
 \Big)^{-1}\odot\Big(i_{\functor{F}_1
 (f^1_{\CATA}\circ\operatorname{v}^1_{\CATA})}\ast\sigma_{\CATB}^{-1}\Big).
\end{gather}

Using the interchange law on $\CATB$, we get the following identity:

\begin{gather}
\nonumber \alpha'_{\CATB}\ast i_{\operatorname{z}'_{\CATB}}\stackrel{\eqref{eq-01}}{=} \\
\nonumber \stackrel{\eqref{eq-01}}{=}\Big(\psi^{\functor{F}}_{f^2_{\CATA},
  \operatorname{v}^2_{\CATA}}\ast i_{\operatorname{u}_{\CATB}
  \circ\operatorname{r}_{\CATB}\circ\operatorname{z}'_{\CATB}}\Big)\odot\Big(
  \widetilde{\alpha}_{\CATB}\ast i_{\operatorname{z}'_{\CATB}}\Big)\odot\Big(
  \psi^{\functor{F}}_{f^1_{\CATA},\operatorname{v}^1_{\CATA}}\ast i_{\operatorname{u}_{\CATB}
  \circ\operatorname{r}_{\CATB}\circ\operatorname{z}'_{\CATB}}\Big)^{-1}\stackrel{\eqref{eq-22}}{=}\\
\nonumber \stackrel{\eqref{eq-22}}{=}\Big(\psi^{\functor{F}}_{f^2_{\CATA},
 \operatorname{v}^2_{\CATA}}\ast
 i_{\operatorname{u}_{\CATB}\circ\operatorname{r}_{\CATB}\circ\operatorname{z}'_{\CATB}}\Big)\odot
 \Big(i_{\functor{F}_1(f^2_{\CATA}\circ\operatorname{v}^2_{\CATA})}\ast\sigma_{\CATB}\Big)\odot \\
\nonumber \odot\Big(\psi^{\functor{F}}_{f^2_{\CATA}\circ\operatorname{v}^2_{\CATA},
 \operatorname{u}_{\CATA}}\ast i_{\operatorname{z}_{\CATB}}\Big)\odot\Big(\functor{F}_2
 (\alpha_{\CATA})\ast i_{\operatorname{z}_{\CATB}}\Big)\odot\Big(\psi^{\functor{F}}_{f^1_{\CATA}
 \circ\operatorname{v}^1_{\CATA},\operatorname{u}_{\CATA}}\ast i_{\operatorname{z}_{\CATB}}
 \Big)^{-1}\odot \\
\nonumber \odot\Big(i_{\functor{F}_1
 (f^1_{\CATA}\circ\operatorname{v}^1_{\CATA})}\ast\sigma_{\CATB}^{-1}\Big)\odot\Big(
 \psi^{\functor{F}}_{f^1_{\CATA},\operatorname{v}^1_{\CATA}}\ast i_{\operatorname{u}_{\CATB}\circ
 \operatorname{r}_{\CATB}\circ\operatorname{z}'_{\CATB}}\Big)^{-1}= \\
\nonumber =\Big(i_{\functor{F}_1(f^2_{\CATA})\circ\functor{F}_1(\operatorname{v}^2_{\CATA})}\ast
 \sigma_{\CATB}\Big)\odot\Big(\psi^{\functor{F}}_{f^2_{\CATA},\operatorname{v}^2_{\CATA}}\ast
 i_{\functor{F}_1(\operatorname{u}_{\CATA})\circ\operatorname{z}_{\CATB}}\Big)\odot \\
\nonumber \odot\Big(\psi^{\functor{F}}_{f^2_{\CATA}\circ\operatorname{v}^2_{\CATA},
  \operatorname{u}_{\CATA}}\ast i_{\operatorname{z}_{\CATB}}\Big)\odot\Big(\functor{F}_2
  (\alpha_{\CATA})\ast i_{\operatorname{z}_{\CATB}}\Big)\odot\Big(\psi^{\functor{F}}_{f^1_{\CATA}
  \circ\operatorname{v}^1_{\CATA},\operatorname{u}_{\CATA}}\ast i_{\operatorname{z}_{\CATB}}
  \Big)^{-1}\odot \\
\label{eq-21} \odot\Big(\psi^{\functor{F}}_{f^1_{\CATA},\operatorname{v}^1_{\CATA}}\ast
 i_{\functor{F}_1(\operatorname{u}_{\CATA})\circ\operatorname{z}_{\CATB}}\Big)^{-1}\odot
 \Big(i_{\functor{F}_1(f^1_{\CATA})\circ\functor{F}_1(\operatorname{v}^1_{\CATA})}\ast
 \sigma_{\CATB}^{-1}\Big). 
\end{gather} 

Since $\functor{F}$ is a pseudofunctor and since we are assuming for simplicity that
$\CATA$ and $\CATB$ are $2$-categories, then for each $m=1,2$ we have

\[\Big(\psi^{\functor{F}}_{f^m_{\CATA},\operatorname{v}^m_{\CATA}}
\ast i_{\functor{F}_1(\operatorname{u}_{\CATA})}\Big)\odot\Big(\psi^{\functor{F}}_{f^m_{\CATA}\circ
\operatorname{v}^m_{\CATA},\operatorname{u}_{\CATA}}\Big)=\Big(i_{\functor{F}_1(f^m_{\CATA})}\ast
\psi^{\functor{F}}_{\operatorname{v}^m_{\CATA},\operatorname{u}_{\CATA}}\Big)\odot\Big(
\psi^{\functor{F}}_{f^m_{\CATA},\operatorname{v}^m_{\CATA}\circ\operatorname{u}_{\CATA}}\Big).\]

So by replacing in \eqref{eq-21} we get

\begin{gather}
\alpha'_{\CATB}\ast i_{\operatorname{z}'_{\CATB}}=\Big\{i_{\functor{F}_1(f^2_{\CATA})}\ast\Big[
 \Big(i_{\functor{F}_1(\operatorname{v}^2_{\CATA})}\ast\sigma_{\CATB}\Big)\odot
 \Big(\psi^{\functor{F}}_{\operatorname{v}^2_{\CATA},
 \operatorname{u}_{\CATA}}\ast i_{\operatorname{z}_{\CATB}}\Big)\Big]\Big\}
 \odot \\
\nonumber \odot\Big\{\Big[\psi^{\functor{F}}_{f^2_{\CATA},\operatorname{v}^2_{\CATA}
 \circ\operatorname{u}_{\CATA}}\odot\functor{F}_2(\alpha_{\CATA})\odot
 \Big(\psi^{\functor{F}}_{f^1_{\CATA},\operatorname{v}^1_{\CATA}
 \circ\operatorname{u}_{\CATA}}\Big)^{-1}\Big]\ast
 i_{\operatorname{z}_{\CATB}}\Big\}\odot \\
\label{eq-30} \odot\Big\{i_{\functor{F}_1(f^1_{\CATA})}\ast\Big[\Big(
 \psi^{\functor{F}}_{\operatorname{v}^1_{\CATA},\operatorname{u}_{\CATA}}\ast
 i_{\operatorname{z}_{\CATB}}\Big)^{-1}\odot\Big(i_{\functor{F}_1
 (\operatorname{v}^1_{\CATA})}\ast\sigma_{\CATB}^{-1}\Big)\Big]\Big\}.
\end{gather}

Now we define an invertible $2$-morphism

\begin{gather*}
\gamma''_{\CATB}:=\Big\{i_{\functor{F}_1(\operatorname{w}^2_{\CATA})}\ast\Big[
 \Big(\psi^{\functor{F}}_{\operatorname{v}^2_{\CATA},\operatorname{u}_{\CATA}}\ast
 i_{\operatorname{z}_{\CATB}}\Big)^{-1}\odot\Big(i_{\functor{F}_1(\operatorname{v}^2_{\CATA})}\ast
 \sigma_{\CATB}^{-1}\Big)\odot\Big(\phi^2_{\CATB}\ast
 i_{\operatorname{r}_{\CATB}\circ\operatorname{z}'_{\CATB}}\Big)\Big]\Big\}\odot \\
\odot\Big(\gamma_{\CATB}\ast
 i_{\operatorname{u}'_{\CATB}\circ\operatorname{r}_{\CATB}\circ\operatorname{z}'_{\CATB}}\Big)\odot \\
\odot\Big\{i_{\functor{F}_1(\operatorname{w}^1_{\CATA})}\ast\Big[\Big(\left(\phi^1_{\CATB}
 \right)^{-1}
 \ast i_{\operatorname{r}_{\CATB}\circ\operatorname{z}'_{\CATB}}\Big)\odot
 \Big(i_{\functor{F}_1(\operatorname{v}^1_{\CATA})}\ast\sigma_{\CATB}
 \Big)\odot\Big(\psi^{\functor{F}}_{\operatorname{v}^1_{\CATA},
 \operatorname{u}_{\CATA}}\ast i_{\operatorname{z}_{\CATB}}\Big)\Big]\Big\}: \\
\phantom{\Big(}\functor{F}_1(\operatorname{w}^1_{\CATA})\circ\functor{F}_1
 (\operatorname{v}^1_{\CATA}\circ\operatorname{u}_{\CATA})\circ\operatorname{z}_{\CATB}
 \Longrightarrow\functor{F}_1(\operatorname{w}^2_{\CATA})\circ\functor{F}_1
 (\operatorname{v}^2_{\CATA}\circ\operatorname{u}_{\CATA})\circ\operatorname{z}_{\CATB}.
\end{gather*}

Then by definition of $\gamma'_{\CATB}$ and $\gamma''_{\CATB}$ we have:

\begin{gather}
\nonumber \gamma'_{\CATB}\ast i_{\operatorname{z}'_{\CATB}}=
 \Big\{i_{\functor{F}_1(\operatorname{w}^2_{\CATA})}\ast\Big[\Big(
 i_{\functor{F}_1(\operatorname{v}^2_{\CATA})}\ast\sigma_{\CATB}\Big)\odot\Big(
 \psi^{\functor{F}}_{\operatorname{v}^2_{\CATA},\operatorname{u}_{\CATA}}
  \ast i_{\operatorname{z}_{\CATB}}\Big)\Big]\Big\}\odot \\
\label{eq-29} \odot\,\gamma''_{\CATB}\,\odot
  \Big\{i_{\functor{F}_1(\operatorname{w}^1_{\CATA})}\ast\Big[
 \Big(\psi^{\functor{F}}_{\operatorname{v}^1_{\CATA},\operatorname{u}_{\CATA}}\ast
 i_{\operatorname{z}_{\CATB}}\Big)^{-1}\odot\Big(i_{\functor{F}_1(\operatorname{v}^1_{\CATA})}
 \ast\sigma_{\CATB}^{-1}\Big)\Big]\Big\}.
\end{gather}

Then let us consider the following set of data

\[
\begin{tikzpicture}[xscale=-2.6,yscale=-0.8]
    \node (A0_4) at (4, 0) {$\functor{F}_0(A^1_{\CATA})$};
    \node (A2_2) at (3, 2) {$A'_{\CATB}$,};
    \node (A2_4) at (4, 2) {$A'_{\CATB}$};
    \node (A2_6) at (5, 2) {$A''_{\CATB}$};
    \node (A4_4) at (4, 4) {$\functor{F}_0(A^2_{\CATA})$};
        
    \node (A1_4) at (4, 0.8) {$\Rightarrow$};
    \node (B1_4) at (4, 1.3) {$\xi^1_{\CATB}$};
    \node (A3_4) at (4, 3.2) {$\Leftarrow$};
    \node (B3_4) at (4, 2.7) {$\xi^2_{\CATB}$};
    
    \path (A2_2) edge [->]node [auto] {$\scriptstyle{\functor{F}_1(\operatorname{v}^2_{\CATA}
      \circ\operatorname{u}_{\CATA})\circ\operatorname{z}_{\CATB}}$} (A4_4);
    \path (A2_4) edge [->]node [auto,swap] {$\scriptstyle{\operatorname{z}'_{\CATB}}$} (A2_6);
    \path (A2_6) edge [->]node [auto,swap] {$\scriptstyle{\functor{F}_1
      (\operatorname{v}^2_{\CATA})\circ\operatorname{u}_{\CATB}
      \circ\operatorname{r}_{\CATB}}$} (A4_4);
    \path (A2_2) edge [->]node [auto,swap] {$\scriptstyle{\functor{F}_1(\operatorname{v}^1_{\CATA}
      \circ\operatorname{u}_{\CATA})\circ\operatorname{z}_{\CATB}}$} (A0_4);
    \path (A2_4) edge [->]node [auto] {$\scriptstyle{\id_{A'_{\CATB}}}$} (A2_2);
    \path (A2_6) edge [->]node [auto] {$\scriptstyle{\functor{F}_1
      (\operatorname{v}^1_{\CATA})\circ\operatorname{u}_{\CATB}
      \circ\operatorname{r}_{\CATB}}$} (A0_4);
\end{tikzpicture}
\]
where

\[\xi^1_{\CATB}:=\Big(\psi^{\functor{F}}_{\operatorname{v}^1_{\CATA},
\operatorname{u}_{\CATA}}\ast i_{\operatorname{z}_{\CATB}}\Big)^{-1}\odot\Big(i_{\functor{F}_1
(\operatorname{v}^1_{\CATA})}\ast\sigma_{\CATB}^{-1}\Big)\]
and

\[\xi^2_{\CATB}:=\Big(i_{\functor{F}_1(\operatorname{v}^2_{\CATA})}\ast\sigma_{\CATB}
 \Big)\odot\Big(\psi^{\functor{F}}_{\operatorname{v}^2_{\CATA},
 \operatorname{u}_{\CATA}}\ast i_{\operatorname{z}_{\CATB}}\Big).
\]

Using such data with \eqref{eq-24}, \eqref{eq-30}, \eqref{eq-29} and the definition of
$2$-morphisms in~\cite[\S~2.3]{Pr}, we conclude that:

\begin{gather}
\nonumber \Gamma_{\CATB}=\Big[A'_{\CATB},\functor{F}_1(\operatorname{v}^1_{\CATA}\circ
 \operatorname{u}_{\CATA})\circ\operatorname{z}_{\CATB},\functor{F}_1(\operatorname{v}^2_{\CATA}\circ
 \operatorname{u}_{\CATA})\circ\operatorname{z}_{\CATB},\gamma''_{\CATB}, \\
\label{eq-02} \Big(\psi^{\functor{F}}_{f^2_{\CATA},\operatorname{v}^2_{\CATA}\circ
 \operatorname{u}_{\CATA}}\odot\functor{F}_2(\alpha_{\CATA})\odot\Big(\psi^{\functor{F}}_{f^1_{\CATA},
 \operatorname{v}^1_{\CATA}\circ\operatorname{u}_{\CATA}}\Big)^{-1}\Big)\ast
 i_{\operatorname{z}_{\CATB}}\Big].
\end{gather}

Now we consider the invertible $2$-morphism

\begin{gather*}
\widetilde{\gamma}_{\CATB}:=\Big(\psi^{\functor{F}}_{\operatorname{w}^2_{\CATA},
 \operatorname{v}^2_{\CATA}\circ\operatorname{u}_{\CATA}}\ast i_{\operatorname{z}_{\CATB}}
 \Big)^{-1}\odot\,\gamma''_{\CATB}\odot
 \Big(\psi^{\functor{F}}_{\operatorname{w}^1_{\CATA},\operatorname{v}^1_{\CATA}\circ
 \operatorname{u}_{\CATA}}\ast i_{\operatorname{z}_{\CATB}}\Big): \\
\phantom{\Big(}\functor{F}_1(\operatorname{w}^1_{\CATA}\circ\operatorname{v}^1_{\CATA}\circ
 \operatorname{u}_{\CATA})\circ\operatorname{z}_{\CATB}\Longrightarrow\functor{F}_1
 (\operatorname{w}^2_{\CATA}\circ\operatorname{v}^2_{\CATA}\circ\operatorname{u}_{\CATA})\circ
 \operatorname{z}_{\CATB}.
\end{gather*}

Since we are assuming (\hyperref[A4]{A4}) and (\hyperref[A5]{A5}), then we can apply
Lemma~\ref{lem-03} on the set of data:

\[A'_{\CATA},\quad A_{\CATA},\quad A'_{\CATB},\quad\operatorname{w}^1_{\CATA}\circ
\operatorname{v}^1_{\CATA}\circ\operatorname{u}_{\CATA},\quad\operatorname{w}^2_{\CATA}\circ
\operatorname{v}^2_{\CATA}\circ\operatorname{u}_{\CATA},\quad\operatorname{z}_{\CATB},\quad
\widetilde{\gamma}_{\CATB}.\]

Then there are a pair of objects $\widetilde{A}_{\CATA},\widetilde{A}_{\CATB}$, a triple of morphisms
$\operatorname{z}_{\CATA}:\widetilde{A}_{\CATA}\rightarrow A'_{\CATA}$ in $\SETWA$,
$\operatorname{t}_{\CATB}:\widetilde{A}_{\CATB}\rightarrow\functor{F}_0(\widetilde{A}_{\CATA})$
in $\SETWB$ and $\operatorname{t}'_{\CATB}:\widetilde{A}_{\CATB}\rightarrow A'_{\CATB}$, an
\emph{invertible} $2$-morphism

\[
\begin{tikzpicture}[xscale=1.8,yscale=-0.6]
    \node (B0_0) at (-1, 0) {}; 
    \node (B1_1) at (5, 0) {}; 

    \node (A0_1) at (2, 0) {$A'_{\CATA}$};
    \node (A1_0) at (0, 1) {$\widetilde{A}_{\CATA}$};
    \node (A1_2) at (4, 1) {$A_{\CATA}$};
    \node (A2_1) at (2, 2) {$A'_{\CATA}$};

    \node (A1_1) at (2, 1) {$\Downarrow\,\gamma_{\CATA}$};

    \path (A1_0) edge [->,bend right=20]node [auto] {$\scriptstyle{\operatorname{z}_{\CATA}}$} (A0_1);
    \path (A1_0) edge [->,bend left=20]node [auto,swap]
      {$\scriptstyle{\operatorname{z}_{\CATA}}$} (A2_1);
    \path (A0_1) edge [->,bend right=20]node [auto] {$\scriptstyle{\operatorname{w}^1_{\CATA}
      \circ\operatorname{v}^1_{\CATA}\circ\operatorname{u}_{\CATA}}$} (A1_2);
    \path (A2_1) edge [->,bend left=20]node [auto,swap] {$\scriptstyle{\operatorname{w}^2_{\CATA}
      \circ\operatorname{v}^2_{\CATA}\circ\operatorname{u}_{\CATA}}$} (A1_2);
\end{tikzpicture}
\]
and an invertible $2$-morphism
  
\[
\begin{tikzpicture}[xscale=1.8,yscale=-0.6]
    \node (B0_0) at (-1, 0) {}; 
    \node (B1_1) at (5, 0) {}; 
    
    \node (A0_1) at (2, 0) {$\functor{F}_0(\widetilde{A}_{\CATA})$};
    \node (A1_0) at (0, 1) {$\widetilde{A}_{\CATB}$};
    \node (A1_2) at (4, 1) {$\functor{F}_0(A'_{\CATA})$,};
    \node (A2_1) at (2, 2) {$A'_{\CATB}$};

    \node (A1_1) at (2, 1) {$\Downarrow\,\nu_{\CATB}$};
    
    \path (A1_0) edge [->,bend right=20]node [auto]
      {$\scriptstyle{\operatorname{t}_{\CATB}}$} (A0_1);
    \path (A1_0) edge [->,bend left=20]node [auto,swap]
      {$\scriptstyle{\operatorname{t}'_{\CATB}}$} (A2_1);
    \path (A0_1) edge [->,bend right=20]node [auto] {$\scriptstyle{\functor{F}_1
      (\operatorname{z}_{\CATA})}$} (A1_2);
    \path (A2_1) edge [->,bend left=20]node [auto,swap]
      {$\scriptstyle{\operatorname{z}_{\CATB}}$} (A1_2);
\end{tikzpicture}
\]
such that

\begin{gather}
\label{eq-27} \widetilde{\gamma}_{\CATB}\ast i_{\operatorname{t}'_{\CATB}}=\Big(i_{\functor{F}_1
 (\operatorname{w}^2_{\CATA}\circ\operatorname{v}^2_{\CATA}\circ\operatorname{u}_{\CATA})}\ast
 \nu_{\CATB}\Big)\odot\Big(\psi^{\functor{F}}_{\operatorname{w}^2_{\CATA}\circ
 \operatorname{v}^2_{\CATA}\circ\operatorname{u}_{\CATA},\operatorname{z}_{\CATA}}\ast
 i_{\operatorname{t}_{\CATB}}\Big)\odot \\
\nonumber \odot\Big(\functor{F}_2(\gamma_{\CATA})\ast i_{\operatorname{t}_{\CATB}}\Big)\odot
 \Big(\psi^{\functor{F}}_{\operatorname{w}^1_{\CATA}\circ\operatorname{v}^1_{\CATA}\circ
 \operatorname{u}_{\CATA},\operatorname{z}_{\CATA}}\ast i_{\operatorname{t}_{\CATB}}\Big)^{-1}\odot
 \Big(i_{\functor{F}_1(\operatorname{w}^1_{\CATA}\circ\operatorname{v}^1_{\CATA}\circ
 \operatorname{u}_{\CATA})}\ast\nu_{\CATB}^{-1}\Big).
\end{gather}

Now we replace in \eqref{eq-27} the definition of $\widetilde{\gamma}_{\CATB}$ and we do a series
of computations analogous to those leading from \eqref{eq-22} to \eqref{eq-30}. So we get that

\begin{gather}
\nonumber \gamma''_{\CATB}\ast i_{\operatorname{t}'_{\CATB}}= \\
\nonumber =\Big\{i_{\functor{F}_1(\operatorname{w}^2_{\CATA})}\ast\Big[\Big(i_{\functor{F}_1
 (\operatorname{v}^2_{\CATA}\circ
 \operatorname{u}_{\CATA})}\ast\nu_{\CATB}\Big)\odot\Big(
 \psi^{\functor{F}}_{\operatorname{v}^2_{\CATA}\circ
 \operatorname{u}_{\CATA},\operatorname{z}_{\CATA}}\ast i_{\operatorname{t}_{\CATB}}\Big)\Big]
 \Big\}\odot \\
\nonumber \odot\Big\{\Big[\psi^{\functor{F}}_{\operatorname{w}^2_{\CATA},\operatorname{v}^2_{\CATA}
 \circ\operatorname{u}_{\CATA}\circ\operatorname{z}_{\CATA}}\odot\functor{F}_2(
 \gamma_{\CATA})\odot\Big(\psi^{\functor{F}}_{\operatorname{w}^1_{\CATA},
 \operatorname{v}^1_{\CATA}\circ\operatorname{u}_{\CATA}\circ\operatorname{z}_{\CATA}}\Big)^{-1}
 \Big]\ast i_{\operatorname{t}_{\CATB}}\Big\}\odot \\
\label{eq-03} \odot\Big\{i_{\functor{F}_1(\operatorname{w}^1_{\CATA})}\ast\Big[\Big(
 \psi^{\functor{F}}_{\operatorname{v}^1_{\CATA}\circ\operatorname{u}_{\CATA},
 \operatorname{z}_{\CATA}}\ast i_{\operatorname{t}_{\CATB}}\Big)^{-1}\odot\Big(i_{\functor{F}_1
 (\operatorname{v}^1_{\CATA}\circ\operatorname{u}_{\CATA})}\ast\nu_{\CATB}^{-1}\Big)\Big]\Big\}.
\end{gather}

Moreover, using the interchange law on the $2$-category $\CATB$, we have:

\begin{gather}
\nonumber \Big(\psi^{\functor{F}}_{f^2_{\CATA},\operatorname{v}^2_{\CATA}\circ
 \operatorname{u}_{\CATA}}\odot\functor{F}_2(\alpha_{\CATA})\odot\Big(\psi^{\functor{F}}_{f^1_{\CATA},
 \operatorname{v}^1_{\CATA}\circ\operatorname{u}_{\CATA}}\Big)^{-1}\Big)
 \ast i_{\operatorname{z}_{\CATB}\circ\operatorname{t}'_{\CATB}}= \\
\nonumber =\Big\{i_{\functor{F}_1(f^2_{\CATA})}\ast\Big[\Big(i_{\functor{F}_1
 (\operatorname{v}^2_{\CATA}\circ\operatorname{u}_{\CATA})}\ast\nu_{\CATB}\Big)\odot
 \Big(\psi^{\functor{F}}_{\operatorname{v}^2_{\CATA}\circ
 \operatorname{u}_{\CATA},\operatorname{z}_{\CATA}}\ast i_{\operatorname{t}_{\CATB}}\Big)\Big]
 \Big\}\odot \\
\nonumber \odot\Big\{\Big[\psi^{\functor{F}}_{f^2_{\CATA},
 \operatorname{v}^2_{\CATA}\circ\operatorname{u}_{\CATA}\circ\operatorname{z}_{\CATA}}\odot
 \functor{F}_2(\alpha_{\CATA}
 \ast i_{\operatorname{z}_{\CATA}})\odot\Big(\psi^{\functor{F}}_{f^1_{\CATA},
 \operatorname{v}^1_{\CATA}\circ\operatorname{u}_{\CATA}\circ\operatorname{z}_{\CATA}}\Big)^{-1}
 \Big]\ast i_{\operatorname{t}_{\CATB}}\Big\}\odot \\
\label{eq-04} \odot\Big\{i_{\functor{F}_1(f^1_{\CATA})}\ast
 \Big[\Big(\psi^{\functor{F}}_{\operatorname{v}^1_{\CATA}\circ\operatorname{u}_{\CATA},
 \operatorname{z}_{\CATA}}\ast i_{\operatorname{t}_{\CATB}}\Big)^{-1}\odot\Big(i_{\functor{F}_1
 (\operatorname{v}^1_{\CATA}\circ\operatorname{u}_{\CATA})}\ast\nu_{\CATB}^{-1}\Big)\Big]\Big\}.
\end{gather}

Then we consider the following diagram

\[
\begin{tikzpicture}[xscale=2.8,yscale=-0.8]
    \node (A0_1) at (1, 0) {$\functor{F}_0(A^1_{\CATA})$};
    \node (A2_0) at (0, 2) {$A'_{\CATB}$};
    \node (A2_1) at (1, 2) {$\widetilde{A}_{\CATB}$};
    \node (A2_2) at (2, 2) {$\functor{F}_0(\widetilde{A}_{\CATA})$,};
    \node (A4_1) at (1, 4) {$\functor{F}_0(A^2_{\CATA})$};
    
    \node (A3_0) at (1, 3.2) {$\Leftarrow$};
    \node (A1_0) at (1, 0.8) {$\Rightarrow$};
    \node (A1_1) at (1, 1.3) {$\eta^1_{\CATB}$};
    \node (B1_1) at (1, 2.7) {$\eta^2_{\CATB}$};
    
    \path (A2_1) edge [->]node [auto,swap] {$\scriptstyle{\operatorname{t}'_{\CATB}}$} (A2_0);
    \path (A2_1) edge [->]node [auto] {$\scriptstyle{\operatorname{t}_{\CATB}}$} (A2_2);
    \path (A2_0) edge [->]node [auto] {$\scriptstyle{\functor{F}_1
      (\operatorname{v}^1_{\CATA}\circ\operatorname{u}_{\CATA})\circ
      \operatorname{z}_{\CATB}}$} (A0_1);
    \path (A2_2) edge [->]node [auto,swap] {$\scriptstyle{\functor{F}_1
      (\operatorname{v}^1_{\CATA}\circ\operatorname{u}_{\CATA}\circ
      \operatorname{z}_{\CATA})}$} (A0_1);
    \path (A2_2) edge [->]node [auto] {$\scriptstyle{\functor{F}_1
      (\operatorname{v}^2_{\CATA}\circ\operatorname{u}_{\CATA}\circ
      \operatorname{z}_{\CATA})}$} (A4_1);
    \path (A2_0) edge [->]node [auto,swap] {$\scriptstyle{\functor{F}_1
      (\operatorname{v}^2_{\CATA}\circ\operatorname{u}_{\CATA})\circ
      \operatorname{z}_{\CATB}}$} (A4_1);
\end{tikzpicture}
\]
where

\[\eta^1_{\CATB}:=\Big(\psi^{\functor{F}}_{\operatorname{v}^1_{\CATA}\circ
\operatorname{u}_{\CATA},\operatorname{z}_{\CATA}}\ast i_{\operatorname{t}_{\CATB}}\Big)^{-1}
\odot\Big(i_{\functor{F}_1(\operatorname{v}^1_{\CATA}\circ\operatorname{u}_{\CATA})}\ast
\nu^{-1}_{\CATB}\Big)\]
and

\[\eta^2_{\CATB}:=\Big(i_{\functor{F}_1
(\operatorname{v}^2_{\CATA}\circ\operatorname{u}_{\CATA})}\ast\nu_{\CATB}\Big)\odot\Big(
\psi^{\functor{F}}_{\operatorname{v}^2_{\CATA}\circ\operatorname{u}_{\CATA},
\operatorname{z}_{\CATA}}\ast i_{\operatorname{t}_{\CATB}}\Big).\]

Using \eqref{eq-02}, \eqref{eq-03}, \eqref{eq-04} and the previous diagram, we conclude that:

\begin{gather}
\nonumber \Gamma_{\CATB}=\Big[\functor{F}_0(\widetilde{A}_{\CATA}),
 \functor{F}_1(\operatorname{v}^1_{\CATA}\circ
 \operatorname{u}_{\CATA}\circ\operatorname{z}_{\CATA}),\functor{F}_1(\operatorname{v}^2_{\CATA}
 \circ\operatorname{u}_{\CATA}\circ\operatorname{z}_{\CATA}), \\
\nonumber \psi^{\functor{F}}_{\operatorname{w}^2_{\CATA},\operatorname{v}^2_{\CATA}\circ
 \operatorname{u}_{\CATA}\circ\operatorname{z}_{\CATA}}\odot 
 \functor{F}_2(\gamma_{\CATA})\odot\Big(\psi^{\functor{F}}_{\operatorname{w}^1_{\CATA},
 \operatorname{v}^1_{\CATA}\circ\operatorname{u}_{\CATA}\circ\operatorname{z}_{\CATA}}\Big)^{-1}, \\
\label{eq-09} \psi^{\functor{F}}_{f^2_{\CATA},\operatorname{v}^2_{\CATA}\circ\operatorname{u}_{\CATA}
 \circ\operatorname{z}_{\CATA}}\odot
 \functor{F}_2(\alpha_{\CATA}\ast i_{\operatorname{z}_{\CATA}})\odot
 \Big(\psi^{\functor{F}}_{f^1_{\CATA},\operatorname{v}^1_{\CATA}\circ\operatorname{u}_{\CATA}
 \circ\operatorname{z}_{\CATA}}\Big)^{-1}\Big].
\end{gather}

Then we define a $2$-morphism in $\CATA\left[\SETWAinv\right]$ as follows

\begin{gather*}
\Gamma_{\CATA}:=\Big[\widetilde{A}_{\CATA},\operatorname{v}^1_{\CATA}\circ
 \operatorname{u}_{\CATA}\circ\operatorname{z}_{\CATA},\operatorname{v}^2_{\CATA}\circ
 \operatorname{u}_{\CATA}\circ\operatorname{z}_{\CATA},\gamma_{\CATA},\alpha_{\CATA}\ast
 i_{\operatorname{z}_{\CATA}}\Big]: \\
 \Big(A^1_{\CATA},\operatorname{w}^1_{\CATA},f^1_{\CATA}\Big)\Longrightarrow
 \Big(A^2_{\CATA},\operatorname{w}^2_{\CATA},f^2_{\CATA}\Big).
\end{gather*}

Using \eqref{eq-09} and \eqref{eq-07}, we conclude that
$\Gamma_{\CATB}=\widetilde{\functor{G}}_2(\Gamma_{\CATA})$; this proves that
$\widetilde{\functor{G}}$ satisfies condition (\hyperref[X2c]{X2c}).
\end{proof}

Note that in the proof above the $2$-morphism $\Gamma_{\CATA}$ is well-defined because
$\gamma_{\CATA}$ is an invertible $2$-morphism thanks to Lemma~\ref{lem-03} (we recall that
by~\cite[\S~2.3]{Pr} the data defining a $2$-morphism in a bicategory of fractions must satisfy
such a technical condition). This explains why we needed to prove such a result before
Lemma~\ref{lem-12}.

\section{Proofs of the main results}
\begin{proof}[Proof of Theorem~\ref{theo-02}.]
Let us fix any pair $(\functor{G},\kappa)$ as in Theorem~\ref{theo-01}(iv). By that Theorem,
$\functor{G}:\CATA\left[\SETWAinv\right]\rightarrow\CATB\left[\SETWBinv\right]$ is an equivalence
of bicategories if and only if $\widetilde{\functor{G}}:\CATA\left[\SETWAinv\right]\rightarrow\CATB
\left[\SETWBsatinv
\right]$ is an equivalence of bicategories. Using Lemmas from~\ref{lem-01} to~\ref{lem-06}, if
$\widetilde{\functor{G}}$ is an equivalence of bicategories, then $\functor{F}$ satisfies conditions
(\hyperref[A1]{A1}) -- (\hyperref[A5]{A5}). Conversely, if $\functor{F}$ satisfies
such conditions, then $\widetilde{\functor{G}}$ is an equivalence of bicategories by
Lemmas from~\ref{lem-11} to~\ref{lem-12}. This is enough to conclude.
\end{proof}

In the remaining part of this section we are going to prove Theorem~\ref{theo-04} and
Corollary~\ref{cor-01}.\\

We recall (see~\cite[Definition~3.3]{PP}) that a \emph{quasi-unit} of any given bicategory $\CATB$
is any morphism of the form $f_{\CATB}:A_{\CATB}\rightarrow A_{\CATB}$, admitting an
invertible $2$-morphism to $\id_{A_{\CATB}}$. We denote by $\SETWBmin$ the class of quasi-units
of $\CATB$; it is easy to see that $(\CATB,\SETWBmin)$ satisfies
conditions (BF), so it makes sense to consider the associated bicategory of fractions.
Then we have:

\begin{prop}\label{prop-02}
Let us fix any pair $(\CATA,\SETWA)$ satisfying conditions \emphatic{(BF)}, and any pseudofunctor
$\functor{F}:\CATA\rightarrow\CATB$, such that $\functor{F}_1(\SETWA)\subseteq\SETWBequiv$.
Then for each $i=1,\cdots,5$ the following facts are equivalent:

\begin{itemize}
 \item $\functor{F}$ satisfies condition \emphatic{(\hyperref[A]{A$\,i$})} when $\SETWB:=\SETWBmin$;
 \item $\functor{F}$ satisfies condition \emphatic{(\hyperref[B]{B$\,i$})}.
\end{itemize}
\end{prop}

\begin{proof}
We recall (see~\cite[Lemma~2.5(ii)]{T4}) that the right saturation of $\SETWBmin$ is the class
$\SETWBequiv$ of internal equivalences of $\CATB$.\\

Clearly (\hyperref[B1]{B1}) implies (\hyperref[A1]{A1}) when $\SETWB=\SETWBmin$:
it suffices to take $A'_{\CATB}:=
\functor{F}_0(A_{\CATA})$, $\operatorname{w}^1_{\CATB}:=\id_{\functor{F}_0(A_{\CATA})}$ and
$\operatorname{w}^2_{\CATB}:=e_{\CATB}$ (this makes sense since $\SETWBequiv$ is the right saturation
of $\SETWBmin$). Conversely, let us assume that (\hyperref[A1]{A1}) holds
for $\SETWBmin$;
since $\operatorname{w}^1_{\CATB}$ and $\operatorname{w}^2_{\CATB}$ belong to $\SETWBmin$ and
$\SETWBequiv$ respectively, then $A'_{\CATB}=\functor{F}_0(A_{\CATA})$ and
$\operatorname{w}^2_{\CATB}$ is an internal equivalence from $\functor{F}_0(A_{\CATA})$ to
$A_{\CATB}$, so (\hyperref[B1]{B1}) holds.\\

Let us suppose that (\hyperref[A2]{A2}) holds for $\SETWBmin$, let us fix any pair of
objects $A^1_{\CATA},
A^2_{\CATA}$ and any internal equivalence $e_{\CATB}:\functor{F}_0(A^1_{\CATA})\rightarrow
\functor{F}_0(A^2_{\CATA})$ (i.e.\ any element of the right saturated of $\SETWBmin$) and
let us apply (\hyperref[A2]{A2}) to the following set of data:

\[
\begin{tikzpicture}[xscale=3.0,yscale=-1.2]
    \node (A0_0) at (0, 0) {$\functor{F}_0(A^1_{\CATA})$};
    \node (A0_1) at (1, 0) {$\functor{F}_0(A^1_{\CATA})$};
    \node (A0_2) at (2, 0) {$\functor{F}_0(A^2_{\CATA})$.};
    
    \path (A0_1) edge [->]node [auto,swap] {$\scriptstyle{\id_{\functor{F}_0(A^1_{\CATA})}}$} (A0_0);
    \path (A0_1) edge [->]node [auto] {$\scriptstyle{e_{\CATB}}$} (A0_2);
\end{tikzpicture}
\]

Then there are an object $A^3_{\CATA}$, a pair of morphisms $\operatorname{w}^1_{\CATA}:A^3_{\CATA}
\rightarrow A^1_{\CATA}$ in $\SETWA$ and $\operatorname{w}^2_{\CATA}:A^3_{\CATA}\rightarrow 
A^2_{\CATA}$ in $\SETWAsat$ and a set of data in $\CATB$ as in the internal part of the following
diagram

\[
\begin{tikzpicture}[xscale=1.8,yscale=-0.8]
    \node (A0_2) at (2, 0) {$\functor{F}_0(A^1_{\CATA})$};
    \node (A2_2) at (2, 2) {$A'_{\CATB}$};
    \node (A2_0) at (0, 2) {$\functor{F}_0(A^1_{\CATA})$};
    \node (A2_4) at (4, 2) {$\functor{F}_0(A^2_{\CATA})$,};
    \node (A4_2) at (2, 4) {$\functor{F}_0(A^3_{\CATA})$};

    \node (A2_3) at (2.8, 2) {$\Downarrow\,\gamma^2_{\CATB}$};
    \node (A2_1) at (1.2, 2) {$\Downarrow\,\gamma^1_{\CATB}$};
    
    \path (A4_2) edge [->]node [auto,swap]
      {$\scriptstyle{\functor{F}_1(\operatorname{w}^2_{\CATA})}$} (A2_4);
    \path (A0_2) edge [->]node [auto] {$\scriptstyle{e_{\CATB}}$} (A2_4);
    \path (A2_2) edge [->]node [auto,swap] {$\scriptstyle{\operatorname{z}^1_{\CATB}}$} (A0_2);
    \path (A2_2) edge [->]node [auto] {$\scriptstyle{e'_{\CATB}}$} (A4_2);
    \path (A4_2) edge [->]node [auto]
      {$\scriptstyle{\functor{F}_1(\operatorname{w}^1_{\CATA})}$} (A2_0);
    \path (A0_2) edge [->]node [auto,swap] {$\scriptstyle{\id_{\functor{F}_0(A^1_{\CATA})}}$} (A2_0);
\end{tikzpicture}
\]
such that $\operatorname{z}^1_{\CATB}$ belongs to $\SETWBmin$ and both $\gamma^1_{\CATB}$ and
$\gamma^2_{\CATB}$ are invertible.  Since conditions (BF1), (BF2) and (BF5)
hold for $(\CATB,\SETWBmin)$, then the morphism
$\functor{F}_1(\operatorname{w}^1_{\CATA})\circ e'_{\CATB}$ belongs to $\SETWBmin\subseteq
\SETWBequiv$. Moreover, $\functor{F}_1(\operatorname{w}^1_{\CATA})$ belongs to $\SETWBequiv$ by
hypothesis.
So by~\cite[Proposition~2.11(ii)]{T4} for $(\CATB,\SETWBmin)$ we get that $e'_{\CATB}$ belongs to
$\SETW_{\CATB,\operatorname{min},\operatorname{sat}}=\SETWBequiv$, i.e.\ it is an internal equivalence
of $\CATB$.\\

Since $\operatorname{z}^1_{\CATB}$ belongs to $\SETWBmin$, then $A'_{\CATB}=\functor{F}_0
(A^1_{\CATA})$ and there is an invertible $2$-morphism $\xi_{\CATB}:\operatorname{z}^1_{\CATB}
\Rightarrow\id_{\functor{F}_0(A^1_{\CATA})}$. Then we define a pair of invertible $2$-morphisms:

\[\delta^1_{\CATB}:=\xi_{\CATB}\odot\upsilon_{\operatorname{z}^1_{\CATB}}\odot\Big(\gamma^1_{\CATB}
\Big)^{-1}:\,\,\functor{F}_1(\operatorname{w}^1_{\CATA})\circ e'_{\CATB}
\Longrightarrow\id_{\functor{F}_0(A^1_{\CATA})}\]
and

\[\delta^2_{\CATB}:=\gamma^2_{\CATB}\odot\Big(i_{e_{\CATB}}\ast\xi_{\CATB}^{-1}\Big)\odot
\pi_{e_{\CATB}}^{-1}:\,\,e_{\CATB}
\Longrightarrow\functor{F}_1(\operatorname{w}^2_{\CATA})\circ e'_{\CATB}.\]

This proves that (\hyperref[B2]{B2}) holds.\\

Conversely, let us suppose that (\hyperref[B2]{B2}) holds and let us fix any triple of objects
$A^1_{\CATA},A^2_{\CATA},A_{\CATB}$ and any pair of morphisms $\operatorname{w}^1_{\CATB}:
A_{\CATB}\rightarrow\functor{F}_0(A^1_{\CATA})$ in $\SETWBmin$ and $\operatorname{w}^2_{\CATB}:
A_{\CATB}\rightarrow\functor{F}_0(A^2_{\CATA})$ in $\SETWBequiv$. Since $\operatorname{w}^1_{\CATB}$
belongs to $\SETWBmin$, then $A_{\CATB}=\functor{F}_0(A^1_{\CATA})$ and there is an invertible
$2$-morphism $\xi_{\CATB}:\operatorname{w}^1_{\CATB}\Rightarrow\id_{\functor{F}_0(A^1_{\CATA})}$.
Now let us apply (\hyperref[B2]{B2}) to the internal equivalence $\operatorname{w}^2_{\CATB}:
\functor{F}_0(A^1_{\CATA})\rightarrow\functor{F}_0(A^2_{\CATA})$. Then there are an object
$A^3_{\CATA}$, a pair of morphisms $\operatorname{w}^1_{\CATA}:A^3_{\CATA}\rightarrow A^1_{\CATA}$ in 
$\SETWA$ and $\operatorname{w}^2_{\CATA}:A^3_{\CATA}\rightarrow A^2_{\CATA}$ in $\SETWAsat$, an 
internal equivalence
$e'_{\CATB}:\functor{F}_0(A^1_{\CATA})\rightarrow\functor{F}_0(A^3_{\CATA})$ and invertible
$2$-morphisms $\delta^1_{\CATB}:\functor{F}_1(\operatorname{w}^1_{\CATA})\circ e'_{\CATB}\Rightarrow
\id_{\functor{F}_0(A^1_{\CATA})}$ and $\delta^2_{\CATB}:\operatorname{w}^2_{\CATB}
\Rightarrow\functor{F}_1(\operatorname{w}^2_{\CATA})\circ e'_{\CATB}$. Then the following diagram
proves that (\hyperref[A2]{A2}) holds.

\[
\begin{tikzpicture}[xscale=2.8,yscale=-0.8]
    \node (A0_2) at (2, 0) {$\functor{F}_0(A^1_{\CATA})=A_{\CATB}$};
    \node (A2_2) at (2, 2) {$\functor{F}_0(A^1_{\CATA})$};
    \node (A2_0) at (0, 2) {$\functor{F}_0(A^1_{\CATA})$};
    \node (A2_4) at (4, 2) {$\functor{F}_0(A^2_{\CATA})$.};
    \node (A4_2) at (2, 4) {$\functor{F}_0(A^3_{\CATA})$};
    
    \node (A2_3) at (2.8, 2) {$\Downarrow\,\delta^2_{\CATB}\odot\pi_{\operatorname{w}^2_{\CATB}}$};
    \node (A2_1) at (1.1, 2) {$\Downarrow\,\Big(\delta^1_{\CATB}\Big)^{-1}\odot
      \xi_{\CATB}\odot\pi_{\operatorname{w}^1_{\CATB}}$};
    
    \path (A4_2) edge [->]node [auto,swap]
      {$\scriptstyle{\functor{F}_1(\operatorname{w}^2_{\CATA})}$} (A2_4);
    \path (A0_2) edge [->]node [auto] {$\scriptstyle{\operatorname{w}^2_{\CATB}}$} (A2_4);
    \path (A2_2) edge [->]node [auto,swap] {$\scriptstyle{\id_{\functor{F}_0(A^1_{\CATA})}}$} (A0_2);
    \path (A2_2) edge [->]node [auto] {$\scriptstyle{e'_{\CATB}}$} (A4_2);
    \path (A4_2) edge [->]node [auto]
      {$\scriptstyle{\functor{F}_1(\operatorname{w}^1_{\CATA})}$} (A2_0);
    \path (A0_2) edge [->]node [auto,swap] {$\scriptstyle{\operatorname{w}^1_{\CATB}}$} (A2_0);
\end{tikzpicture}
\]

Now let us suppose that (\hyperref[B3]{B3}) holds and let us fix any pair of objects $B_{\CATA},
A_{\CATB}$ and any morphism $f_{\CATB}:A_{\CATB}\rightarrow\functor{F}_0(B_{\CATA})$. Then there are
data $(A_{\CATA},f_{\CATA},e_{\CATB},\alpha_{\CATB})$ as in (\hyperref[B3]{B3}). So
(\hyperref[A3]{A3}) for $\SETWBmin$ is verified by the data $(A_{\CATA},f_{\CATA},A_{\CATB},
\id_{A_{\CATB}},e_{\CATB},\alpha_{\CATB}\odot\pi_{f_{\CATB}})$.\\

Conversely, let us suppose that (\hyperref[A3]{A3}) holds for $\SETWBmin$ and let us fix any set 
of data $B_{\CATA},
A_{\CATB}$ and $f_{\CATB}$ as before. Then there is a set of data $(A_{\CATA},f_{\CATA},A'_{\CATB},
\operatorname{v}^1_{\CATB},\operatorname{v}^2_{\CATB},\alpha_{\CATB})$ as
in (\hyperref[A3]{A3}). In particular,
$\operatorname{v}^1_{\CATB}$ belongs to $\SETWBmin$, hence $A'_{\CATB}=A_{\CATB}$ and there is an
invertible $2$-morphism $\xi_{\CATB}:\operatorname{v}^1_{\CATB}\Rightarrow\id_{A_{\CATB}}$;
moreover $\operatorname{v}^2_{\CATB}$ belongs to the saturation of $\SETWBmin$, i.e.\ it belongs
to $\SETWBequiv$. Then the
set of data $(A_{\CATA},f_{\CATA},\operatorname{v}^2_{\CATB},\alpha_{\CATB}\odot(i_{f_{\CATB}}
\ast\xi_{\CATB}^{-1})\odot\pi^{-1}_{f_{\CATB}})$ proves that (\hyperref[B3]{B3}) holds.\\

Clearly (\hyperref[A4]{A4}) for $\SETWBmin$ implies (\hyperref[B4]{B4}) in the special case when
$A'_{\CATB}=
\functor{F}_0(A_{\CATA})$ and $\operatorname{z}_{\CATB}=\id_{\functor{F}_0(A_{\CATA})}$
(since this morphism belongs to $\SETWBmin$). Conversely, let us assume (\hyperref[B4]{B4})
and let us fix any set of data $(A_{\CATA},f^1_{\CATA},f^2_{\CATA},\gamma^1_{\CATA},\gamma^2_{\CATA},
A'_{\CATB},\operatorname{z}_{\CATB})$ as in (\hyperref[A4]{A4}). In particular,
$\operatorname{z}_{\CATB}$ belongs to $\SETWBmin$, hence $A'_{\CATB}=\functor{F}_0
(A_{\CATA})$ and there is an invertible $2$-morphism $\xi_{\CATB}:\operatorname{z}_{\CATB}
\Rightarrow\id_{\functor{F}_0(A_{\CATA})}$. Since $\functor{F}_2(\gamma^1_{\CATA})\ast
i_{\operatorname{z}_{\CATB}}=\functor{F}_2(\gamma^2_{\CATA})\ast i_{\operatorname{z}_{\CATB}}$, then
using $\xi_{\CATB}$ we get that $\functor{F}_2(\gamma^1_{\CATA})=\functor{F}_2(\gamma^2_{\CATA})$. So
we can use (\hyperref[B4]{B4}) in order to conclude that (\hyperref[A4]{A4}) holds for $\SETWBmin$.\\

Now let us suppose that (\hyperref[B5]{B5}) holds and let us fix any set of data $(A_{\CATA},
B_{\CATA},A_{\CATB},f^1_{\CATA},$ $f^2_{\CATA},\operatorname{v}_{\CATB},\alpha_{\CATB})$ as in
(\hyperref[A5]{A5}) for $\SETWBmin$. In particular, $\operatorname{v}_{\CATB}$ belongs to
$\SETWBmin$, hence
$A_{\CATB}=\functor{F}_0(A_{\CATA})$ and there is an invertible $2$-morphism $\xi_{\CATB}:
\operatorname{v}_{\CATB}\Rightarrow\id_{\functor{F}_0(A_{\CATA})}$. Hence we can define a
$2$-morphism

\[\overline{\alpha}_{\CATB}:=\pi_{\functor{F}_1(f^2_{\CATA})}\odot\Big(i_{\functor{F}_1
(f^2_{\CATA})}\ast\xi_{\CATB}\Big)\odot\,\alpha_{\CATB}\,\odot\Big(i_{\functor{F}_1(f^1_{\CATA})}
\ast \xi_{\CATB}^{-1}\Big)\odot\pi^{-1}_{\functor{F}_1(f^1_{\CATA})}:\,\,\functor{F}_1
(f^1_{\CATA})\Rightarrow\functor{F}_1(f^2_{\CATA}).\]

Then by (\hyperref[B5]{B5}) for $\overline{\alpha}_{\CATB}$,
there are an object $A'_{\CATA}$, a morphism $\operatorname{v}_{\CATA}:
A'_{\CATA}\rightarrow A_{\CATA}$ in $\SETWA$, a $2$-morphism $\alpha_{\CATA}:f^1_{\CATA}\circ
\operatorname{v}_{\CATA}\Rightarrow f^2_{\CATA}\circ\operatorname{v}_{\CATA}$, such that

\[\overline{\alpha}_{\CATB}\ast i_{\functor{F}_1(\operatorname{v}_{\CATA})}=
\psi^{\functor{F}}_{f^2_{\CATA},\operatorname{v}_{\CATA}}\odot\functor{F}_2(\alpha_{\CATA})\odot
\Big(\psi^{\functor{F}}_{f^1_{\CATA},\operatorname{v}_{\CATA}}\Big)^{-1}.\]

Using the definition of $\overline{\alpha}_{\CATB}$ and the coherence axioms on the bicategory
$\CATB$, this implies that

\begin{gather*}
\alpha_{\CATB}\ast i_{\functor{F}_1(\operatorname{v}_{\CATA})}=\thetaa{\functor{F}_1(f^2_{\CATA})}
 {\operatorname{v}_{\CATB}}{\functor{F}_1(\operatorname{v}_{\CATA})}\odot\Big\{i_{\functor{F}_1
 (f^2_{\CATA})}\ast\Big[\Big(\xi_{\CATB}^{-1}\ast i_{\functor{F}_1(\operatorname{v}_{\CATA})}\Big)
 \odot \\
\odot\upsilon_{\functor{F}_1(\operatorname{v}_{\CATA})}^{-1}\odot\pi_{\functor{F}_1
 (\operatorname{v}_{\CATA})}\Big]\Big\}\odot\thetab{\functor{F}_1(f^2_{\CATA})}
 {\functor{F}_1(\operatorname{v}_{\CATA})}{\id_{\functor{F}_0(A'_{\CATA})}}\odot \\
\odot\Big\{\Big[\psi^{\functor{F}}_{f^2_{\CATA},\operatorname{v}_{\CATA}}\odot
 \functor{F}_2(\alpha_{\CATA})\odot\Big(\psi^{\functor{F}}_{f^1_{\CATA},\operatorname{v}_{\CATA}}
 \Big)^{-1}\Big]\ast i_{\id_{\functor{F}_0(A'_{\CATA})}}\Big\}
 \odot\thetaa{\functor{F}_1(f^1_{\CATA})}{\functor{F}_1(\operatorname{v}_{\CATA})}
 {\id_{\functor{F}_0(A'_{\CATA})}}\odot \\
\odot\Big\{i_{\functor{F}_1(f^1_{\CATA})}\ast\Big[\pi^{-1}_{\functor{F}_1
 (\operatorname{v}_{\CATA})}\odot\upsilon_{\functor{F}_1(\operatorname{v}_{\CATA})}
 \odot\Big(\xi_{\CATB}\ast i_{\functor{F}_1
 (\operatorname{v}_{\CATA})}\Big)\Big]\Big\}\odot\thetab{\functor{F}_1(f^1_{\CATA})}
 {\operatorname{v}_{\CATB}}{\functor{F}_1(\operatorname{v}_{\CATA})}.
\end{gather*}

Then (\hyperref[A5]{A5}) for $\SETWBmin$ is satisfied if we set $A'_{\CATB}:=\functor{F}_0
(A'_{\CATA})$,
$\operatorname{z}_{\CATB}:=\id_{\functor{F}_0(A'_{\CATA})}$, $\operatorname{z}'_{\CATB}:=
\functor{F}_1(\operatorname{v}_{\CATA})$ and

\[\sigma_{\CATB}:=\Big(\xi_{\CATB}^{-1}\ast i_{\functor{F}_1(\operatorname{v}_{\CATA})}\Big)
\odot\upsilon^{-1}_{\functor{F}_1(\operatorname{v}_{\CATA})}\odot\pi_{\functor{F}_1
(\operatorname{v}_{\CATA})}.\]

Conversely, let us suppose that (\hyperref[A5]{A5}) holds for $\SETWBmin$ and let us fix
any set of data
$(A_{\CATA},B_{\CATA},f^1_{\CATA},$ $f^2_{\CATA},\alpha_{\CATB})$ as in (\hyperref[B5]{B5}). Then let
us apply (\hyperref[A5]{A5}) in the case when we fix $A_{\CATB}:=\functor{F}_0(A_{\CATA})$,
$\operatorname{v}_{\CATB}:=\id_{\functor{F}_0(A_{\CATA})}$ and we replace $\alpha_{\CATB}$ with
$\pi_{\functor{F}_1(f^2_{\CATA})}^{-1}\odot\alpha_{\CATB}\odot\pi_{\functor{F}_1(f^1_{\CATA})}$.
Then there are a pair of objects
$A'_{\CATA},A'_{\CATB}$, a triple of morphisms $\operatorname{v}_{\CATA}:A'_{\CATA}\rightarrow
A_{\CATA}$ in $\SETWA$, $\operatorname{z}_{\CATB}:A'_{\CATB}\rightarrow\functor{F}_0(A'_{\CATA})$
in $\SETWBmin$ and $\operatorname{z}'_{\CATB}:A'_{\CATB}\rightarrow\functor{F}_0(A_{\CATA})$, a
$2$-morphism $\alpha_{\CATA}:f^1_{\CATA}\circ\operatorname{v}_{\CATA}\Rightarrow f^2_{\CATA}\circ
\operatorname{v}_{\CATA}$ and an invertible $2$-morphism $\sigma_{\CATB}:\functor{F}_1
(\operatorname{v}_{\CATA})\circ\operatorname{z}_{\CATB}\Rightarrow\id_{\functor{F}_0(A_{\CATA})}
\circ\operatorname{z}'_{\CATB}$, such that

\begin{gather}
\nonumber \Big(\pi^{-1}_{\functor{F}_1(f^2_{\CATA})}\odot\alpha_{\CATB}\odot\pi_{\functor{F}_1
 (f^1_{\CATA})}\Big)\ast i_{\operatorname{z}'_{\CATB}}=
 \thetaa{\functor{F}_1(f^2_{\CATA})}{\id_{\functor{F}_0(A_{\CATA})}}{\operatorname{z}'_{\CATB}}
 \odot\Big(i_{\functor{F}_1(f^2_{\CATA})}\ast\sigma_{\CATB}\Big)\odot \\
\nonumber \odot\,\thetab{\functor{F}_1(f^2_{\CATA})}{\functor{F}_1(\operatorname{v}_{\CATA})}
 {\operatorname{z}_{\CATB}}\odot\Big(\Big(\psi^{\functor{F}}_{f^2_{\CATA},\operatorname{v}_{\CATA}}
 \odot\functor{F}_2(\alpha_{\CATA})\odot\Big(\psi^{\functor{F}}_{f^1_{\CATA},\operatorname{v}_{\CATA}}
 \Big)^{-1}\Big)\ast i_{\operatorname{z}_{\CATB}}\Big)\odot \\
\label{eq-11} \odot\,\thetaa{\functor{F}_1(f^1_{\CATA})}
 {\functor{F}_1(\operatorname{v}_{\CATA})}{\operatorname{z}_{\CATB}}\odot
 \Big(i_{\functor{F}_1(f^1_{\CATA})}\ast\sigma_{\CATB}^{-1}\Big)\odot
 \thetab{\functor{F}_1(f^1_{\CATA})}{\id_{\functor{F}_0(A_{\CATA})}}{\operatorname{z}'_{\CATB}}.
\end{gather}
 
Since $\operatorname{z}_{\CATB}$ belongs to $\SETWBmin$, then $A'_{\CATB}=\functor{F}_0(A'_{\CATA})$
and there is an invertible $2$-morphism $\xi_{\CATB}:\operatorname{z}_{\CATB}\Rightarrow
\id_{\functor{F}_0(A'_{\CATA})}$. Using the coherence axioms for the bicategory $\CATB$ several times,
we have:

\begin{gather*}
\alpha_{\CATB}\ast i_{\functor{F}_1(\operatorname{v}_{\CATA})}= \\
= \Big(i_{\functor{F}_1(f^2_{\CATA})}\ast\Big(\pi_{\functor{F}_1(\operatorname{v}_{\CATA})}\odot
 \Big(i_{\functor{F}_1(\operatorname{v}_{\CATA})}\ast\xi_{\CATB}\Big)\Big)\Big)\odot
 \Big(i_{\functor{F}_1(f^2_{\CATA})}\ast\sigma^{-1}_{\CATB}\Big)\odot \\
\odot\Big(\alpha_{\CATB}\ast i_{\id_{\functor{F}_0(A_{\CATA})}\circ\operatorname{z}'_{\CATB}}
 \Big)\odot \\
\odot\Big(i_{\functor{F}_1(f^1_{\CATA})}\ast\sigma_{\CATB}\Big)\odot\Big(i_{\functor{F}_1
 (f^1_{\CATA})}\ast\Big(\Big(i_{\functor{F}_1(\operatorname{v}_{\CATA})}
 \ast\xi_{\CATB}^{-1}\Big)\odot\pi^{-1}_{\functor{F}_1(\operatorname{v}_{\CATA})}\Big)\Big)= \\
= \Big(i_{\functor{F}_1(f^2_{\CATA})}\ast\Big(\pi_{\functor{F}_1(\operatorname{v}_{\CATA})}\odot
 \Big(i_{\functor{F}_1(\operatorname{v}_{\CATA})}\ast\xi_{\CATB}\Big)\Big)\Big)\odot
 \Big(i_{\functor{F}_1(f^2_{\CATA})}\ast\sigma^{-1}_{\CATB}\Big)\odot \\
\odot\,\thetab{\functor{F}_1(f^2_{\CATA})}{\id_{\functor{F}_0(A_{\CATA})}}{\operatorname{z}'_{\CATB}}
 \odot\Big(\Big(\pi^{-1}_{\functor{F}_1(f^2_{\CATA})}\odot\alpha_{\CATB}\odot
 \pi_{\functor{F}_1(f^1_{\CATA})}\Big)\ast i_{\operatorname{z}'_{\CATB}}\Big)\odot
 \thetaa{\functor{F}_1(f^1_{\CATA})}{\id_{\functor{F}_0(A_{\CATA})}}{\operatorname{z}'_{\CATB}}
 \odot \\
\odot\Big(i_{\functor{F}_1(f^1_{\CATA})}\ast\sigma_{\CATB}\Big)\odot\Big(i_{\functor{F}_1
 (f^1_{\CATA})}\ast\Big(\Big(i_{\functor{F}_1(\operatorname{v}_{\CATA})}
 \ast\xi_{\CATB}^{-1}\Big)\odot\pi^{-1}_{\functor{F}_1(\operatorname{v}_{\CATA})}\Big)\Big)
 \stackrel{\eqref{eq-11}}{=} \\
\stackrel{\eqref{eq-11}}{=} \Big(i_{\functor{F}_1(f^2_{\CATA})}\ast\Big(\pi_{\functor{F}_1
 (\operatorname{v}_{\CATA})}\odot\Big(i_{\functor{F}_1(\operatorname{v}_{\CATA})}\ast\xi_{\CATB}\Big)
 \Big)\Big)\odot \\
\odot\,\thetab{\functor{F}_1(f^2_{\CATA})}{\functor{F}_1(\operatorname{v}_{\CATA})}
 {\operatorname{z}_{\CATB}}\odot\Big(\Big(\psi^{\functor{F}}_{f^2_{\CATA},
 \operatorname{v}_{\CATA}}\odot
 \functor{F}_2(\alpha_{\CATA})\odot\Big(\psi^{\functor{F}}_{f^1_{\CATA},\operatorname{v}_{\CATA}}
 \Big)^{-1}\Big)\ast i_{\operatorname{z}_{\CATB}}\Big)\odot \\
\odot\,\thetaa{\functor{F}_1(f^1_{\CATA})}{\functor{F}_1(\operatorname{v}_{\CATA})}
 {\operatorname{z}_{\CATB}}\odot\Big(i_{\functor{F}_1(f^1_{\CATA})}\ast\Big(\Big(
 i_{\functor{F}_1(\operatorname{v}_{\CATA})}\ast\xi_{\CATB}^{-1}\Big)\odot
 \pi^{-1}_{\functor{F}_1(\operatorname{v}_{\CATA})}\Big)\Big)= \\
=\Big(i_{\functor{F}_1(f^2_{\CATA})}\ast\pi_{\functor{F}_1(\operatorname{v}_{\CATA})}\Big)
 \odot\thetab{\functor{F}_1
 (f^1_{\CATA})}{\functor{F}_1(\operatorname{v}_{\CATA})}{\id_{\functor{F}_0(A_{\CATA})}}\odot
 \Big(\Big(\psi^{\functor{F}}_{f^2_{\CATA},\operatorname{v}_{\CATA}}\odot
 \functor{F}_2(\alpha_{\CATA})\,\odot \\
\odot\Big(\psi^{\functor{F}}_{f^1_{\CATA},
 \operatorname{v}_{\CATA}}\Big)^{-1}\Big)\ast i_{\id_{\functor{F}_0(A_{\CATA})}}\Big)\odot
 \thetaa{\functor{F}_1(f^1_{\CATA})}{\functor{F}_1(\operatorname{v}_{\CATA})}{\id_{\functor{F}_0
 (A_{\CATA})}}\odot\Big(
 i_{\functor{F}_1(f^1_{\CATA})}\ast\pi^{-1}_{\functor{F}_1(\operatorname{v}_{\CATA})}\Big)= \\\
=\psi^{\functor{F}}_{f^2_{\CATA},\operatorname{v}_{\CATA}}\odot\functor{F}_2
 (\alpha_{\CATA})\odot\Big(\psi^{\functor{F}}_{f^1_{\CATA},\operatorname{v}_{\CATA}}\Big)^{-1}.
\end{gather*}

So we have proved that (\hyperref[B5]{B5}) holds.
\end{proof}

Then we are ready to give the following:

\begin{proof}[Proof of Theorem~\ref{theo-04}.]
By~\cite[Lemma~2.5(iii)]{T4} for $\CATC:=\CATB$, the pseudofunctor

\[\functor{U}_{\SETWBmin}:\CATB\longrightarrow\CATB\left[\SETWBmininv\right]\]
is an equivalence of bicategories. Let us fix any pair $(\overline{\functor{G}},
\overline{\kappa})$ as in Theorem~\ref{theo-03}(2); then it makes sense to set

\[\functor{L}:=\functor{U}_{\SETWBmin}\circ\overline{\functor{G}}:\,\,\CATA\Big[
\SETWAinv\Big]\longrightarrow\CATB\Big[\SETWBmininv\Big]\]
and

\[\rho:=\thetaa{\functor{U}_{\SETWBmin}}{\overline{\functor{G}}}{\functor{U}_{\SETWA}}\odot
\Big(i_{\functor{U}_{\SETWBmin}}\ast\overline{\kappa}\Big):\,\,
\functor{U}_{\SETWBmin}\circ\functor{F}\Longrightarrow\functor{L}\circ\functor{U}_{\SETWA}.\]

By hypothesis the pseudofunctor $\CATA\rightarrow\operatorname{Cyl}(\CATB)$ associated to
$\overline{\kappa}$ sends each morphism of $\SETWA$ to an internal equivalence, hence so does
the pseudofunctor $\CATA\rightarrow\operatorname{Cyl}\left(\CATB\left[\SETWBmininv\right]\right)$
associated to $i_{\functor{U}_{\SETWBmin}}\ast\overline{\kappa}$, hence so does the
pseudofunctor associated to $\rho$. So the pair $(\functor{L},\rho)$ satisfies
Theorem~\ref{theo-01}(iv) when $\SETWB:=\SETWBmin$. Therefore, by Theorem~\ref{theo-02}, $\functor{L}$
is an equivalence of bicategories if and only if $\functor{F}$ satisfies conditions
(\hyperref[A1]{A1}) -- (\hyperref[A5]{A5}) when $\SETWB:=\SETWBmin$. So by Proposition~\ref{prop-02},
$\functor{L}$ is an equivalence of bicategories if and only if $\functor{F}$ satisfies conditions
(\hyperref[B1]{B1}) -- (\hyperref[B5]{B5}).\\

Moreover, since $\functor{U}_{\SETWBmin}$ is an equivalence of bicategories, then
$\overline{\functor{G}}$ is an equivalence of bicategories if and only if
$\functor{L}$ is an equivalence of bicategories, i.e.\ if and only if and only if 
$\functor{F}$ satisfies conditions (\hyperref[B1]{B1}) -- (\hyperref[B5]{B5}).
\end{proof}

Lastly, we are able to give also the following proof:

\begin{proof}[Proof of Corollary~\ref{cor-01}.]
Let us suppose that (a) holds, i.e.\ let us suppose that there is 
an equivalence of bicategories $\overline{\functor{G}}:\CATA
\left[\SETWAinv\right]\rightarrow\CATB$. Then we define $\functor{F}:=\overline{\functor{G}}\circ
\functor{U}_{\SETWA}:\CATA\rightarrow\CATB$. Since $\functor{U}_{\SETWA}$ sends each morphism
of $\SETWA$ to an internal equivalence, so does $\functor{F}$. Moreover, we set

\[\overline{\kappa}:=i_{\overline{\functor{G}}\circ\,\functor{U}_{\SETWA}}:\,\,\functor{F}
\Longrightarrow\overline{\functor{G}}\circ\functor{U}_{\SETWA}.\]

So the pair $(\overline{\functor{G}},
\overline{\kappa})$ satisfies the conditions of Theorem~\ref{theo-03}(2), hence by
Theorem~\ref{theo-04} we get that $\functor{F}$ satisfies conditions (\hyperref[B1]{B1}) --
(\hyperref[B5]{B5}). So we have proved that (a) implies (b).\\

Conversely, let us suppose that there is a pseudofunctor $\functor{F}:\CATA\rightarrow\CATB$, such
that $\functor{F}_1(\SETWA)\subseteq\SETWBequiv$. Therefore,
there are a pseudofunctor $\overline{\functor{G}}:\CATA\left[\SETWAinv\right]\rightarrow\CATB$ and
a pseudonatural equivalence $\overline{\kappa}$ as in Theorem~\ref{theo-03}(2).
If we further assume that $\functor{F}$ satisfies conditions (\hyperref[B1]{B1}) --
(\hyperref[B5]{B5}), then by Theorem~\ref{theo-04}
we conclude that $\overline{\functor{G}}$ is an equivalence of bicategories, so we have proved
that (b) implies (a).
\end{proof}

\section*{Appendix}
Let us fix any pseudofunctor $\functor{F}:\CATA\rightarrow\CATB$ sending each morphism of $\SETWA$
to an internal equivalence. As we mentioned in the introduction, Theorem~\ref{theo-04} fixed some
problems appearing in~\cite[Proposition~24]{Pr}. As Theorem~\ref{theo-04},
also~\cite[Proposition~24]{Pr} deals with necessary and sufficient conditions such that $\functor{F}$
induces an equivalence of bicategories from $\CATA\left[\SETWAinv\right]$
to $\CATB$. However, such a statement is partially incorrect. The existence a problem in such a
statement was
already mentioned in~\cite[Remark before Proposition~1.8.2]{R1} and~\cite[Remark after
Proposition~6.3]{R2}, but with no further details. For clarity of exposition, in the
next lines we describe where the problem lies exactly. We recall that~\cite[Proposition~24]{Pr}
states the following: having fixed any $\functor{F}:\CATA\rightarrow\CATB$ such that
$\functor{F}_1(\SETWA)\subseteq\SETWBequiv$, one should have an induced equivalence
from $\CATA\left[\SETWAinv\right]$ to $\CATB$ if and only if $\functor{F}$ satisfies the following
$3$ conditions.

\begin{enumerate}[({EF}1)]
 \item\label{EF1} $\functor{F}$ is essentially surjective, i.e.\ for each object $A_{\CATB}$ there are
  an object $A_{\CATA}$ and an \emph{isomorphism} $t_{\CATB}:\functor{F}_0(A_{\CATA})\rightarrow
  A_{\CATB}$.
 \item\label{EF2} For each pair of objects $A_{\CATA},B_{\CATA}$ and for each morphism $f_{\CATB}:
  \functor{F}_0(A_{\CATA})\rightarrow\functor{F}_0(B_{\CATA})$, there are an object $A'_{\CATA}$,
  a morphism $f_{\CATA}:A'_{\CATA}\rightarrow B_{\CATA}$, a morphism $\operatorname{w}_{\CATA}:
  A'_{\CATA}\rightarrow A_{\CATA}$ in $\SETWA$ and an invertible $2$-morphism $\alpha_{\CATB}:
  \functor{F}_1(f_{\CATA})\Rightarrow f_{\CATB}\circ\functor{F}_1(\operatorname{w}_{\CATA})$;.
 \item\label{EF3} For each pair of objects $A_{\CATA},B_{\CATA}$, for each pair of morphisms
  $f^1_{\CATA},f^2_{\CATA}:A_{\CATA}\rightarrow B_{\CATA}$ and for each $2$-morphism $\alpha_{\CATB}:
  \functor{F}_1(f^1_{\CATA})\Rightarrow\functor{F}_1(f^2_{\CATA})$, there is a \emph{unique}
  $2$-morphism $\alpha_{\CATA}:f^1_{\CATA}\Rightarrow f^2_{\CATA}$, such that $\functor{F}_2
  (\alpha_{\CATA})=\alpha_{\CATB}$.
\end{enumerate}

One can easily see that such conditions are \emph{too strong than necessary}: clearly
(\hyperref[EF1]{EF1}) must be relaxed by allowing not only isomorphisms but also
internal equivalences. But most importantly, (\hyperref[EF3]{EF3}) is definitely too much
restrictive. This can be seen easily with a toy example as follows: let us consider a $2$-category
$\CATC$ with only $2$ objects $A,B$, only a non-trivial morphism $\operatorname{v}:A\rightarrow B$
and only a non-trivial $2$-morphism $\gamma:\id_B\Rightarrow\id_B$ (together with all
the necessary identities
and $2$-identities). This clearly satisfies all the axioms of a $2$-category and we have $\gamma\neq
i_{\id_B}$ but $\gamma\ast i_{\operatorname{v}}=i_{\operatorname{v}}$, since there are
no non-trivial $2$-morphisms over $\operatorname{v}$. Then we set $\SETW:=
\{\operatorname{v},\id_A,\id_B\}$ and it is easy to prove that $(\CATC,\SETW)$ satisfies axioms
(BF). Then following the constructions in~\cite[\S~2.3 and~2.4]{Pr}, we have

\begin{gather*}
\functor{U}_{\SETW,2}(\gamma)=\Big[B,\id_B,\id_B,i_{\id_B},\gamma\Big]=\Big[A,\operatorname{v},
\operatorname{v},i_{\operatorname{v}},\gamma\ast i_{\operatorname{v}}\Big]= \\
=\Big[A,\operatorname{v},\operatorname{v},i_{\operatorname{v}},i_{\operatorname{v}}\Big]=
\Big[B,\id_B,\id_B,i_{\id_B},i_{\id_B}\Big]=\functor{U}_{\SETW,2}(i_{\id_B}).
\end{gather*}

So the pseudofunctor $\functor{F}:=\functor{U}_{\SETW}:\CATC\rightarrow\CATC\left[\SETWinv\right]$
does not satisfy the uniqueness condition of (\hyperref[EF3]{EF3}).
However, if we consider the pair $(\overline{\functor{G}},
\overline{\kappa})$ given by $\overline{\functor{G}}:=\id_{\CATC\left[\SETWinv\right]}$ and
$\overline{\kappa}:=i_{\functor{U}_{\SETW}}:\functor{U}_{\SETW}\Rightarrow\id_{\CATC\left[\SETWinv
\right]}\circ\,\functor{U}_{\SETW}$, then such a pair satisfies Theorem~\ref{theo-03}(2) and is such
that $\overline{\functor{G}}$ is an equivalence of bicategories. This proves that \emph{condition}
(\hyperref[EF3]{EF3}) \emph{is not necessary}, hence that~\cite[Proposition~24]{Pr} has some flaw.\\

Actually, conditions (\hyperref[EF1]{EF1}) -- (\hyperref[EF3]{EF3}) are \emph{sufficient} in
order to have that $\overline{\functor{G}}$ is an equivalence of bicategories.
We are going to prove this in the remaining part of this Appendix.

\begin{lem}\label{lem-13}
Let us fix any pair $(\CATA,\SETWA)$ satisfying conditions \emphatic{(BF)}, any
bicategory $\CATB$ and any pseudofunctor $\functor{F}:\CATA\rightarrow\CATB$, such that:

\begin{itemize}
 \item $\functor{F}$ sends each morphism of $\SETWA$ to an internal equivalence of $\CATB$;
 \item $\functor{F}$ satisfies conditions \emphatic{(\hyperref[EF2]{EF2})} and
 \emphatic{(\hyperref[EF3]{EF3})}.
\end{itemize}

Moreover, let us fix any morphism $f_{\CATA}:B_{\CATA}\rightarrow A_{\CATA}$, such that 
$\functor{F}_1(f_{\CATA})$ is an internal equivalence in $\CATB$. Then
there are an object $C_{\CATA}$ and a morphism $g_{\CATA}:C_{\CATA}\rightarrow
B_{\CATA}$, such that $f_{\CATA}\circ g_{\CATA}$ belongs to $\SETWA$ and $\functor{F}_1(g_{\CATA})$ is
an internal equivalence.
\end{lem}

\begin{proof}
Since $\functor{F}_1(f_{\CATA})$ is an internal equivalence, then there are an internal equivalence
$\overline{e}_{\CATB}:\functor{F}_0(A_{\CATA})\rightarrow\functor{F}_0(B_{\CATA})$ and an invertible
$2$-morphism $\xi_{\CATB}:\functor{F}_1(f_{\CATA})\circ\overline{e}_{\CATB}\Rightarrow
\id_{\functor{F}_0(A_{\CATA})}$. By applying (\hyperref[EF2]{EF2}) to $\functor{F}_1(f_{\CATA})\circ
\overline{e}_{\CATB}:\functor{F}_0(A_{\CATA})\rightarrow\functor{F}_0(A_{\CATA})$, we get
an object $A'_{\CATA}$, a morphism $m_{\CATA}:A'_{\CATA}\rightarrow A_{\CATA}$, a morphism
$\operatorname{u}_{\CATA}:A'_{\CATA}\rightarrow A_{\CATA}$ in $\SETWA$ and an invertible
$2$-morphism $\gamma_{\CATB}:\functor{F}_1(m_{\CATA})\Rightarrow(\functor{F}_1(f_{\CATA})\circ
\overline{e}_{\CATB})\circ\functor{F}_1(\operatorname{u}_{\CATA})$. Then we consider the following
invertible $2$-morphism

\[\eta_{\CATB}:=\upsilon_{\functor{F}_1(\operatorname{u}_{\CATA})}\odot\Big(\xi_{\CATB}
\ast i_{\functor{F}_1(\operatorname{u}_{\CATA})}\Big)\odot\gamma_{\CATB}:\,\,
\functor{F}_1(\operatorname{m}_{\CATA})\Longrightarrow\functor{F}_1(u_{\CATA}).\]

By (\hyperref[EF3]{EF3}), there is a unique $2$-morphism $\eta_{\CATA}:
\operatorname{m}_{\CATA}\Rightarrow u_{\CATA}$, such that $\functor{F}_2(\eta_{\CATA})=\eta_{\CATB}$.
Again by (\hyperref[EF3]{EF3}) we get that $\eta_{\CATA}$ is invertible since
$\eta_{\CATB}$ is so. Since $\operatorname{u}_{\CATA}$ belongs to $\SETWA$ by construction, then by
(BF5) for $(\CATA,\SETWA)$ we conclude that also $m_{\CATA}$ belongs to $\SETWA$.\\

Now we apply again (\hyperref[EF2]{EF2}) to the morphism $\overline{e}_{\CATB}\circ\functor{F}_1
(\operatorname{u}_{\CATA}):\functor{F}_0(A'_{\CATA})\rightarrow\functor{F}_0(B_{\CATA})$. So
there are an object $C_{\CATA}$, a morphism $g_{\CATA}:C_{\CATA}\rightarrow B_{\CATA}$, a morphism
$\operatorname{z}_{\CATA}:C_{\CATA}\rightarrow A'_{\CATA}$ in $\SETWA$ and an invertible
$2$-morphism $\delta_{\CATB}:\functor{F}_1(g_{\CATA})\Rightarrow(\overline{e}_{\CATB}\circ
\functor{F}_1(\operatorname{u}_{\CATA}))\circ\functor{F}_1(\operatorname{z}_{\CATA})$. Then we
consider the following invertible $2$-morphism:

\begin{gather*}
\sigma_{\CATB}:=\Big(\psi^{\functor{F}}_{\operatorname{m}_{\CATA},\operatorname{z}_{\CATA}}
 \Big)^{-1}\odot\Big(\gamma_{\CATB}^{-1}\ast i_{\functor{F}_1(\operatorname{z}_{\CATA})}\Big)\odot
 \Big(\thetaa{\functor{F}_1(f_{\CATA})}{\overline{e}_{\CATB}}{\functor{F}_1(\operatorname{u}_{\CATA})}
 \ast i_{\functor{F}_1(\operatorname{z}_{\CATA})}\Big)\odot \\
\odot\,\thetaa{\functor{F}_1(f_{\CATA})}{\overline{e}_{\CATB}\circ\functor{F}_1
 (\operatorname{u}_{\CATA})}{\functor{F}_1(\operatorname{z}_{\CATA})}\odot\Big(i_{\functor{F}_1
 (f_{\CATA})}\ast\delta_{\CATB}\Big)\odot\psi^{\functor{F}}_{f_{\CATA},g_{\CATA}}: \\
\functor{F}_1(f_{\CATA}\circ g_{\CATA})\Longrightarrow
 \functor{F}_1(m_{\CATA}\circ\operatorname{z}_{\CATA}).
\end{gather*}

By (\hyperref[EF3]{EF3}), there is a unique $2$-morphism $\sigma_{\CATA}:f_{\CATA}\circ
g_{\CATA}\Rightarrow m_{\CATA}\circ\operatorname{z}_{\CATA}$, such that $\functor{F}_2(\sigma_{\CATA})
=\sigma_{\CATB}$. Moreover, again by (\hyperref[EF3]{EF3}) we get that
$\sigma_{\CATA}$ is invertible. Now both $m_{\CATA}$ and $\operatorname{z}_{\CATA}$ belong to
$\SETWA$, hence by (BF2)
for $(\CATA,\SETWA)$, so does their composition. Then by (BF5) applied to $\sigma_{\CATA}$, also
the morphism $f_{\CATA}\circ g_{\CATA}$ belongs to $\SETWA$.\\

Since $f_{\CATA}\circ g_{\CATA}$ belongs to $\SETWA$, then by hypothesis we get that $\functor{F}_1
(f_{\CATA}\circ g_{\CATA})$ is an internal equivalence of $\CATB$. Therefore we get easily
that also $\functor{F}_1(f_{\CATA})\circ\functor{F}_1(g_{\CATA})$ is an internal equivalence. Since
$\functor{F}_1(f_{\CATA})$ is an internal equivalence by hypothesis, then
by~\cite[Lemma~1.2]{T4} we conclude that $\functor{F}_1(g_{\CATA})$ is an internal equivalence. 
\end{proof}

\begin{prop}\label{prop-01}
Let us fix any pair $(\CATA,\SETWA)$ satisfying conditions \emphatic{(BF)},
any bicategory $\CATB$ and any pseudofunctor $\functor{F}:\CATA\rightarrow\CATB$, such that:

\begin{itemize}
 \item $\functor{F}$ sends each morphism in $\SETWA$ to an internal equivalence of $\CATB$;
 \item $\functor{F}$ satisfies conditions \emphatic{(\hyperref[EF1]{EF1})},
   \emphatic{(\hyperref[EF2]{EF2})} and \emphatic{(\hyperref[EF3]{EF3})}.
\end{itemize}

Then $\functor{F}$ satisfies conditions \emphatic{(\hyperref[B1]{B1})} --
\emphatic{(\hyperref[B5]{B5})}, hence by Corollary~\ref{cor-01} there is an equivalence
of bicategories $\CATA\left[\SETWAinv\right]\rightarrow\CATB$. In other terms,
\cite[Proposition~24]{Pr} lists sufficient \emphatic{(}\emph{but not necessary}\emphatic{)}
conditions for having an induced equivalence from $\CATA\left[\SETWAinv\right]$ to $\CATB$.
\end{prop}

\begin{proof}
Clearly (\hyperref[EF1]{EF1}) implies (\hyperref[B1]{B1}).\\

Let us prove (\hyperref[B2]{B2}), so let us fix any pair of objects $A^1_{\CATA},A^2_{\CATA}$
and any internal equivalence $e_{\CATB}:\functor{F}_0(A^1_{\CATA})\rightarrow\functor{F}_0
(A^2_{\CATA})$. Then by (\hyperref[EF2]{EF2}) applied to $e_{\CATB}$, there are an object
$A^3_{\CATA}$, a morphism $\operatorname{w}^1_{\CATA}:A^3_{\CATA}\rightarrow A^1_{\CATA}$ in
$\SETWA$, a morphism $\operatorname{w}^2_{\CATA}:A^3_{\CATA}\rightarrow A^2_{\CATA}$ and an
invertible $2$-morphism
$\alpha_{\CATB}:\functor{F}_1(\operatorname{w}^2_{\CATA})\Rightarrow e_{\CATB}\circ\functor{F}_1
(\operatorname{w}^1_{\CATA})$. Since $\functor{F}$ sends each morphism of $\SETWA$ to an
internal equivalence, then $\functor{F}_1(\operatorname{w}^1_{\CATA})$ is an internal equivalence.
Then by~\cite[Lemmas~1.1 and~1.2]{T4} we conclude that also $\functor{F}_1
(\operatorname{w}^2_{\CATA})$ is an internal equivalence in $\CATB$.
Then by Lemma~\ref{lem-13} applied to
$\operatorname{w}^2_{\CATA}$, there are an object $A^4_{\CATA}$ and a morphism
$\operatorname{v}_{\CATA}:A^4_{\CATA}\rightarrow A^3_{\CATA}$, such that $\operatorname{w}^2_{\CATA}
\circ\operatorname{v}_{\CATA}$ belongs to $\SETWA$ and $\functor{F}_1(\operatorname{v}_{\CATA})$ is
an internal equivalence. Again by Lemma~\ref{lem-13} applied to $\operatorname{v}_{\CATA}$,
there are an object $A^5_{\CATA}$ and a morphism
$\operatorname{z}_{\CATA}:A^5_{\CATA}\rightarrow A^4_{\CATA}$, such that $\operatorname{v}_{\CATA}
\circ\operatorname{z}_{\CATA}$ belongs to $\SETWA$. Using the definition of right saturated, this
proves that $\operatorname{w}^2_{\CATA}$ belongs to $\SETWAsat$.\\

Since $\functor{F}_1(\operatorname{w}^1_{\CATA})$ is an internal equivalence, then there are an
internal equivalence
$e'_{\CATB}:\functor{F}_0(A^1_{\CATA})\rightarrow\functor{F}_0(A^3_{\CATA})$ and an invertible
$2$-morphism $\delta^1_{\CATB}:\functor{F}_1(\operatorname{w}^1_{\CATA})\circ e'_{\CATB}\Rightarrow
\id_{\functor{F}_0(A^1_{\CATA})}$. Then we define an invertible $2$-morphism as follows:

\[\delta^2_{\CATB}:=\Big(\alpha_{\CATB}^{-1}\ast i_{e'_{\CATB}}\Big)\odot\thetaa{e_{\CATB}}
{\functor{F}_1(\operatorname{w}^1_{\CATA})}{e'_{\CATB}}\odot\Big(i_{e_{\CATB}}\ast
(\delta^1_{\CATB})^{-1}\Big)\odot\pi^{-1}_{e_{\CATB}}:\,\,
e_{\CATB}\Longrightarrow\functor{F}_1(\operatorname{w}^2_{\CATA})\circ e'_{\CATB}.\]

This proves that (\hyperref[B2]{B2}) holds.\\

Now let us prove that condition (\hyperref[B3]{B3}) is satisfied, so let us fix any pair of objects
$B_{\CATA},A_{\CATB}$ and any morphism $f_{\CATB}:A_{\CATB}\rightarrow\functor{F}_0(B_{\CATA})$. By
(\hyperref[EF1]{EF1}) there are an object $A_{\CATA}$ and an isomorphism $\operatorname{t}_{\CATB}:
\functor{F}_0
(A_{\CATA})\rightarrow A_{\CATB}$. Let us apply (\hyperref[EF2]{EF2}) to the morphism $f_{\CATB}
\circ t_{\CATB}:\functor{F}_0(A_{\CATA})\rightarrow\functor{F}_0(B_{\CATA})$. Then there are an object
$A'_{\CATA}$, a morphism $f_{\CATA}:A'_{\CATA}\rightarrow B_{\CATA}$, a morphism
$\operatorname{w}_{\CATA}:A'_{\CATA}\rightarrow A_{\CATA}$ in $\SETWA$ and an invertible $2$-morphism
$\delta_{\CATB}:\functor{F}_1(f_{\CATA})\Rightarrow(f_{\CATB}\circ t_{\CATB})\circ\functor{F}_1
(\operatorname{w}_{\CATA})$. Since $\functor{F}$ sends each morphism of $\SETWA$ to an internal
equivalence, then there are an internal equivalence $\operatorname{w}_{\CATB}:\functor{F}_0(A_{\CATA})
\rightarrow\functor{F}_0(A'_{\CATA})$ and an invertible $2$-morphism $\xi_{\CATB}:
\id_{\functor{F}_0(A_{\CATA})}\Rightarrow\functor{F}_1(\operatorname{w}_{\CATA})\circ
\operatorname{w}_{\CATB}$. Then we define an internal equivalence

\[e_{\CATB}:=\operatorname{w}_{\CATB}\circ\,\operatorname{t}^{-1}_{\CATB}:\,A_{\CATB}\longrightarrow
\functor{F}_0(A'_{\CATA})\]
and an invertible $2$-morphism as follows

\begin{gather*}
\alpha_{\CATB}:=\thetab{\functor{F}_1(f_{\CATA})}{\operatorname{w}_{\CATB}}
 {\operatorname{t}^{-1}_{\CATB}}\odot
 \Big(\Big(\delta_{\CATB}^{-1}\ast i_{\operatorname{w}_{\CATB}}\Big)\ast
 i_{\operatorname{t}^{-1}_{\CATB}}\Big)\odot
 \Big(\thetaa{f_{\CATB}\circ t_{\CATB}}{\functor{F}_1(\operatorname{w}_{\CATA})}
 {\operatorname{w}_{\CATB}}\ast i_{\operatorname{t}^{-1}_{\CATB}}\Big)\odot \\
\odot\Big(\Big(i_{f_{\CATB}\circ t_{\CATB}}\ast\xi_{\CATB}\Big)\ast i_{\operatorname{t}^{-1}_{\CATB}}
 \Big)\odot\Big(\pi^{-1}_{f_{\CATB}\circ t_{\CATB}}\ast i_{\operatorname{t}^{-1}_{\CATB}}\Big)\odot
 \thetaa{f_{\CATB}}{t_{\CATB}}{\operatorname{t}^{-1}_{\CATB}}\odot\pi^{-1}_{f_{\CATB}}: \\
f_{\CATB}\Longrightarrow\functor{F}_1(f_{\CATA})\circ e_{\CATB}.
\end{gather*}

This proves that (\hyperref[B3]{B3}) holds. The proof that (\hyperref[B4]{B4}) and
(\hyperref[B5]{B5}) hold is a direct consequence of (\hyperref[EF3]{EF3}).
\end{proof}

\begin{rem}
The already mentioned~\cite[Proposition~24]{Pr} is applied in~\cite{Pr} only for what concerns the
\emph{sufficiency} of conditions (\hyperref[EF1]{EF1}) -- (\hyperref[EF3]{EF3}). Hence even if such
conditions are only sufficient but not necessary, Proposition~\ref{prop-01} shows that the error
in~\cite[Proposition~24]{Pr} does not affect the rest of the computations in~\cite{Pr}.
\end{rem}

\begin{rem}
Let us consider the special case when $\CATB:=\CATA\left[\SETWAinv\right]$ and $\functor{F}:=
\functor{U}_{\SETWA}$.
In this case, one can consider the pair $(\overline{\functor{G}},\overline{\kappa})$ given by
$\overline{\functor{G}}:=\id_{\CATA\left[\SETWAinv\right]}$ and $\overline{\kappa}:=
i_{\functor{U}_{\SETWA}}:
\functor{U}_{\SETWA}\Rightarrow\overline{\functor{G}}\circ\functor{U}_{\SETWA}$. This pair
satisfies Theorem~\ref{theo-03}(2) and $\overline{\functor{G}}$ is an equivalence of bicategories, so
Theorem~\ref{theo-04} tells us that $\functor{U}_{\SETWA}$ satisfies conditions (\hyperref[B]{B}).
As a check for the correctness of Theorem~\ref{theo-04}, we want to verify by hand that result.\\

In this case, clearly (\hyperref[B1]{B1}) holds since
$\functor{U}_{\SETWA}$ is a bijection on objects. Now let us prove (\hyperref[B2]{B2}), so let us
fix any pair of objects $A^1_{\CATA},A^2_{\CATA}$ and any internal equivalence $e_{\CATB}$ from
$A^1_{\CATA}$ to $A^2_{\CATA}$
in $\CATA\left[\SETWAinv\right]$. By~\cite[Corollary~2.7]{T4} for $(\CATA,\SETWA)$,
$e_{\CATB}$ is necessarily of the form

\[
\begin{tikzpicture}[xscale=2.4,yscale=-1.2]
    \node (A0_0) at (-0.2, 0) {$e_{\CATB}=\Big(A^1_{\CATA}$};
    \node (A0_1) at (1, 0) {$A^3_{\CATA}$};
    \node (A0_2) at (2, 0) {$A^2_{\CATA}\Big)$,};
    
    \path (A0_1) edge [->]node [auto,swap] {$\scriptstyle{\operatorname{w}^1_{\CATA}}$} (A0_0);
    \path (A0_1) edge [->]node [auto] {$\scriptstyle{\operatorname{w}^2_{\CATA}}$} (A0_2);
\end{tikzpicture}
\]
with $\operatorname{w}^1_{\CATA}$ in $\SETWA$ and $\operatorname{w}^2_{\CATA}$ in $\SETWAsat$. Then
we define an internal equivalence $e'_{\CATB}:A^1_{\CATA}\rightarrow A^3_{\CATA}$ in $\CATB:=\CATA
\left[\SETWAinv\right]$ as follows:

\[
\begin{tikzpicture}[xscale=2.4,yscale=-1.2]
    \node (A0_0) at (-0.2, 0) {$e'_{\CATB}:=\Big(A^1_{\CATA}$};
    \node (A0_1) at (1, 0) {$A^3_{\CATA}$};
    \node (A0_2) at (2, 0) {$A^3_{\CATA}\Big)$.};
    
    \path (A0_1) edge [->]node [auto,swap] {$\scriptstyle{\operatorname{w}^1_{\CATA}}$} (A0_0);
    \path (A0_1) edge [->]node [auto] {$\scriptstyle{\id_{A^3_{\CATA}}}$} (A0_2);
\end{tikzpicture}
\]

Moreover, we define a pair of invertible $2$-morphisms in $\CATB=\CATA\left[\SETWAinv\right]$:

\begin{gather*}
\delta^1_{\CATB}:=\Big[A^3_{\CATA},\id_{A^3_{\CATA}},\operatorname{w}^1_{\CATA},
 \upsilon_{\operatorname{w}^1_{\CATA}}^{-1}\odot\pi_{\operatorname{w}^1_{\CATA}}\odot
 \pi_{\operatorname{w}^1_{\CATA}\circ\id_{A^3_{\CATA}}},
 \upsilon_{\operatorname{w}^1_{\CATA}}^{-1}\odot\pi_{\operatorname{w}^1_{\CATA}}\odot
 \pi_{\operatorname{w}^1_{\CATA}\circ\id_{A^3_{\CATA}}}\Big]: \\
\functor{U}_{\SETWA,1}(\operatorname{w}^1_{\CATA})\circ e'_{\CATB}=
 \Big(A^3_{\CATA},\operatorname{w}^1_{\CATA}\circ\id_{A^3_{\CATA}},\operatorname{w}^1_{\CATA}
 \circ\id_{A^3_{\CATA}}\Big)
 \Longrightarrow\Big(A^1_{\CATA},\id_{A^1_{\CATA}},\id_{A^1_{\CATA}}
 \Big)=\id_{A^1_{\CATA}}
\end{gather*}
and

\begin{gather*}
\delta^2_{\CATB}:=\Big[A^3_{\CATA},\id_{A^3_\CATA},\id_{A^3_{\CATA}},
 \pi^{-1}_{\operatorname{w}^1_{\CATA}\circ\id_{A^3_{\CATA}}},
 \pi^{-1}_{\operatorname{w}^2_{\CATA}\circ\id_{A^3_{\CATA}}}\Big]: \\
\,e_{\CATB}\Longrightarrow\Big(A^3_{\CATA},
 \operatorname{w}^1_{\CATA}\circ\id_{A^3_{\CATA}},\operatorname{w}^2_{\CATA}\circ\id_{A^3_{\CATA}}
 \Big)=\functor{U}_{\SETW,1}(\operatorname{w}^2_{\CATA})\circ e'_{\CATB}.
\end{gather*}

This suffices to prove that
(\hyperref[B2]{B2}) holds for $\functor{U}_{\SETWA}$.
Condition (\hyperref[B3]{B3}) is a direct condition of the definition of
morphisms in a bicategory of fractions.\\

In order to prove (\hyperref[B4]{B4}), it suffices to use~\cite[Proposition~0.7]{T3}
in the special case when
$(\CATC,\SETW):=(\CATA,\SETWA)$, $A^1=A^2=:A_{\CATA}$, $\operatorname{w}^1=\operatorname{w}^2:=
\id_{A_{\CATA}}$, $\alpha:=i_{\id_{A_{\CATA}}\circ\id_{A_{\CATA}}}$ and $\gamma^m:=
\gamma^m_{\CATA}\ast i_{\id_{A_{\CATA}}}$, so that $\Gamma^m=
\functor{U}_{\SETWA,2}(\gamma^m_{\CATA})$ for $m=1,2$.\\

Lastly, let us prove (\hyperref[B5]{B5}), so let us fix any pair of objects $A_{\CATA},B_{\CATA}$,
any pair of morphisms $f^1_{\CATA},f^2_{\CATA}:A_{\CATA}\rightarrow B_{\CATA}$ and any $2$-morphism
$\Gamma_{\CATA}$ from $(A_{\CATA},\id_{A_{\CATA}},f^1_{\CATA})=\functor{U}_{\SETWA,1}(f^1_{\CATA})$
to $(A_{\CATA},\id_{A_{\CATA}},
f^2_{\CATA})=\functor{U}_{\SETWA,1}(f^2_{\CATA})$ in $\CATA\left[\SETWAinv\right]$.
By~\cite[Lemma~6.1]{T4} applied to $\alpha:=i_{\id_{A_{\CATA}}\circ\id_{A_{\CATA}}}$, there are an
object $A'_{\CATA}$ and a morphism $\operatorname{v}_{\CATA}:A'_{\CATA}\rightarrow A_{\CATA}$ in
$\SETWA$ and a $2$-morphism $\gamma_{\CATA}:f^1_{\CATA}\circ\operatorname{v}_{\CATA}
\Rightarrow f^2_{\CATA}\circ\operatorname{v}_{\CATA}$, such that

\[\Gamma_{\CATA}=\Big[A'_{\CATA},\operatorname{v}_{\CATA},\operatorname{v}_{\CATA},
i_{\id_{A_{\CATA}}\circ\operatorname{v}_{\CATA}},\gamma_{\CATA}\Big].\]

Hence, using~\cite[pagg.~259--261]{Pr} and the description of $\functor{U}_{\SETWA}$ we have that

\[\Gamma_{\CATA}\ast i_{\functor{U}_{\SETWA,1}(\operatorname{v}_{\CATA})}=\Gamma_{\CATA}\ast
i_{\left(A'_{\CATA},\id_{A'_{\CATA}},\operatorname{v}_{\CATA}\right)}\]
is represented by the following diagram

\[
\begin{tikzpicture}[xscale=2.6,yscale=-0.8]
    \node (A0_2) at (2, 0) {$A'_{\CATA}$};
    \node (A2_0) at (0, 2) {$A'_{\CATA}$};
    \node (A2_2) at (2, 2) {$A'_{\CATA}$};
    \node (A2_4) at (4, 2) {$B_{\CATA}$.};
    \node (A4_2) at (2, 4) {$A'_{\CATA}$};

    \node (A2_1) at (1.15, 2) {$\Downarrow\,i_{(\id_{A'_{\CATA}}
      \circ\id_{A'_{\CATA}})\circ\id_{A'_{\CATA}}}$};
    \node (A2_3) at (2.85, 2) {$\Downarrow\,\gamma_{\CATA}\ast i_{\id_{A'_{\CATA}}}$};
    
    \path (A0_2) edge [->]node [auto,swap] {$\scriptstyle{\id_{A'_{\CATA}}
       \circ\id_{A'_{\CATA}}}$} (A2_0);
    \path (A0_2) edge [->]node [auto] {$\scriptstyle{f^1_{\CATA}
      \circ\operatorname{v}_{\CATA}}$} (A2_4);
    \path (A4_2) edge [->]node [auto,swap] {$\scriptstyle{f^2_{\CATA}
      \circ\operatorname{v}_{\CATA}}$} (A2_4);
    \path (A2_2) edge [->]node [auto,swap] {$\scriptstyle{\id_{A'_{\CATA}}}$} (A0_2);
    \path (A2_2) edge [->]node [auto] {$\scriptstyle{\id_{A'_{\CATA}}}$} (A4_2);
    \path (A4_2) edge [->]node [auto] {$\scriptstyle{\id_{A'_{\CATA}}\circ\id_{A'_{\CATA}}}$} (A2_0);
\end{tikzpicture}
\]

Moreover, for each $m=1,2$, a direct check proves that the associator
$\psi_{f^m_{\CATA},\operatorname{v}_{\CATA}}^{\functor{U}_{\SETWA}}$ for $\functor{U}_{\SETWA}$
is represented by the following diagram:

\[
\begin{tikzpicture}[xscale=2.6,yscale=-0.8]
    \node (A0_2) at (2, 0) {$A'_{\CATA}$};
    \node (A2_0) at (0, 2) {$A'_{\CATA}$};
    \node (A2_2) at (2, 2) {$A'_{\CATA}$};
    \node (A2_4) at (4, 2) {$B_{\CATA}$.};
    \node (A4_2) at (2, 4) {$A'_{\CATA}$};

    \node (A2_3) at (2.7, 2) {$\Downarrow\,i_{(f^m_{\CATA}\circ
      \operatorname{v}_{\CATA})\circ\id_{A'_{\CATA}}}$};
    \node (A2_1) at (1.2, 2) {$\Downarrow\,\pi^{-1}_{\id_{A'_{\CATA}}\circ\id_{A'_{\CATA}}}$};
    
    \path (A0_2) edge [->]node [auto,swap] {$\scriptstyle{\id_{A'_{\CATA}}}$} (A2_0);
    \path (A4_2) edge [->]node [auto] {$\scriptstyle{\id_{A'_{\CATA}}\circ\id_{A'_{\CATA}}}$} (A2_0);
    \path (A4_2) edge [->]node [auto,swap] {$\scriptstyle{f^m_{\CATA}
      \circ\operatorname{v}_{\CATA}}$} (A2_4);
    \path (A2_2) edge [->]node [auto,swap] {$\scriptstyle{\id_{A'_{\CATA}}}$} (A0_2);
    \path (A2_2) edge [->]node [auto] {$\scriptstyle{\id_{A'_{\CATA}}}$} (A4_2);
    \path (A0_2) edge [->]node [auto] {$\scriptstyle{f^m_{\CATA}
      \circ\operatorname{v}_{\CATA}}$} (A2_4);
\end{tikzpicture}
\]

Therefore, it is easy to see that the composition

\begin{equation}\label{eq-06}
\psi_{f^2_{\CATA},\operatorname{v}_{\CATA}}^{\functor{U}_{\SETWA}}\odot
\Gamma_{\CATA}\ast i_{\functor{U}_{\SETWA,1}(\operatorname{v}_{\CATA})}\odot
\Big(\psi_{f^1_{\CATA},\operatorname{v}_{\CATA}}^{\functor{U}_{\SETWA}}\Big)^{-1}
\end{equation}

is represented by the diagram

\[
\begin{tikzpicture}[xscale=2.6,yscale=-0.8]
    \node (A0_2) at (2, 0) {$A'_{\CATA}$};
    \node (A2_0) at (0, 2) {$A'_{\CATA}$};
    \node (A2_2) at (2, 2) {$A'_{\CATA}$};
    \node (A2_4) at (4, 2) {$B_{\CATA}$,};
    \node (A4_2) at (2, 4) {$A'_{\CATA}$};

    \node (A2_1) at (1.15, 2) {$\Downarrow\,i_{\id_{A'_{\CATA}}
      \circ\id_{A'_{\CATA}}}$};
    \node (A2_3) at (2.85, 2) {$\Downarrow\,\gamma_{\CATA}\ast i_{\id_{A'_{\CATA}}}$};
    
    \path (A0_2) edge [->]node [auto,swap] {$\scriptstyle{\id_{A'_{\CATA}}}$} (A2_0);
    \path (A0_2) edge [->]node [auto] {$\scriptstyle{f^1_{\CATA}
      \circ\operatorname{v}_{\CATA}}$} (A2_4);
    \path (A4_2) edge [->]node [auto,swap] {$\scriptstyle{f^2_{\CATA}
      \circ\operatorname{v}_{\CATA}}$} (A2_4);
    \path (A2_2) edge [->]node [auto,swap] {$\scriptstyle{\id_{A'_{\CATA}}}$} (A0_2);
    \path (A2_2) edge [->]node [auto] {$\scriptstyle{\id_{A'_{\CATA}}}$} (A4_2);
    \path (A4_2) edge [->]node [auto] {$\scriptstyle{\id_{A'_{\CATA}}}$} (A2_0);
\end{tikzpicture}
\]
i.e.\ \eqref{eq-06} coincides with $\functor{U}_{\SETWA,2}(\gamma_{\CATA})$ (see~\cite[\S~2.4]{Pr}),
so (\hyperref[B5]{B5}) holds for $\functor{U}_{\SETWA}$.
\end{rem}


\end{document}